% ...

\documentclass[preprint,3p,times]{elsarticle}

\usepackage[english]{babel}
\usepackage[utf8]{inputenc}
\usepackage{pdflscape} 
\usepackage{amssymb,amsmath}
\usepackage{amsthm}
\usepackage{mathtools}
\usepackage[normalem]{ulem}
\usepackage{enumitem}
\usepackage[dvipsnames,svgnames]{xcolor}
\usepackage{array}
\usepackage{multirow}

\usepackage{afterpage} 

\usepackage[nodots]{numcompress} 

\usepackage{graphicx}
\usepackage[FIGBOTCAP,TABTOPCAP]{subfigure} 
\usepackage{tikz}
\usetikzlibrary{calc,shapes,arrows,backgrounds,patterns}
\tikzstyle{every picture}+=[remember picture]
\usepackage{pgfplots} 

\usepackage{epstopdf} 
\usepackage{hyperref}
\hypersetup{colorlinks=true,linkcolor=blue,citecolor=blue}
\newcommand{\myhy}[2]{\hyperref[#1]{\color{blue}\setulcolor{blue}\ul{#2}}}

\usepackage[all]{hypcap}

\mathtoolsset{showonlyrefs} 
\allowdisplaybreaks 
\biboptions{sort&compress,numbers}

\usepackage{xspace}
\makeatletter
\DeclareRobustCommand\onedot{\futurelet\@let@token\@onedot}
\newcommand{\@onedot}{\ifx\@let@token.\else.\null\fi\xspace}
\newcommand{\eg}{{e.g}\onedot} 
\newcommand{\ie}{{i.e}\onedot}

\makeatother

\newtheorem{prop}{Proposition}
\newtheorem{rem}{Remark}

\newtheorem{lem}{Lemma}

\theoremstyle{definition}

\newcommand{\RR}{\mathbb{R}}

\newcommand{\vectbi}[1]{\boldsymbol{#1}} 
\newcommand{\matri}[1]{#1}               
\newcommand{\vect}[1]{\vectbi{#1}}
\newcommand{\matr}[1]{\matri{#1}}

\newcommand{\norm}[1]{\left\lVert#1\right\rVert}

\renewcommand{\leq}{\leqslant}
\renewcommand{\geq}{\geqslant}

\newcommand{\dr}[1]{\frac{\partial^r}{\partial {#1}^r}}

\newcommand{\dq}[1]{\frac{\partial^q}{\partial {#1}^q}}
\newcommand{\drq}[1]{\frac{\partial^{r-q}}{\partial {#1}^{r-q}}}

\newcommand{\bF}{\vect{F}}
\newcommand{\bH}{\vect{H}}
\newcommand{\bP}{\vect{P}}
\newcommand{\bS}{\vect{S}}
\newcommand{\bT}{\vect{T}}

\newcommand{\bd}{\vect{d}}
\newcommand{\be}{\vect{e}}
\newcommand{\bff}{\vect{f}}

\newcommand{\bn}{\vect{n}}
\newcommand{\bp}{\vect{p}}
\newcommand{\bq}{\vect{q}}
\newcommand{\br}{\vect{r}}

\newcommand{\bw}{\vect{w}}
\newcommand{\bgamma}{\vect{\gamma}}

\newcommand{\bTau}{\vect{T}}

\newcommand{\bxi}{\vect{\xi}}

\newcommand{\btau}{\vect{\tau}}
\newcommand{\bchi}{\vect{\chi}}

\newcommand{\blambda}{\vect{\lambda}}

\newcommand{\qtext}[1]{``#1''}

\definecolor{colspline}{rgb}{0.32,0.66,1.0}
\definecolor{colcd}{rgb}{0.86,0.08,0.23}
\definecolor{colcd2}{rgb}{0.19,0.8,0.19}

\newlength{\casesvsep}
\newcommand{\figpath}{figure}
\graphicspath{{\figpath/}}

\makeatletter
\newcommand*{\shifttext}[2]{%
\settowidth{\@tempdima}{#2}%
\makebox[\@tempdima]{\hspace*{#1}#2}%
}
\makeatother

\newcommand{\apptext}{}

\makeatletter
\providecommand{\doi}[1]{%
  \begingroup
    \let\bibinfo\@secondoftwo
    \urlstyle{rm}%
    \href{http://dx.doi.org/#1}{%
      doi:\discretionary{}{}{}%
      \nolinkurl{#1}%
    }%
  \endgroup
}
\makeatother

\newcommand{\class}{$D^gC^kP^mS^w$}

\makeatletter
\def\ps@pprintTitle{%
  \let\@oddhead\@empty
  \let\@evenhead\@empty
  \def\@oddfoot{\reset@font\hfil\thepage\hfil}
  \let\@evenfoot\@oddfoot
}
\makeatother

\begin{document}

\newcommand{\abstracttext}{
In CAGD the design of a surface that interpolates an arbitrary quadrilateral mesh is definitely a challenging task.
The basic requirement is to satisfy both criteria concerning the regularity of the surface and aesthetic concepts.

With regard to the aesthetic quality, it is well known that interpolatory methods often produce shape artifacts when the data points are unevenly spaced. In the univariate setting, this problem can be overcome, or at least mitigated, by exploiting a proper non-uniform parametrization, that accounts for the geometry of the data. Moreover, recently, the same principle has been generalized and proven to be effective in the context of bivariate interpolatory subdivision schemes.

In this paper, we propose a construction for parametric surfaces of good aesthetic quality and high smoothness that interpolate quadrilateral meshes of arbitrary topology.

In the classical tensor product setting the same parameter interval must be shared by an entire row or column of mesh edges. Conversely, in this paper, we assign a different parameter interval to each edge of the mesh.
This particular structure, which we call an \emph{augmented parametrization}, allows us to interpolate each section polyline of the mesh at parameters values that prevent wiggling of the resulting curve or other interpolation artifacts. This yields high quality interpolatory surfaces.

The proposed surfaces are a generalization of the local univariate spline interpolants introduced in Beccari et al. (2013) and  Antonelli et al. (2014), that can have arbitrary continuity and arbitrary order of polynomial reproduction.
In particular, these surfaces retain the same smoothness of the underlying class of univariate splines in the regular regions of the mesh (where, locally, all vertices have valence 4). Moreover, in mesh regions containing vertices of valence other than 4, we suitably define $G^1$- or $G^2$-continuous surface patches that join the neighboring regular ones.
}

\newcommand{\titletext}{
High quality local interpolation by composite parametric surfaces
}
\newcommand{\titlerunningtext}{\dots}

\begin{frontmatter}

\title{\titletext}

\author[label1]{Michele Antonelli}
\ead{antonelm@math.unipd.it}
\author[label2]{Carolina Vittoria Beccari}
\ead{carolina.beccari2@unibo.it}
\author[label2]{Giulio Casciola}
\ead{giulio.casciola@unibo.it}


\address[label1]{Department of Mathematics, University of Padova,
Via Trieste 63, 35121 Padova, Italy}
\address[label2]{Department of Mathematics, University of Bologna,
Piazza di Porta San Donato 5, 40126 Bologna, Italy}

\begin{abstract}
\abstracttext
\end{abstract}


\end{frontmatter}


\section{Introduction}
\label{sec:intro}
A central topic in computer-aided geometric design is the construction of parametric curves and surfaces that interpolate a given set of points. In 2D, these points are the vertices of a control polygon, whereas in 3D they are the vertices of a control polyhedron (also called a mesh).
In this context, a \qtext{good} interpolant is one that faithfully mimics the shape suggested by the input data: this means, \eg,  that it does not present self intersections, nor it bends to much or it is too tight with respect to the given control polygon or polyhedron.

Concerning univariate interpolation, it is well known that a proper choice of parametrization is crucial to obtain good quality curves that interpolate unevenly spaced data (basic references are \cite{Farin2002a,Farin2002b}).
In particular, a uniform parametrization may give rise to noticeable interpolation artifacts, since it does not account for the geometry of the data. These artifacts often disappear, or are greatly mitigated, when a suitable non-uniform parametrization is used (Figure \ref{fig:cfr_param} is an example).

Although there is probably no \qtext{best} parametrization, since any method can be defeated by a suitably chosen data set, some techniques produce good results in the majority of critical cases.
One of the most effective \cite{Lee1989,Floater2008,Yuksel2011}, and probably the most used, is the centripetal parametrization. Alongside this, it is worth mentioning the Nielson-Foley parametization \cite{Foley-Nielson1989} and the recent method proposed in \cite{Fang2013}, which, in some cases, shows better performance than either of the previous.
\begin{figure}[t]
\centering
\subfigure[Uniform parametrization]{\includegraphics[width=0.3\textwidth]{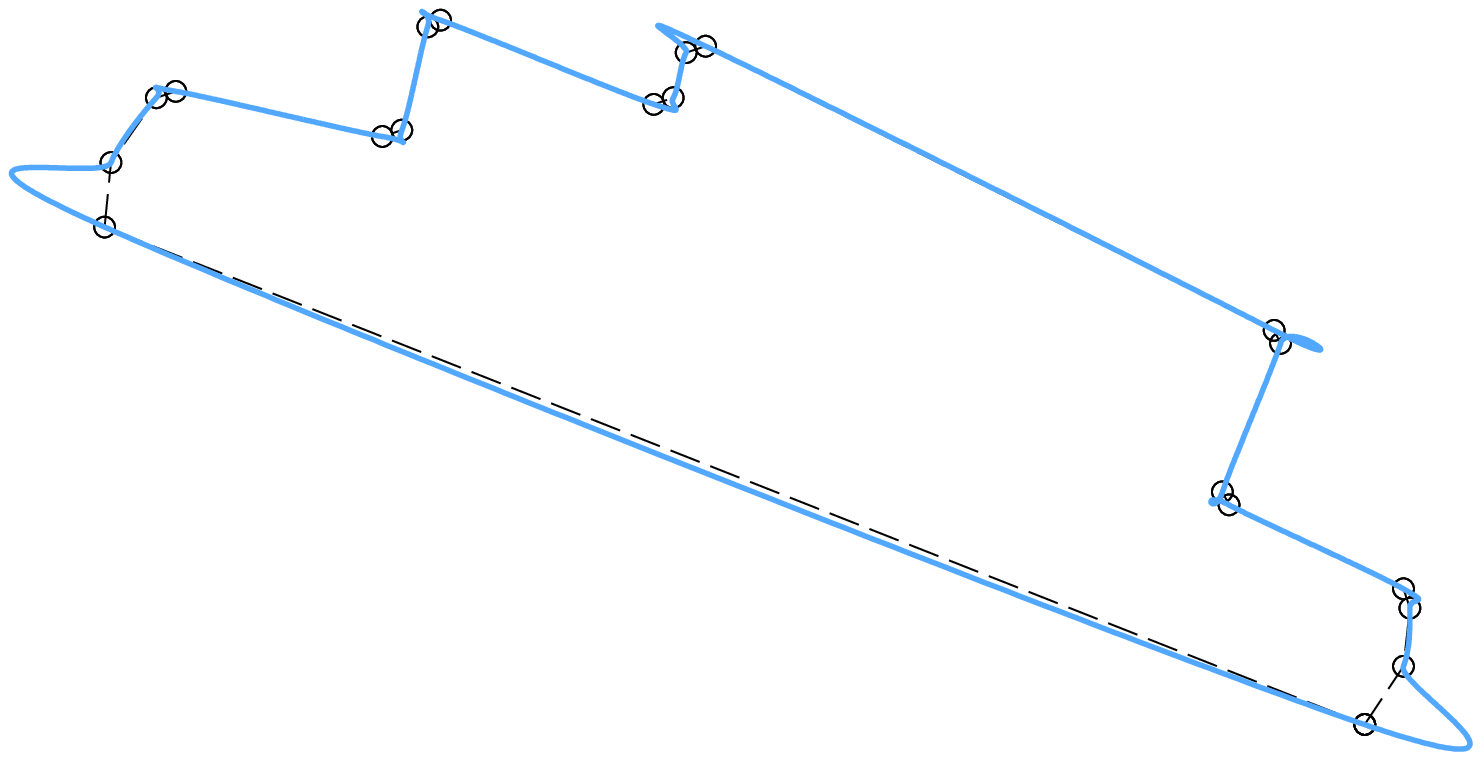}}
\subfigure[Centripetal parametrization]{\includegraphics[width=0.3\textwidth]{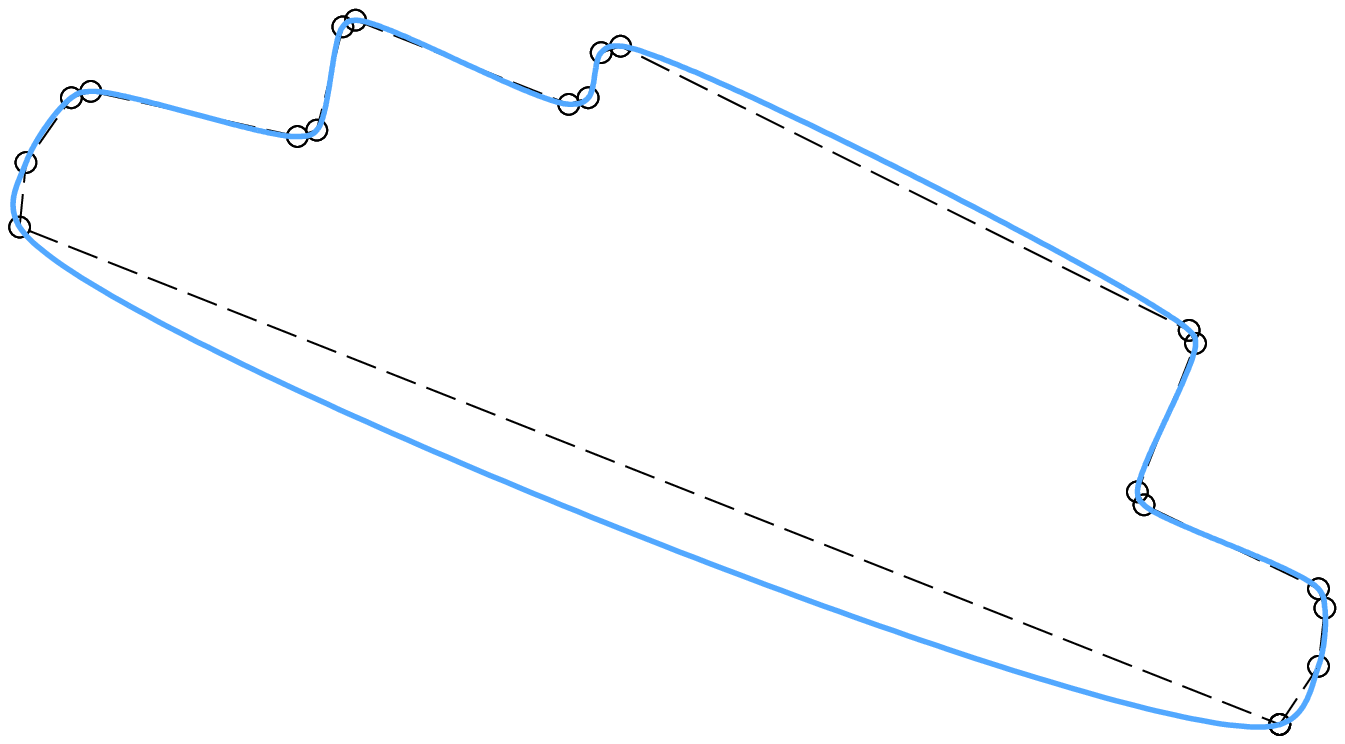}}
\caption{Cubic (global) spline interpolation of unevenly spaced data. The uniform parametrization fails whereas a proper non-uniform parametrization yields a good interpolant.}
\label{fig:cfr_param}
\end{figure}

The possibility to manage a non-uniform parametrization is even more crucial when using local interpolation methods.
The latter are more efficient than global approaches, because they do not
involve solving systems of linear equations, which can readily be of large dimension when complex surfaces have to be generated. Moreover, local modification of the data can be handled by local updating of the interpolation model and does not require recomputing the entire model from scratch.
However, local interpolation techniques are more prone to generating interpolation artifacts than global ones and therefore they suffer even more from a poor choice of parametrization.

The construction of local spline interpolants with specific properties and non-uniform parametrization has long been an open question.
This is probably one of the reasons why, so far, local interpolation methods have had low uptake in computer design applications.
In the univariate setting, two recent works have developed a general approach for solving the problem \cite{BCR13a,ABC13}.
The technique proposed in those papers allows us to choose an arbitrary support width and, correspondingly, construct various classes of local spline interpolants that differ one from another in their degree, continuity and approximation order.
In this way, one can pick the interpolating spline whose properties best suit the context of application.
Moreover, the popular Catmull-Rom splines \cite{Catmull-Rom1974}, as well as other types of local interpolants previously appeared (a non-exhaustive list includes \cite{CdV96,BTU03,BS03,UTO07,Han2011}) are special instances of the construction.

In this paper we generalize the univariate local interpolatory splines in \cite{BCR13a,ABC13} to the bivariate setting. In particular, by introducing an appropriate new parametrization technique, we will obtain parametric surfaces of good aesthetic quality and high smoothness.

We shall start by considering the case of a regular mesh, which is one where each vertex has exactly 4 incident edges. The main hurdle is represented by the need of defining a non-uniform parametrization that preserves the quality of the local interpolants, when generalizing from the univariate to the bivariate setting.
The standard approach is to construct a tensor product surface from the univariate splines that one wishes to use.
Being $uv$ the parametric domain,
the tensor product requires the generation of one set of parameter values for
all isoparametric curves in the $u$-direction; the same holds for the $v$-direction.
For finding such a parametrization, the most reasonable way of proceeding is to create a good parametrization for each isoparametric curve using a suitable
method, such as, \eg the centripetal parametrization, and then average each of these parameterizations to yield one.
This approach will only produce acceptable results if all the isoparametric curves essentially
yield the same parametrization.
Conversely, when the data points significantly deviate from a regular grid, it will result in a poor choice in
parameters. The isoparametric curves will unnaturally wiggle, and this defect will
be reflected in the surface (see, \eg \cite[\S 7.5.1]{Farin2002b}).

The above discussion suggests that the classical idea of averaging the parametrization should be avoided.
Hence we will take another avenue. In particular we assign to each mesh edge a parameter interval, without requiring that
the same parameter interval is shared by an entire row or column of mesh edges.
We call such a parametrization an \emph{augmented parametrization}.

A similar parametrization concept has previously been exploited in the context of subdivision schemes.
It was initially proposed in connection with Catmull-Clark surfaces \cite{Sederberg1998,Cashman-Augsdorfer-Dodgson-Sabin2009,Muller2006,Muller2010}, to enrich the method with greater control over the shape of free-form objects and local shape effects such as crease edges.
In fact, the term \qtext{augmented} was firstly introduced in \cite{Muller2006} with a similar meaning to the one we use here.
More recently, an interpolatory subdivision scheme with augmented parametrization has been discussed in \cite{BCR13b}.
The scheme allows for interpolating each section polyline of the mesh at independent parameter values and therefore the most appropriate parametrization is used to construct each section curve of the limit surface.
This approach yields surfaces of far superior aesthetic quality compared to their uniform or tensor-product counterparts.
Unfortunately, it is subject to the typical issues of interpolatory subdivision schemes. In particular the resulting surfaces cannot be evaluated at arbitrary parameters and have smoothness limited to $C^1$.

The construction of the local interpolatory surfaces proposed in this paper proceeds as follows. We start by assuming that a good parametrization has been computed independently for each section polyline of the mesh using a suitable method. Accordingly, a different parameter interval is assigned to each mesh edge, which gives rise to an augmented parametrization. Next, we choose a class of local univariate spline interpolants based on the desired smoothness, order of polynomial reproduction or support width.
We then suitably construct a composite surface that contains all the interpolatory curves of this class generated from the individual section polylines of the mesh with the associated parameterizations.
In this way, each section polyline of the mesh is interpolated at the parameter values that allow the best quality of the resulting section curve. Moreover, the section curves are \qtext{feature curves} of the surface, which, in turn, is a good quality interpolant to the input mesh. We also prove that, if the underlying univariate splines are $C^k$ continuous, then the composite surface is $G^k$ (meaning that derivatives agree after suitable reparametrization \cite{PetersHandbook02}). 
As a consequence, the composite surface retains the good properties of the corresponding univariate splines.

Oppositely to global techniques, the considered approach does not require to solve large systems of linear equations and allows for local editing of the generated surfaces.
Moreover, other local methods typically require to supply additional input data, such as cross-boundary derivatives and twist vectors (an example is given by the well-known Coons patches \cite{Farin2002a}).
A major drawback is that these quantities need to be heuristically estimated and
this is likely to be one of the reasons why, to the best of our knowledge, only $G^1$ or $G^2$ Coons patches have been considered so far. Conversely, the surfaces proposed in this paper can have higher continuity.

An important consequence of the locality of the proposed construction is that the method can be extended to handle meshes containing extraordinary vertices (namely vertices with valence different than 4).
Any such mesh can be partitioned into regular and extraordinary regions.
Loosely speaking, regular regions are those where all vertices are regular and therefore the local interpolation scheme can be applied.
Patching the regular regions will generate a surface with \qtext{holes} corresponding to the extraordinary regions.
Across the boundary of such holes, the tangent field (and higher order derivative fields) are determined by the surrounding regular patches and, as a consequence of the augmented parametrization, they change at every boundary point in a peculiar way.
The missing part of the surface shall be defined in such a way to interpolate the existing derivative fields, up to a reasonable order of continuity.

The general problem of how to fill the hole around an extraordinary vertex is an active research topic even in the case of the uniform parametrization. Recent developments include \cite{VaradyRockwoodSalvi2011,SalviVaradyRockwood2014,SalviVarady2014a},
where transfinite multi-sided patches are generated by interpolating derivative ribbons.
Our main point here is not to provide a one-stop solution, but to demonstrate how to obtain good quality local interpolating surfaces by exploiting, in synergy, the augmented parametrization, the regular patches described above and a local patching scheme for extraordinary vertices.
We will address this issue by mans of either $G^1$ Coons-Gregory patches \cite{Gregory1974,Hoschek93,Farin2002a} and their $G^2$ version \cite{MW91,Hermann1996}.
In both cases, we will see how to suitably tweak the definition of the patches, in order to interpolate the desired boundary information. This will allow us to generate augmented surfaces that interpolate a mesh containing extraordinary vertices, have arbitrary smoothness in the regular regions, and are $G^1$ or $G^2$ continuous in the extraordinary regions.\\

We would like to remark that the research reported in this paper was carried out under the European Eurostars project NIIT4CAD, aimed at the development of new technologies, mainly based on subdivision methods, for modeling arbitrary topology surfaces within a CAD system.
One of the advanced objectives of the project was to construct local interpolatory surfaces suitable for integration in a CAD system and this paper is a part of such a study.\\

The remainder of the paper is organized as follows. In Section \ref{sec:fundamental} we recall the univariate local spline interpolants presented in \cite{BCR13a,ABC13}.
Based on these, in Section \ref{sec:regular_case} we introduce our local method to generate interpolatory surfaces of good quality and high order of continuity, given as input a regular mesh. We also discuss all the relevant properties of the constructed surfaces.
Section \ref{sec:examples_regular} presents some application examples to demonstrate the effectiveness of the method and its superior performance with respect to the classical tensor product approach.
In Section \ref{sec:augmented_extraordinary} we address the problem of patching the extraordinary regions of the mesh by $G^1$ and $G^2$ augmented Coons-Gregory patches.
The latter patches are constructed to interpolate boundary curves and cross-boundary derivative fields. When this information is not available, it needs to be suitably generated and Section \ref{sec:network} illustrates a possible approach to handle this issue.
Section \ref{sec:examples_extraordinary} presents examples of augmented surfaces containing extraordinary vertices and finally
Section \ref{sec:conclusion} summarizes the results achieved and suggests some topics for future research.
\section{Local, non-uniform, univariate spline interpolation}
\label{sec:fundamental}
Suppose we are given a sequence of points $\bp_0, \bp_1,\dots,\bp_N$ in $\RR^n$, $n\geq2$, where $\bp_j$ and $\bp_{j+1}$ are distinct and $\bp_N=\bp_{0}$ and an associated sequence of parameter values $x_0<x_1<\dots<x_N$.
We consider the periodic spline curve $\bF:[x_0,x_N]\rightarrow \RR^n$ defined as
\begin{equation}\label{eq:spline}
\bF(x)=\sum_i \bp_{i} \, \psi_i(x),
\end{equation}
where $\psi_i:[x_0,x_N]\rightarrow \RR$ are \emph{fundamental spline functions}, namely they are piecewise functions on the partition $\{x_j\}$ such that $\psi_i(x_j)=\delta_{i,j}$.
Since we are interested in local interpolation, we require that
every $\psi_i$ has compact support, namely that it is nonzero in a finite number of parametric intervals.
In this way, when the data points are in $\RR^n$, $n=2,3$, formula \eqref{eq:spline} represents a parametric curve that locally interpolates the given data. It is evident that this approach allows us to overcome the need for solving linear systems of equations that arise when global spline interpolation is used.

In the remainder of the paper, we suppose that the fundamental functions $\psi_i$ have even support width $w$, namely that
\begin{equation}\label{eq:support}
\psi_i(x)=0,\qquad x\notin [x_{i-\frac{w}{2}},x_{i+\frac{w}{2}}].
\end{equation}
This assumption guarantees that an interpolating spline \eqref{eq:spline} preserves possible symmetries in the data and allows us to consistently simplify the notation that will be introduced later on.

If we limit ourselves to considering polynomial splines, the salient properties that characterize the fundamental functions are degree, continuity, maximum degree of polynomials that can be reproduced (this is equal to the approximation order minus one) and support width. In this respect, it is easily seen that the parametric interpolant $\bF$ in \eqref{eq:spline} has the same properties of the fundamental functions $\psi_i$.

A general method for constructing local interpolatory splines with specific properties and non-uniform parametrization  has been developed in the two recent papers \cite{BCR13a,ABC13}.
In that framework, we can choose an arbitrary support width and, correspondingly, construct various classes of local spline interpolants of the form \eqref{eq:spline} that differ one from another in their degree, continuity and approximation order.
Special instances of such splines are Catmull-Rom splines \cite{Catmull-Rom1974}, as well as other types of local interpolants including those in \cite{CdV96,BTU03,BS03,UTO07,Han2011}.

Adopting the notation in \cite{BCR13a,ABC13}, we indicate a class of splines having Degree $g$, Continuity order $k$, Polynomial reproduction degree $m$ and Support width $w$ by the shorthand notation $D^gC^kP^mS^w$.
Tables 1, 2, and 3 in \cite{BCR13a} and Table 2 in \cite{ABC13} summarize the different classes that we can construct for the most usual choices of support width, namely 4, 6 and 8.
Some examples with $w=4$ include the families $D^3C^0P^3S^4$, $D^3C^1P^2S^4$, $D^4C^2P^1S^4$, $D^5C^2P^2S^4$ from \cite{BCR13a} and $D^2C^1P^2S^4$ or $D^3C^2P^2S^4$ from \cite{ABC13}. Even more classes of splines can be found widening the support width to $6$ or $8$.

\begin{rem}
The local spline interpolants in \cite{BCR13a} and \cite{ABC13} differ in the following way.
In \cite{BCR13a} a fundamental function is formed by one polynomial segment in each interval $[x_i,x_{i+1}]$.
The work  \cite{ABC13} generalizes the construction in such a way that, in a parametric interval, a fundamental function can be composed of more than one polynomial segment, with proper continuity at the join between one segment and another. Such splines are called B$r$-splines, where $r$ denotes the number of different polynomial pieces contained in each parameter interval. Compared to B$1$ splines, B$r$ splines can have lower degree, at equal support width, continuity and order of approximation. For example, the classes $D^2C^1P^2S^4$ or $D^3C^2P^2S^4$ mentioned above are respectively of B$2$ and B$3$ type.
\end{rem}

We now introduce a local notation for the restriction of a spline interpolant $\bF$ to a parametric interval $[x_s,x_{s+1}]$. This will allow us to emphasize the local dependence on data and parameters and will be useful to generalize the above framework to the bivariate case.
To this aim, we shall observe that in each parametric interval there are exactly $w$ nonzero fundamental functions and each of them depends on a sequence of $w-1$ parameter intervals of the form
\begin{equation}\label{eq:di}
d_i = x_{i+1}-x_{i}.
\end{equation}
Therefore, we can conveniently define the vector
\begin{equation}\label{eq:dell}
\bd = \left(d_{s-\frac{w}{2}+1},\dots,d_s,\dots,d_{s+\frac{w}{2}-1}\right),
\end{equation}
which contains all the parameter intervals necessary to evaluate $\bF$ in $[x_s,x_{s+1}]$. We call such a vector the
\emph{local parameter vector} relative to $[x_s,x_{s+1}]$ (or, equivalently, relative to $\overline{\bp_s \bp_{s+1}}$).
If we now map the interval $[x_s,x_{s+1}]$ to $[0,d_s]$,
we can write the segment of $\bF$ bounded by $\bp_s$ and $\bp_{s+1}$
as
\begin{equation}\label{eq:spline_di}
\bF(x)\big|_{[0,d_s]}=\sum_{i=s-\frac{w}{2}+1}^{s+\frac{w}{2}} \bp_{i} \, \psi_{i}(x,\bd),
\end{equation}
where $\psi_{i}(x,\bd)$ is the restriction of $\psi_{i}(x)$ to $[x_s,x_{s+1}]$ expressed in terms of the sequence of parameter intervals \eqref{eq:dell}.
%
\begin{figure}[t]
\centering
\subfigure[]{\includegraphics[width=0.35\textwidth]{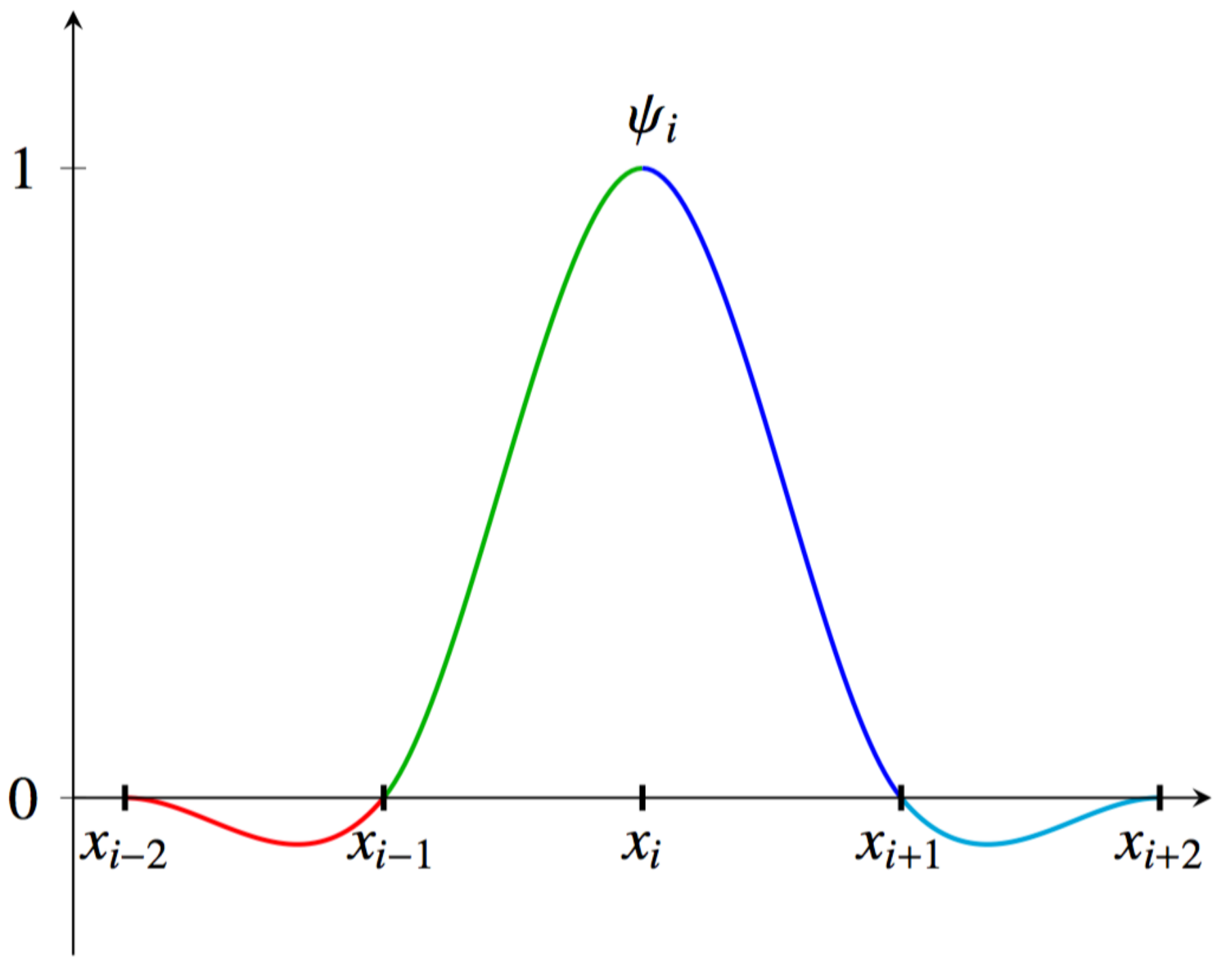}\label{fig:fund_D3C1P2S4}}
\hspace{1cm}
\subfigure[]{\includegraphics[width=0.35\textwidth]{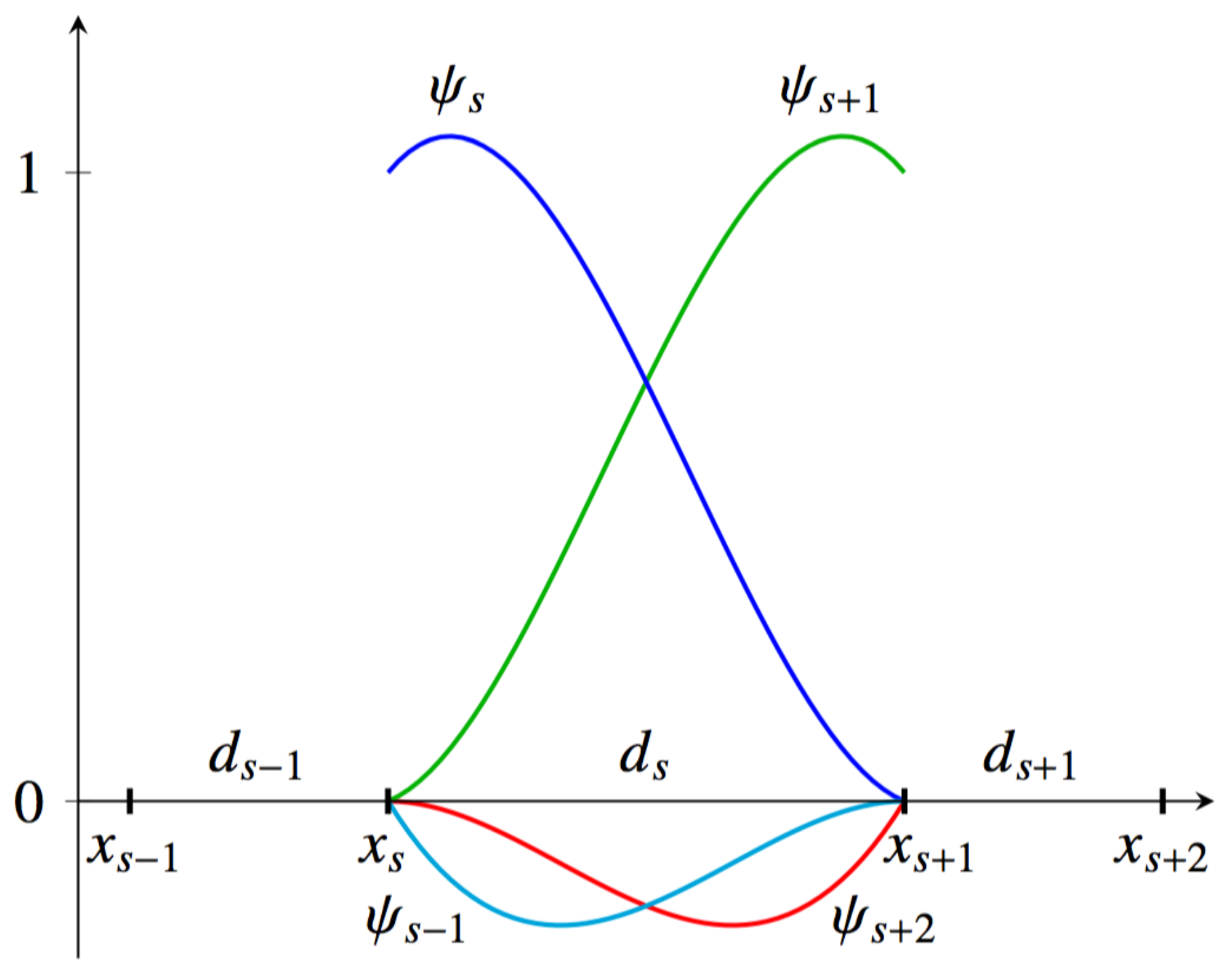}\label{fig:fund_D3C1P2S4_pieces}}
\caption{\subref{fig:fund_D3C1P2S4}
Fundamental function $\psi_i$ of the class $D^3C^1P^2S^4$, with the different pieces represented in different colors.
\subref{fig:fund_D3C1P2S4_pieces}
The fundamental functions of the class $D^3C^1P^2S^4$ that are nonzero in the interval $[x_s,x_{s+1}]$.
}
\label{fig:fund}
\end{figure}

As an example of our setting and notation, we provide in \apptext \ref{app:fund_eqs} the expressions for
the well-known Catmull-Rom splines \cite{Catmull-Rom1974}, corresponding to the class $D^3C^1P^2S^4$ in \cite{BCR13a}.
The fundamental functions of this class are also
illustrated in Figure \ref{fig:fund}.

Moreover, to facilitate the reader in reproducing the examples proposed in the following sections,
we also list in appendix the fundamental functions $D^5C^2P^2S^4$. The latter are an interesting application of the approach in \cite{BCR13a}, since they have high continuity, in spite of a very limited support.
Moreover we remark that, given our interest in interpolating 3-dimensional data, it may be also useful to consider splines with continuity $C^3$, or higher, and some graphical examples will be presented in the forthcoming sections.
As it can reasonably be expected, in this case the explicit expression of the fundamental functions is more complicated. Nevertheless, the evaluation of the spline interpolant can still be performed in a computationally efficient way as discussed in \cite{BCR13a,ABC13}.

As recalled in the previous section, the choice of the parameter sequence $\{x_j\}$ has a large influence on
the shape of an interpolating spline curve and in this respect various effective methods can be considered, such as  \cite{ANW1967,Lee1989,Foley-Nielson1989,Fang2013}.
Since the focus of this paper is not on comparing different techniques, we will take as a running example the centripetal parametrization, which is acknowledged to produce good results for the majority of critical data sets.
According to the centripetal parametrization, fixed $x_0$, the parameter values $x_j$, $j=1,\dots,N$ are computed through
\begin{equation}\label{eq:xi}
x_{i+1}=x_{i}+\norm{\bp_{i+1}-\bp_i}_2^\alpha,
\end{equation}
with  $\alpha=\frac{1}{2}$ and $\norm{\cdot}_2$ denoting the Euclidean norm.
It is therefore clear that the parameter values depend on the geometry of the interpolation points.
In addition, from \eqref{eq:xi}, the chordal and uniform parameterizations can be derived, setting respectively $\alpha=1$ and $\alpha=0$ \cite{Farin2002a,Farin2002b}.

In the following section we generalize to the bivariate setting the considered local, non-uniform, univariate spline interpolants.
Before proceeding, we need to remark that our assumption to work with periodic data has the sole purpose of simplifying the presentation.
Open curves can similarly be defined, provided that \qtext{special} fundamental functions are used in order to evaluate \eqref{eq:spline} in the boundary intervals.
These fundamental functions are defined so as to interpolate the derivatives of $\bF$ up to suitable order at $x_0$ and $x_N$ and we refer the reader to \cite{ABC13} for more details. The surface construction developed in the following section can be adapted to handle open data sets, along the same lines of the univariate case.

%
%
%
\section{Local interpolation of regular meshes by high quality surfaces of arbitrary continuity}
\label{sec:regular_case}
In this section we assume that a regular mesh of points in $\RR^3$ is given and
we develop a local method to generate an interpolatory surface of good quality and high order of continuity, based on the univariate fundamental functions introduced in the previous section and on a proper technique of parametrization.
By \qtext{good quality} we mean that the interpolating surface will fit the shape suggested by the input data in a faithful manner and will not present undesired interpolation artifacts.
Moreover, we will prove that the surface has the same smoothness of the underlying class of fundamental functions, which can have arbitrary continuity in the general framework \cite{BCR13a,ABC13}.

A regular mesh is one where every vertex belongs to four edges (and faces)
and can thus be seen as a rectangular grid of 3D points, whose edges are associated with two independent domain directions.
We say that two edges are \emph{opposite} when they have no vertex in common and they belong to the same face,
whereas we say that two edges are \emph{adjacent} when they have
one vertex in common and do not belong to the same face.
We call an \emph{edge ribbon} any ordered sequence of pairwise opposite edges and a
\emph{section polyline} a polyline formed by a sequence of pairwise adjacent edges. Moreover, we call a \emph{section curve} any curve that interpolates the vertices of a section polyline.
For simplicity, the discussion will be limited to meshes without boundary and therefore we can assume that all section polylines and curves are closed.

In the introduction to this work, we have drawn the reader's attention the well-known fact that a tensor-product surface is not, in general, an interpolant of good quality.
This is because all isocurves of such a surface must have the same parametrization.
As recalled, the latter requirement may result in an unnatural wiggling of the section curves that is very likely to happen when the data points are unevenly spaced and is even more evident when local interpolation methods are used.

In contrast to the tensor product technique, we wish to construct a surface where the vertices of each section polyline are interpolated at the parameter values that allow for the best quality of the resulting section curve.

To this aim, for every section polyline, we derive a non-uniform parameter sequence by exploiting an appropriate data-dependent technique of parametrization. Then we assign to the section polyline edges the resulting parameter intervals.
For instance, labeled by $\bp_{i,j}$ the mesh points, the centripetal parametrization applied to each section polyline will yield the edge parameter intervals

\begin{equation}\label{eq:d_e}
d_{i,j} \coloneqq \norm{\bp_{i+1,j}-\bp_{i,j}}_2^\alpha  \qquad \text{and} \qquad
e_{i,j} \coloneqq \norm{\bp_{i,j+1}-\bp_{i,j}}_2^\alpha, \qquad \alpha=\frac{1}{2}.
\end{equation}
This can be easily seen comparing the above expressions with formulae \eqref{eq:di}--\eqref{eq:xi}.
\begin{figure}[t]
\centering

\begin{tikzpicture}

\tikzstyle{knot} = [shape=circle,draw=black,fill=white,minimum size=4pt,inner sep=0pt]
\tikzstyle{cp} = [anchor=west,inner sep=0pt,xshift=0.5ex,yshift=-1.5ex]
\tikzstyle{tagd} = [anchor=south,inner sep=1pt,xshift=0ex,yshift=0ex]
\tikzstyle{tage} = [anchor=east,inner sep=1pt,xshift=0ex,yshift=0ex]
\tikzstyle{ext} = [draw=black]

\newcommand{\GridScale}{2}

\path ($\GridScale*(-2,-0.6)$) coordinate (Pm10); \path ($\GridScale*(-1,-0.6)$) coordinate (P00); \path ($\GridScale*(0,-0.6)$) coordinate (P10); \path ($\GridScale*(1,-0.6)$) coordinate (P20); \path ($\GridScale*(2,-0.6)$) coordinate (P30);
\path ($\GridScale*(-2,0)$) coordinate (Pm11); \path ($\GridScale*(-1,0)$) coordinate (P01); \path ($\GridScale*(0,0)$) coordinate (P11); \path ($\GridScale*(1,0)$) coordinate (P21); \path ($\GridScale*(2,0)$) coordinate (P31);
\path ($\GridScale*(-2,1)$) coordinate (Pm12); \path ($\GridScale*(-1,1)$) coordinate (P02); \path ($\GridScale*(0,1)$) coordinate (P12); \path ($\GridScale*(1,1)$) coordinate (P22); \path ($\GridScale*(2,1)$) coordinate (P32);
\path ($\GridScale*(-2,1.6)$) coordinate (Pm13); \path ($\GridScale*(-1,1.6)$) coordinate (P03); \path ($\GridScale*(0,1.6)$) coordinate (P13); \path ($\GridScale*(1,1.6)$) coordinate (P23); \path ($\GridScale*(2,1.6)$) coordinate (P33);
\path ($\GridScale*(0.25,0)$) coordinate (dx);
\path ($\GridScale*(0,0.25)$) coordinate (dy);
\path ($(Pm10)-(dy)$) coordinate (Pm10ym); \path ($(P00)-(dy)$) coordinate (P00ym); \path ($(P10)-(dy)$) coordinate (P10ym); \path ($(P20)-(dy)$) coordinate (P20ym); \path ($(P30)-(dy)$) coordinate (P30ym);
\path ($(Pm13)+(dy)$) coordinate (Pm13yp); \path ($(P03)+(dy)$) coordinate (P03yp); \path ($(P13)+(dy)$) coordinate (P13yp); \path ($(P23)+(dy)$) coordinate (P23yp); \path ($(P33)+(dy)$) coordinate (P33yp);
\path ($(Pm10)-(dx)$) coordinate (Pm10xm); \path ($(Pm11)-(dx)$) coordinate (Pm11xm); \path ($(Pm12)-(dx)$) coordinate (Pm12xm); \path ($(Pm13)-(dx)$) coordinate (Pm13xm);
\path ($(P30)+(dx)$) coordinate (P30xp); \path ($(P31)+(dx)$) coordinate (P31xp); \path ($(P32)+(dx)$) coordinate (P32xp); \path ($(P33)+(dx)$) coordinate (P33xp);
\path[ext] (Pm10)--(Pm10ym); \path[ext] (P00)--(P00ym); \path[ext] (Pm10)--(Pm10ym); \path[ext] (P10)--(P10ym); \path[ext] (P20)--(P20ym); \path[ext] (P30)--(P30ym);
\path[ext] (Pm13)--(Pm13yp); \path[ext] (P03)--(P03yp); \path[ext] (Pm13)--(Pm13yp); \path[ext] (P13)--(P13yp); \path[ext] (P23)--(P23yp); \path[ext] (P33)--(P33yp);
\path[ext] (Pm10)--(Pm10xm); \path[ext] (Pm11)--(Pm11xm); \path[ext] (Pm12)--(Pm12xm); \path[ext] (Pm13)--(Pm13xm);
\path[ext] (P30)--(P30xp); \path[ext] (P31)--(P31xp); \path[ext] (P32)--(P32xp); \path[ext] (P33)--(P33xp);

\path[fill=black!20,opacity=0.8] (P11)--(P21)--(P22)--(P12)--cycle;
\node[] at ($0.25*(P11)+0.25*(P21)+0.25*(P22)+0.25*(P12)$) {$\bS$};
\path[fill=black!10,opacity=0.7] (P01)--(P11)--(P12)--(P02)--cycle;
\node[] at ($0.25*(P01)+0.25*(P11)+0.25*(P12)+0.25*(P02)$) {$\tilde{\bS}$};

\path ($\GridScale*(0.17,0)$) coordinate (du);
\path ($\GridScale*(0,0.17)$) coordinate (dv);
\path ($(P11)+(0.12,0.12)$) coordinate (S00);
\draw[-stealth] (S00) -- ($(S00)+(du)$);
\node[anchor=west,inner sep=1pt] at ($(S00)+(du)$) {\footnotesize$u$};
\draw[-stealth] (S00) -- ($(S00)+(dv)$);
\node[anchor=south,inner sep=1pt] at ($(S00)+(dv)$) {\footnotesize$v$};
\path ($(P01)+(0.12,0.12)$) coordinate (St00);
\draw[-stealth] (St00) -- ($(St00)+(du)$);
\node[anchor=west,inner sep=1pt] at ($(St00)+(du)$) {\footnotesize$u$};
\draw[-stealth] (St00) -- ($(St00)+(dv)$);
\node[anchor=south,inner sep=1pt] at ($(St00)+(dv)$) {\footnotesize$v$};

\node[cp] at (Pm10) {$\bp_{-2,-1}$}; \node[cp] at (P00) {$\bp_{-1,-1}$}; \node[cp] at (P10) {$\bp_{0,-1}$}; \node[cp] at (P20) {$\bp_{1,-1}$}; \node[cp] at (P30) {$\bp_{2,-1}$};
\node[cp] at (Pm11) {$\bp_{-2,0}$}; \node[cp] at (P01) {$\bp_{-1,0}$}; \node[cp] at (P11) {$\bp_{0,0}$}; \node[cp] at (P21) {$\bp_{1,0}$}; \node[cp] at (P31) {$\bp_{2,0}$};
\node[cp] at (Pm12) {$\bp_{-2,1}$}; \node[cp] at (P02) {$\bp_{-1,1}$}; \node[cp] at (P12) {$\bp_{0,1}$}; \node[cp] at (P22) {$\bp_{1,1}$}; \node[cp] at (P32) {$\bp_{2,1}$};
\node[cp] at (Pm13) {$\bp_{-2,2}$}; \node[cp] at (P03) {$\bp_{-1,2}$}; \node[cp] at (P13) {$\bp_{0,2}$}; \node[cp] at (P23) {$\bp_{1,2}$}; \node[cp] at (P33) {$\bp_{2,2}$};

\draw (Pm10)--(P00)--(P10)--(P20)--(P30);
\draw (Pm11)--(P01)--(P11)--(P21)--(P31);
\draw (Pm12)--(P02)--(P12)--(P22)--(P32);
\draw (Pm13)--(P03)--(P13)--(P23)--(P33);
\draw (Pm10)--(Pm11)--(Pm12)--(Pm13);
\draw (P00)--(P01)--(P02)--(P03);
\draw (P10)--(P11)--(P12)--(P13);
\draw (P20)--(P21)--(P22)--(P23);
\draw (P30)--(P31)--(P32)--(P33);

\foreach \coord in {(Pm10),(P00),(P10),(P20),(P30), (Pm11),(P01),(P11),(P21),(P31), (Pm12),(P02),(P12),(P22),(P32), (Pm13),(P03),(P13),(P23),(P33)}
{ \node[knot] at \coord {}; }

\node[tagd] at ($0.5*(Pm10)+0.5*(P00)$) {$d_{-2,-1}$}; \node[tagd] at ($0.5*(Pm11)+0.5*(P01)$) {$d_{-2,0}$}; \node[tagd] at ($0.5*(Pm12)+0.5*(P02)$) {$d_{-2,1}$}; \node[tagd] at ($0.5*(Pm13)+0.5*(P03)$) {$d_{-2,2}$};
\node[tagd] at ($0.5*(P00)+0.5*(P10)$) {$d_{-1,-1}$}; \node[tagd] at ($0.5*(P01)+0.5*(P11)$) {$d_{-1,0}$}; \node[tagd] at ($0.5*(P02)+0.5*(P12)$) {$d_{-1,1}$}; \node[tagd] at ($0.5*(P03)+0.5*(P13)$) {$d_{-1,2}$};
\node[tagd] at ($0.5*(P10)+0.5*(P20)$) {$d_{0,-1}$}; \node[tagd] at ($0.5*(P11)+0.5*(P21)$) {$d_{0,0}$}; \node[tagd] at ($0.5*(P12)+0.5*(P22)$) {$d_{0,1}$}; \node[tagd] at ($0.5*(P13)+0.5*(P23)$) {$d_{0,2}$};
\node[tagd] at ($0.5*(P20)+0.5*(P30)$) {$d_{1,-1}$}; \node[tagd] at ($0.5*(P21)+0.5*(P31)$) {$d_{1,0}$}; \node[tagd] at ($0.5*(P22)+0.5*(P32)$) {$d_{1,1}$}; \node[tagd] at ($0.5*(P23)+0.5*(P33)$) {$d_{1,2}$};
\node[tage] at ($0.5*(Pm10)+0.5*(Pm11)$) {$e_{-2,-1}$}; \node[tage] at ($0.5*(P00)+0.5*(P01)$) {$e_{-1,-1}$}; \node[tage] at ($0.5*(P10)+0.5*(P11)$) {$e_{0,-1}$}; \node[tage] at ($0.5*(P20)+0.5*(P21)$) {$e_{1,-1}$}; \node[tage] at ($0.5*(P30)+0.5*(P31)$) {$e_{2,-1}$};
\node[tage] at ($0.5*(Pm11)+0.5*(Pm12)$) {$e_{-2,0}$}; \node[tage] at ($0.5*(P01)+0.5*(P02)$) {$e_{-1,0}$}; \node[tage] at ($0.5*(P11)+0.5*(P12)$) {$e_{0,0}$}; \node[tage] at ($0.5*(P21)+0.5*(P22)$) {$e_{1,0}$}; \node[tage] at ($0.5*(P31)+0.5*(P32)$) {$e_{2,0}$};
\node[tage] at ($0.5*(Pm12)+0.5*(Pm13)$) {$e_{-2,1}$}; \node[tage] at ($0.5*(P02)+0.5*(P03)$) {$e_{-1,1}$}; \node[tage] at ($0.5*(P12)+0.5*(P13)$) {$e_{0,1}$}; \node[tage] at ($0.5*(P22)+0.5*(P23)$) {$e_{1,1}$}; \node[tage] at ($0.5*(P32)+0.5*(P33)$) {$e_{2,1}$};

\end{tikzpicture}
\caption{
Labeling of vertices and parameter intervals for the construction of the augmented surface patch $\bS$ having the expression \protect\eqref{eq:S}.
}
\label{fig:grid}
\end{figure}
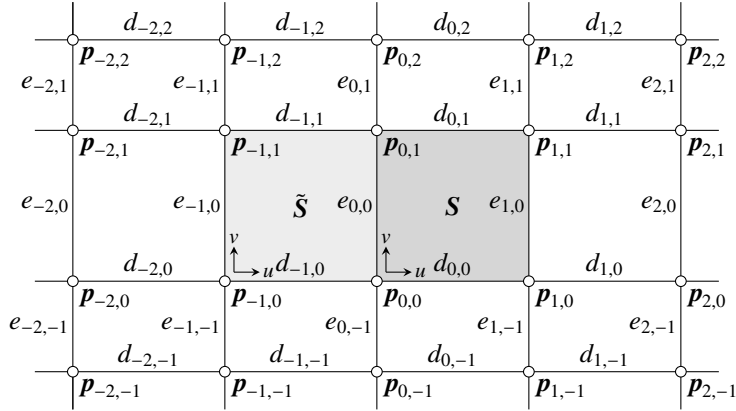

We call the resulting configuration of parameters an \emph{augmented parametrization}. The augmented parametrization is
featured by the fact that the parameter intervals allocated to the edges of one mesh face will not form, in general, a rectangle. 
The terminology was firstly used in \cite{Muller2006}, dealing with non-uniform subdivision schemes, where augmented faces of a mesh are defined with a similar meaning.
Conversely, in a tensor product surface, the parameter intervals must be equal for all edges belonging in the same edge ribbon, or, equivalently, the intervals allocated to the edges of every mesh face must form a rectangle.

We now aim to construct a composite surface, which we call an \emph{augmented surface}, where every section curve is a local, univariate, spline of class $D^gC^kP^mS^w$ and has an independent parametrization determined by the edge parameter intervals of the corresponding section polyline.
This means that the section curve piece bounded by $\bp_{i,j}$ and $\bp_{i+1,j}$ should have an associated parametric interval of length equal to $d_{i,j}$.
Analogously, the parametric interval between any two vertices, $\bp_{i,j}$ and $\bp_{i,j+1}$ of the corresponding section curve should have length $e_{i,j}$.
In this way, each section polyline of the mesh can be interpolated at the parameter values that allow the best quality of the resulting section curve.
Moreover, such curves serve as a guide for the shape of the surface, which will be in turn a good quality interpolant of the input mesh.

It is sufficient to illustrate the construction for a generic surface patch interpolating the points
$\bp_{s,t}, \bp_{s+1,t}, \bp_{s,t+1}, \allowbreak \bp_{s+1,t+1}$ and in particular, without loss of generality, we can consider the patch which interpolates $\bp_{0,0}, \bp_{1,0}, \bp_{0,1}, \bp_{1,1}$, depicted in Figure \ref{fig:grid}.
(As of now, the neighboring patch $\tilde{\bS}$ in Figure \ref{fig:grid} shall be overlooked. It will be used later in the proof of Proposition \ref{prop:Gamma}.)
Any other patch can be analogously derived after proper index shift.

Let $[0,1]^2$ be the parametric domain associated with the considered patch.
The first stage of the construction is to choose a class of local, univariate, spline interpolants $D^gC^kP^mS^w$, based on the properties that we seek in the final surface.
Hence, in view of the $C^k$ continuity of the fundamental functions, we consider the polynomials $\delta_{i,j}(v)$
$i=-\frac{w}{2}+1,\dots,\frac{w}{2}-1, j=0$ and $\epsilon_{i,j}(u)$ $i=0, j=-\frac{w}{2}+1,\dots,\frac{w}{2}-1$,
of degree $2k+1$, uniquely determined by the following conditions:
\begin{equation}\stepcounter{equation}\tag{{\theequation}a}\label{eq:dh}
\begin{alignedat}{5}
& \delta_{i,j}:[0,1]\rightarrow [d_{i,j}, d_{i,j+1}],\\[1ex]
& \delta_{i,j}(0)= d_{i,j}, \qquad \text{and} \qquad \delta_{i,j}(1)=d_{i,j+1},\\[1ex]
& d_{i,j}^{(r)}(0) = d_{i,j}^{(r)}(1)=0, \qquad r=1,\dots,k,\quad
\end{alignedat}
\end{equation}
and
\begin{equation}\tag{{\theequation}b}\label{eq:eh}
\begin{alignedat}{5}
& \epsilon_{i,j}:[0,1]\rightarrow [e_{i,j}, e_{i+1,j}], \\[1ex]
& \epsilon_{i,j}(0)=e_{i,j}, \qquad \text{and} \qquad \epsilon_{i,j}(1)=e_{i+1,j},\\[1ex]
& \epsilon_{i,j}^{(r)}(0) = \epsilon_{i,j}^{(r)}(1)=0, \quad r=1,\dots,k.\quad
\end{alignedat}
\end{equation}
The above polynomials, which we call the \emph{local parametrization functions}, interpolate the two parameter intervals corresponding to opposite edges of a mesh face and therefore have the effect of \qtext{blending} such intervals.  Moreover, the requirement that  $\delta_{i,j}(v)$ and $\epsilon_{i,j}(u)$ have vanishing derivatives up to order $k$ at $0$ and $1$ guarantees that they are monotonic functions with positive first derivative.

We use the local parametrization functions to associate with any couple of coordinates $(u,v)\in [0,1]^2$
two local parameter vectors $\bd$ and $\be$ as follows:
\begin{equation}\stepcounter{equation}\tag{{\theequation}a}\label{eq:poly_d}
\bd = \bd(v)= \left(\delta_{-\frac{w}{2}+1,0}(v),\dots, \delta_{0,0}(v),\dots \delta_{\frac{w}{2}-1,0}(v)\right),
\end{equation}
and
\begin{equation}\tag{{\theequation}b}\label{eq:poly_e}
\be = \be(u) = \left(\epsilon_{0,-\frac{w}{2}+1}(u),\dots,\epsilon_{0,0}(u),\dots,\epsilon_{0,\frac{w}{2}-1}(u)\right).
\end{equation}
In this way, at the patch boundary described by $(u,0)$, $u\in[0,1]$, $\bd= \left(d_{-w/2+1,0},\dots,d_{0,0},\dots,d_{w/2-1,0}\right)$ is the local parameter vector relative to $\overline{\bp_{0,0}\bp_{1,0}}$ (see definition \eqref{eq:dell}).
In addition, at the opposite boundary $(u,1)$, $\bd$ is the local parameter vector relative to $\overline{\bp_{0,1}\bp_{1,1}}$.
Analogously, for any point $(0,v)$ or $(1,v)$, $v\in[0,1]$, $\be$ is the local parameter vector associated respectively with $\overline{\bp_{0,0}\bp_{0,1}}$ or $\overline{\bp_{1,0}\bp_{1,1}}$.

We also consider the two \emph{local variables}
\begin{equation}\label{eq:xy}
\begin{aligned}
x & = u\,\delta_{0,0}(v), \qquad \text{and} \qquad
y & = v\,\epsilon_{0,0}(u),
\end{aligned}
\end{equation}
such that $x$ and $y$ span the intervals $[0,\delta_{0,0}(v)]$ and $[0,\epsilon_{0,0}(u)]$ while respectively $u$ and $v$ vary in $[0,1]$.

Finally, for any $(u,v)\in[0,1]^2$, we define the augmented surface patch $\bS$ by
\begin{equation}\label{eq:S}
\bS(u,v)= \sum_{i=-\frac{w}{2}+1}^{\frac{w}{2}} \sum_{j=-\frac{w}{2}+1}^{\frac{w}{2}}
\bp_{i,j} \Psi_{i,j}\Big(x,y,\bd,\be\Big),
\end{equation}
where
\begin{equation}\label{eq:biv_fund}
\Psi_{i,j}\Big(x,y,\bd,\be\Big) = \psi_i\Big(x,\bd\Big) \, \psi_j\Big(y,\be \Big),
\end{equation}
with local parameter vectors $\bd,\be$ given by \eqref{eq:poly_d}--\eqref{eq:poly_e} and $x,y$ computed through \eqref{eq:xy}.

Note that $\bS$ depends on a local grid of mesh vertices of size equal to $w\times w$, where $w$ is the support width of the underlying functions $\psi_i$ and $\psi_j$. For brevity, we will sometimes say that the patch has support $w$.

\begin{rem}
For simplicity of presentation we have assumed that the functions $\psi_i(x,\bd)$ and $\psi_j(y,\be)$ in \eqref{eq:biv_fund} belong to the same family $D^gC^kP^mS^w$.
However, it is nor difficult to see that the construction of an augmented patch can be adapted for the case where the fundamental functions belong to two different classes $D^gC^kP^mS^w$.
\end{rem}

We remark that the augmented patch $\bS$ is not a tensor product surface. More precisely, we can interpret the above construction as follows.
For a given $(\bar{u},\bar{v})\in [0,1]^2$, let $\bar{\bd}$ and $\bar{\be}$ be the vectors
\begin{equation}
\bar{\bd} = \bd(\bar{v}), \quad \bar{\be} = \be(\bar{u}),
\end{equation}
computed through \eqref{eq:poly_d}--\eqref{eq:poly_e} and let $\bar{d}=\delta_{0,0}(\bar{v})$ and $\bar{e}=\epsilon_{0,0}(\bar{u})$.
Then the value of $\bS$ at $(\bar{u},\bar{v})$ is equal to the value at $(\bar{x},\bar{y})$ of the tensor product patch
\begin{equation}\label{eq:T}
\bT_{(\bar{u},\bar{v})}(x,y) = \sum_{i=-\frac{w}{2}+1}^{\frac{w}{2}} \sum_{j=-\frac{w}{2}+1}^{\frac{w}{2}}
\bp_{i,j} \psi_i\Big(x,\bar{\bd}\Big) \, \psi_j\Big(y,\bar{\be}\Big),
\qquad (x,y)\in[0,\bar{d}]\times [0,\bar{e}].
\end{equation}
\ie $\bS(\bar{u},\bar{v})=\bT_{(\bar{u},\bar{v})}(\bar{x},\bar{y})$.
In this view, a different tensor-product patch $\bT_{(\bar{u},\bar{v})}$ determines the value of $\bS$ at each domain point $(\bar{u},\bar{v})\in [0,1]^2$.


The construction of the augmented surface patch $\bS$ is schematized in
Figure \ref{fig:augm} for a class of fundamental functions having support width $w=4$ and, more generally, can be summarized as follows. Every mesh face gives rise to a surface patch, parameterized over the domain $[0,1]^2$.
For a given $(\bar{u},\bar{v})\in [0,1]^2$ we locally interpolate the parameter intervals, separately in the $u$ and $v$ direction, in order to generate two local parameter vectors $\bar{\bd}=\bd(\bar{v})$ and $\bar{\be}=\be(\bar{u})$.
We also map $(\bar{u},\bar{v})$ into a couple of values $(\bar{x},\bar{y})$. This mapping is such that, at the boundaries $(u,0)$ and $(u,1)$, $u\in[0,1]$, $x$ spans respectively the entire intervals $[0,d_{0,0}]$ and $[0,d_{0,1}]$. Analogously, at the boundaries $(0,v)$ and $(1,v)$, $v\in[0,1]$, $y$ spans respectively the intervals $[0,e_{0,0}]$ and $[0,e_{1,0}]$. These are precisely the parameter intervals allocated to the edges of the given mesh face.
Finally, we consider the surface $\bT_{(\bar{u},\bar{v})}(x,y)$ defined as tensor product of the local univariate fundamental functions of class $D^gC^kP^mS^w$ on $\bar{\bd}$ and $\bar{\be}$ and we set
$\bS(\bar{u},\bar{v})=\bT_{(\bar{u},\bar{v})}(\bar{x},\bar{y})$.

\begin{figure}\label{fig:augm}
\centering
\begin{tikzpicture}
\tikzset{myptr/.style={decoration={markings,mark=at position 1 with %
    {\arrow[scale=6,>=stealth]{>}}},postaction={decorate}}}
\node[anchor=north west] at (0.2,0.4) {\begin{tikzpicture}[scale=1]
\definecolor{cobalt}{rgb}{0.0, 0.28, 0.67}
\definecolor{forestgreenweb}{rgb}{0.13, 0.55, 0.13}
\definecolor{cssgreen}{rgb}{0.0, 0.5, 0.0}


\tikzstyle{patch} = [pattern color=black!30]
\tikzstyle{knot} = [shape=circle,draw=black,fill=white,minimum size=4pt,inner sep=0pt,yshift=1.5,xshift=-1.5]

\pgfmathsetmacro{\GridScale}{1.35}
\pgfmathsetmacro{\dim}{1}
\pgfmathsetmacro{\di}{1}
\pgfmathsetmacro{\dip}{1}
\pgfmathsetmacro{\eim}{1}
\pgfmathsetmacro{\ei}{1}
\pgfmathsetmacro{\eip}{1}
\pgfmathsetmacro{\ubar}{0.2}
\pgfmathsetmacro{\vbar}{0.3}
\pgfmathsetmacro{\tickL}{0.07}

\path ($\GridScale*(-\dim,-\eim)$) coordinate (P00); \path ($\GridScale*(0,-\eim)$) coordinate (P10); \path ($\GridScale*(\di,-\eim)$) coordinate (P20); \path ($\GridScale*(\di+\dip,-\eim)$) coordinate (P30);
\path ($\GridScale*(-\dim,0)$) coordinate (P01); \path ($\GridScale*(0,0)$) coordinate (P11); \path ($\GridScale*(\di,0)$) coordinate (P21); \path ($\GridScale*(\di+\dip,0)$) coordinate (P31);
\path ($\GridScale*(-\dim,\ei)$) coordinate (P02); \path ($\GridScale*(0,\ei)$) coordinate (P12); \path ($\GridScale*(\di,\ei)$) coordinate (P22); \path ($\GridScale*(\di+\dip,\ei)$) coordinate (P32);
\path ($\GridScale*(-\dim,\ei+\eip)$) coordinate (P03); \path ($\GridScale*(0,\ei+\eip)$) coordinate (P13); \path ($\GridScale*(\di,\ei+\eip)$) coordinate (P23); \path ($\GridScale*(\di+\dip,\ei+\eip)$) coordinate (P33);

\path ($(P01)+\GridScale*(0,\vbar*\ei)$) coordinate (D0); \path ($(P11)+\GridScale*(0,\vbar*\ei)$) coordinate (D1); \path ($(P21)+\GridScale*(0,\vbar*\ei)$) coordinate (D2); \path ($(P31)+\GridScale*(0,\vbar*\ei)$) coordinate (D3);
\path ($(P10)+\GridScale*(\ubar*\di,0)$) coordinate (E0); \path ($(P11)+\GridScale*(\ubar*\di,0)$) coordinate (E1); \path ($(P12)+\GridScale*(\ubar*\di,0)$) coordinate (E2); \path ($(P13)+\GridScale*(\ubar*\di,0)$) coordinate (E3);

\path[patch] (P11)--(P21)--(P22)--(P12)--cycle;
\node[] at ($0.25*(P11)+0.25*(P21)+0.25*(P22)+0.25*(P12)$) {};

\path[fill=black!40,fill opacity=0.3] (P11)--(P21)--(P22)--(P12)--cycle;
\draw (P00)--(P10)--(P20)--(P30);
\draw (P01)--(P11)--(P21)--(P31);
\draw (P02)--(P12)--(P22)--(P32);
\draw (P03)--(P13)--(P23)--(P33);
\draw (P00)--(P01)--(P02)--(P03);
\draw (P10)--(P11)--(P12)--(P13);
\draw (P20)--(P21)--(P22)--(P23);
\draw (P30)--(P31)--(P32)--(P33);
\draw[cssgreen,line width=1.2pt] (D0)--(D1)--(D2)--(D3);
\draw[cobalt,line width=1.2pt] (E0)--(E1)--(E2)--(E3);


\foreach \coord in {(P00),(P10),(P20),(P30),(P01),(P11),(P21),(P31),(P02),(P12),(P22),(P32),(P03),(P13),(P23),(P33)}
{ \node[knot] at \coord {}; }

\path (P01)--(P11) node[midway,below,inner sep=1pt] {$d_{-1,0}$};
\path (P11)--(P21) node[midway,below,inner sep=1pt] {$d_{0,0}$};
\path (P21)--(P31) node[midway,below,inner sep=1pt] {$d_{1,0}$};
\path (P02)--(P12) node[midway,above,inner sep=1pt] {${d}_{-1,1}$};
\path (P12)--(P22) node[midway,above,inner sep=1pt] {${d}_{0,1}$};
\path (P22)--(P32) node[midway,above,inner sep=1pt] {${d}_{1,1}$};
\path (P10)--(P11) node[midway,left,inner sep=1pt,yshift=-4pt] {$e_{0,-1}$};
\path (P11)--(P12) node[midway,left,inner sep=1pt,yshift=8pt] {$e_{0,0}$};
\path (P12)--(P13) node[midway,left,inner sep=1pt,yshift=4pt] {$e_{0,1}$};
\path (P20)--(P21) node[midway,right,inner sep=1pt,yshift=-4pt] {${e}_{1,-1}$};
\path (P21)--(P22) node[midway,right,inner sep=1pt,yshift=8pt] {${e}_{1,0}$};
\path (P22)--(P23) node[midway,right,inner sep=1pt,yshift=4pt] {${e}_{1,1}$};

\path (D0)--(D1) node[midway,above,inner sep=1pt,xshift=-4pt,text=cssgreen] {$\delta_{-1,0}(\bar{v})$};
\path (D1)--(D2) node[midway,above,inner sep=1pt,xshift=2pt,text=cssgreen] {$\delta_{0,0}(\bar{v})$};
\path (D2)--(D3) node[midway,above,inner sep=1pt,xshift=4pt,text=cssgreen] {$\delta_{1,0}(\bar{v})$};
\path (E0)--(E1) node[midway,right,inner sep=1pt,yshift=-4pt,text=cobalt] {$\epsilon_{0,-1}(\bar{u})$};
\path (E1)--(E2) node[midway,right,inner sep=1pt,yshift=12pt,text=cobalt] {$\epsilon_{0,0}(\bar{u})$};
\path (E2)--(E3) node[midway,right,inner sep=1pt,yshift=4pt,text=cobalt] {$\epsilon_{0,1}(\bar{u})$};

\end{tikzpicture}};
\node[anchor=north west] at (12.6,0.8) {
%
\begin{tikzpicture}[scale=1.2]

\tikzstyle{every node}=[font=\small]

\tikzstyle{pt} = [shape=circle,anchor=center,draw=black,fill=red,minimum size=4pt,inner sep=0pt]
\tikzstyle{edge} = [draw=black,dashed]
\tikzstyle{crv} = [draw=black,thick]
\tikzstyle{knot} = [shape=circle,draw=black,fill=white,minimum size=4pt,inner sep=0pt]

\pgfmathsetmacro{\ScaleFactor}{1.5}
\pgfmathsetmacro{\ubar}{0.2}
\pgfmathsetmacro{\vbar}{0.3}

\path ($\ScaleFactor*($(0,0)+(0,0)$)$) coordinate (p0);
\path ($\ScaleFactor*($(1,0)+(0.25,0)$)$) coordinate (p1);
\path ($\ScaleFactor*($(1,1)+(0.75,0)$)$) coordinate (p2);
\path ($\ScaleFactor*($(0,1)+(0.5,0)$)$) coordinate (p3);

\draw[edge] (p0)--(p1) node[midway,below] {$d_{0,0}$};
\draw[edge] (p1)--(p2) node[midway,right,xshift=3pt,yshift=7pt] {$e_{1,0}$};
\draw[edge] (p3)--(p2) node[midway,below] {$d_{0,1}$};
\draw[edge] (p0)--(p3) node[midway,right,xshift=2.5pt,yshift=7pt] {$e_{0,0}$};

\node[knot,xshift=-1.5pt,yshift=1.5pt] at (p0) {};
\node[knot,xshift=-1.5pt,yshift=1.5pt] at (p1) {};
\node[knot,xshift=-1.5pt,yshift=1.5pt] at (p2) {};
\node[knot,xshift=-1.5pt,yshift=1.5pt] at (p3) {};

\path [fill=black!40,fill opacity=0.3]
      (p0) to [crv,out=45,in=135,looseness=0.5] (p1)
       to [crv,out=90,in=135,looseness=0.5] (p2)
       to [crv,out=135,in=45,looseness=0.5] (p3)
       to [crv,out=135,in=90,looseness=0.5] (p0);

\draw[crv,out=45,in=135,looseness=0.5] (p0) to (p1);
\draw[crv,out=90,in=135,looseness=0.5] (p1) to (p2);
\draw[crv,out=45,in=135,looseness=0.5] (p3) to (p2);
\draw[crv,out=90,in=135,looseness=0.5] (p0) to (p3);

\path ($(p0)+\ScaleFactor*(1.3*\ubar,\vbar)$) coordinate (uv);
\node[pt] at (uv) {};
\node[anchor=west] at (uv) {$\bS(\bar{u},\bar{v})$};

\end{tikzpicture}};
\node[anchor=north west] at (6.6,-0.45) {
%
\begin{tikzpicture}[scale=1.2]
\definecolor{cobalt}{rgb}{0.0, 0.28, 0.67}
\definecolor{forestgreenweb}{rgb}{0.13, 0.55, 0.13}
\definecolor{cssgreen}{rgb}{0.0, 0.5, 0.0}

\tikzstyle{every node}=[font=\small]

\tikzstyle{pt} = [shape=circle,anchor=center,draw=black,fill=red,minimum size=4pt,inner sep=0pt]
\tikzstyle{crv} = [draw=black,thick,dashed]
\tikzstyle{augm} = [draw=cssgreen,line width=1.2pt]
\tikzstyle{augm2} = [draw=cobalt,line width=1.2pt]

\pgfmathsetmacro{\ScaleFactor}{1.5}
\pgfmathsetmacro{\Radius}{sqrt(5)/2}
\pgfmathsetmacro{\Angle}{acos(1/sqrt(5))}
\pgfmathsetmacro{\dim}{0.5}
\pgfmathsetmacro{\di}{1.25}
\pgfmathsetmacro{\dip}{0.5}
\pgfmathsetmacro{\eim}{0.4}
\pgfmathsetmacro{\ei}{1}
\pgfmathsetmacro{\eip}{0.4}
\pgfmathsetmacro{\ubar}{0.2}
\pgfmathsetmacro{\vbar}{0.3}
\pgfmathsetmacro{\de}{-0.3}
\pgfmathsetmacro{\tickL}{0.07}

\path ($\ScaleFactor*($(0,0)+(0,0)$)$) coordinate (p0);
\path ($\ScaleFactor*($(1,0)+(0.25,0)$)$) coordinate (p1);
\path ($\ScaleFactor*($(1,1)+(0.75,0)$)$) coordinate (p2);
\path ($\ScaleFactor*($(0,1)+(0.5,0)$)$) coordinate (p3);

\path ($(p0)-\ScaleFactor*\eim*\Radius*({cos(\Angle)},{sin(\Angle)})+\ScaleFactor*(-\dim,0)$) coordinate (tm1);
\path ($(tm1)+\ScaleFactor*\eim*\Radius*({cos(\Angle)},{sin(\Angle)})$) coordinate (t0);
\path ($(t0)+\ScaleFactor*\ei*\Radius*({cos(\Angle)},{sin(\Angle)})$) coordinate (t1);
\path ($(t1)+\ScaleFactor*\eip*\Radius*({cos(\Angle)},{sin(\Angle)})$) coordinate (t2);
\path ($(t0)+\ScaleFactor*\vbar*\ei*\Radius*({cos(\Angle)},{sin(\Angle)})$) coordinate (tbar);

\path ($(p0)-\ScaleFactor*\eim*\Radius*({cos(\Angle)},{sin(\Angle)})+\ScaleFactor*(-\dim,0)$) coordinate (sm1);
\path ($(sm1)+\ScaleFactor*\dim*(1,0)$) coordinate (s0);
\path ($(s0)+\ScaleFactor*\di*(1,0)$) coordinate (s1);
\path ($(s1)+\ScaleFactor*\dip*(1,0)$) coordinate (s2);
\path ($(s0)+\ScaleFactor*\ubar*\di*(1,0)$) coordinate (sbar);

\draw[crv,out=45,in=135,looseness=0.5] (p0) to (p1);
\draw[crv,out=90,in=135,looseness=0.5] (p1) to (p2);
\draw[crv,out=45,in=135,looseness=0.5] (p3) to (p2);
\draw[crv,out=90,in=135,looseness=0.5] (p0) to (p3);
\path ($(p0)+\ScaleFactor*(1.3*\ubar,\vbar)$) coordinate (uv);
\node[pt] at (uv) {};
\node[anchor=west] at (uv) {$\bT(\bar{x},\bar{y})=\bS(\bar{u},\bar{v})$};

\draw[augm2] (tm1)--(t0)--(t1)--(t2);
\draw[augm2] ($(tm1)-(\tickL,0)$)--($(tm1)+(\tickL,0)$);
\draw[augm2] ($(t0)-(\tickL,0)$)--($(t0)+(\tickL,0)$);
\draw[augm2] ($(t1)-(\tickL,0)$)--($(t1)+(\tickL,0)$);
\draw[augm2] ($(t2)-(\tickL,0)$)--($(t2)+(\tickL,0)$);
\draw ($(tbar)-(\tickL,0)$)--($(tbar)+(\tickL,0)$);
\path (tm1)--(t0) node[midway,left,inner sep=1pt,text=cobalt] {$\epsilon_{0,-1}(\bar{u})$};
\path (t0)--(t1) node[midway,left,inner sep=1pt,xshift=4pt,yshift=8pt,text=cobalt] {$\epsilon_{0,0}(\bar{u})$};
\path (t1)--(t2) node[midway,left,inner sep=1pt,text=cobalt] {$\epsilon_{0,1}(\bar{u})$};
\node[anchor=east] at (tbar) {$\bar{y}$};

\draw[augm] (sm1)--(s0)--(s1)--(s2);
\draw[augm] ($(sm1)-(0,\tickL)$)--($(sm1)+(0,\tickL)$);
\draw[augm] ($(s0)-(0,\tickL)$)--($(s0)+(0,\tickL)$);
\draw[augm] ($(s1)-(0,\tickL)$)--($(s1)+(0,\tickL)$);
\draw[augm] ($(s2)-(0,\tickL)$)--($(s2)+(0,\tickL)$);
\draw ($(sbar)-(0,\tickL)$)--($(sbar)+(0,\tickL)$);
\path (sm1)--(s0) node[midway,below,inner sep=2pt,text=cssgreen] {$\delta_{-1,0}(\bar{v})$};
\path (s0)--(s1) node[midway,below,inner sep=2pt,xshift=6pt,text=cssgreen] {$\delta_{0,0}(\bar{v})$};
\path (s1)--(s2) node[midway,below,inner sep=2pt,text=cssgreen] {$\delta_{1,0}(\bar{v})$};
\node[anchor=north] at (sbar) {$\bar{x}$};

\end{tikzpicture}};
\node[anchor=north west] at (5.5,1.4) {
%
\begin{tikzpicture}[scale=1.2]

\tikzstyle{pt} = [shape=circle,anchor=center,draw=black,fill=red,minimum size=4pt,inner sep=0pt]

\pgfmathsetmacro{\ScaleFactor}{1}
\pgfmathsetmacro{\ubar}{0.2}
\pgfmathsetmacro{\vbar}{0.3}
\pgfmathsetmacro{\di}{0.7}
\pgfmathsetmacro{\dti}{1.5}
\pgfmathsetmacro{\deltai}{\dti*(3*\vbar*\vbar-2*\vbar*\vbar*\vbar)+\di*(1-3*\vbar*\vbar+2*\vbar*\vbar*\vbar)}
\pgfmathsetmacro{\AxesL}{1.2}
\pgfmathsetmacro{\AxesUnit}{1}
\pgfmathsetmacro{\ArrowL}{1}
\pgfmathsetmacro{\FrameSpacing}{0.2}

\path (0,0) coordinate (Ouv);
\draw[-stealth] (Ouv)--($(Ouv)+\ScaleFactor*(\AxesL,0)$) node[anchor=west,inner sep=1pt] {\footnotesize$u$};
\draw[-stealth] (Ouv)--($(Ouv)+\ScaleFactor*(0,\AxesL)$) node[anchor=south,inner sep=1pt] {\footnotesize$v$};
\draw[black,fill=black!40,fill opacity=0.3] (Ouv) rectangle ($(Ouv)+\ScaleFactor*(\AxesUnit,\AxesUnit)$);

\node[anchor=north east,inner sep=1pt] at (Ouv) {\footnotesize$0$};
\node[anchor=north,inner sep=1pt] at ($(Ouv)+\ScaleFactor*(\AxesUnit,0)$) {\footnotesize$1$};
\node[anchor=east,inner sep=1pt] at ($(Ouv)+\ScaleFactor*(0,\AxesUnit)$) {\footnotesize$1$};
\path ($(Ouv)+\ScaleFactor*(\ubar*\AxesUnit,\vbar*\AxesUnit)$) coordinate (uv);
\node[pt] at (uv) {};
\node[anchor=west,inner sep=2pt,xshift=2pt] at (uv) {\footnotesize$(\bar{u},\bar{v})$};

\end{tikzpicture}};
\draw (6.05,-0.5) edge[out=-90,in=0,->,>=stealth',shorten >=1pt] (5,-2.1);
\draw (6.1,-0.5) edge[out=-90,in=180,->,>=stealth] (7.5,-2.45);
\draw (5,-2.65) edge[out=0,in=180,->,>=stealth] (7.5,-2.65);
\draw (11.5,-2.6) edge[out=0,in=-90,->,>=stealth] (13.2,-1.7);
\end{tikzpicture}
\caption{Construction of an augmented patch in the case $w=4$.}
\end{figure}
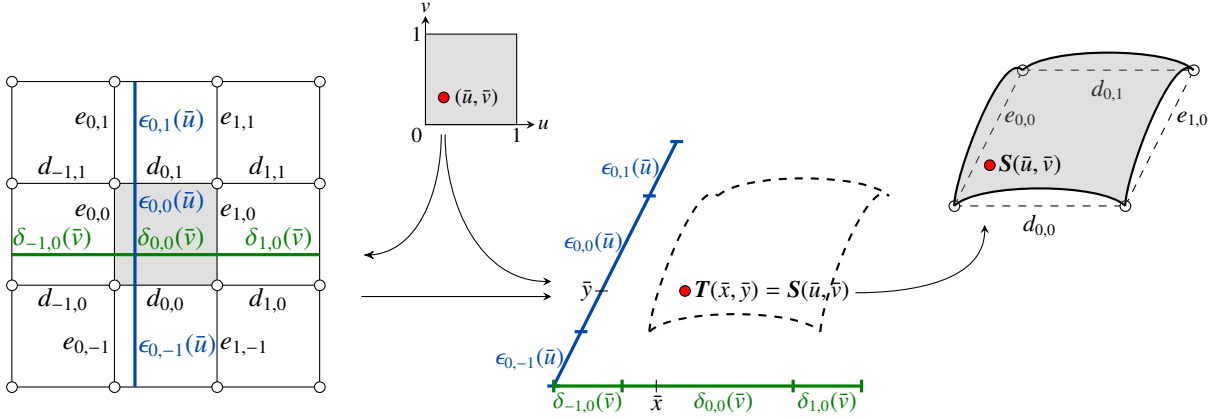

The remainder of this section is devoted to showing that the resulting composite surface has all the sought properties. We start by proving that the surface section curves form a network of univariate spline interpolants of class \class, where each curve has an independent non-uniform parametrization determined by the edge parameter intervals of the corresponding section polyline. 
In other words, if a network of section curves of class \class was independently determined before constructing the surface, then the surface would precisely be the transfinite interpolant of these curves.

\begin{prop}\label{prop:C0}
The section curves of an augmented surface based on a class of fundamental functions $D^gC^kP^mS^w$ are local univariate splines in the same class.
\end{prop}

\begin{proof}
It suffices to show that the section curve segment bounded by $\bp_{0,1}$ and $\bp_{1,1}$ (see Figure \ref{fig:grid}) belongs to the unique local, non-uniform spline curve $\bF$ of class \class that interpolates the section polyline with vertices $\dots,\bp_{-1,1},\bp_{0,1},\bp_{1,1},\bp_{2,1}\dots$ and parameter intervals $\dots,d_{-1,1},d_{0,1},d_{1,1},d_{2,1}\dots$.
Hence the statement can be extended to the entire network of section curves by locally applying the same argument.

According to the local representation \eqref{eq:spline_di}, we have
\begin{equation}\label{eq:F_1}
\bF(x)=\sum_{i=-\frac{w}{2}+1}^{\frac{w}{2}} \bp_{i,1} \, \psi_{i}(x,\bar{\bd}), \quad x\in [0,d_{0,1}],
\end{equation}
with local parameter vector
\begin{equation}
\bar{\bd} = \left(d_{-\frac{w}{2}+1,1},\dots,d_{0,1},\dots,d_{\frac{w}{2}-1,1}\right).
\end{equation}

The patch boundary with endpoints $\bp_{0,1}$ and $\bp_{1,1}$ is described by \eqref{eq:S} for
$u\in[0,1]$, $v=1$.
At any such point $(u,v)$, the fundamental functions $\psi_j$ of the class \class, $j = -w/2+1,\dots,w/2-1$, have the values
$\psi_1(y,\be)=1$ and $\psi_j(y,\be)=0$, for any $j\neq 1$, and, by \eqref{eq:xy},
$x = u d_{0,1}$. 
Substituting in  \eqref{eq:S} we get
\begin{equation}
\bS(u,1) = \sum_{i=-\frac{w}{2}+1}^{\frac{w}{2}}\bp_{i,1}\psi_i(x,\bd).
\end{equation}
Hence, the statement follows by observing that
the local parameter vectors $\bd$ and $\bar{\bd}$ are equal and thus the curve segments represented by the above formula and by \eqref{eq:F_1} are identical.
\end{proof}


We shall now study the continuity of an augmented composite surface.
Preliminarily, we observe that each surface patch is defined as a composition of infinitely differentiable functions (we assume that the edge parameter intervals are nonzero, when the mesh vertices are distinct).
Moreover, Proposition \ref{prop:C0} guarantees the continuity of the constructed surface and entails that two neighboring augmented patches have $C^k$ continuous derivatives in the direction of their common boundary.
The following proposition shows that the same smoothness holds in the cross-boundary direction and in particular allows us to conclude that the constructed surface is globally $G^k$ continuous \footnote{$G^k$ continuity refers to agreement of derivatives agree after suitable reparametrization \cite{PetersHandbook02}.}. 
For ease of notation, we formulate the statement for the two patches $\bS$ and $\tilde{\bS}$ represented in Figure \ref{fig:grid}. It is immediate to see that the result holds when considering any two neighboring patches and their common boundary, provided appropriate adjustment of notation.

\begin{prop}\label{prop:Gamma}
Let $\bS$ and $\tilde{\bS}$ be the two adjacent augmented surface patches depicted in Figure \ref{fig:grid}, based on fundamental functions of class \class.
Then the derivatives across their common boundary satisfy the relation
\begin{equation}\label{eq:cross-der}
\dr{u}\bS(u,v)|_{(0,v)} = \Delta^r(v) \dr{u}\tilde{\bS}(u,v)|_{(1,v)}, \qquad r=0,\dots,k,
\end{equation}
where
\begin{equation}\label{eq:Gammav}
\Delta(v) = \frac{\delta_{0,0}(v)}{\delta_{-1,0}(v)},
\end{equation}
and $\delta_{0,0}$ and $\delta_{-1,0}$ are the local parametrization functions defined in \eqref{eq:dh}.
\end{prop}

The above result can be shown by direct verification, and we postpone the proof to the end of this section. For now it is important to observe that relation \eqref{eq:Gammav} provides the scaling $\Delta(v)$ that relates the cross-boundary derivatives of $\bS$ and $\tilde{\bS}$. In general, $\Delta(v)$  is different at each boundary point, but varies smoothly along the boundary, being defined as the ratio of two positive polynomials.
As a consequence of Proposition \ref{prop:Gamma} we obtain the following result.


\begin{prop}\label{prop:Ck}
Two augmented surface patches built upon a class of fundamental functions \class join along their common boundary with $G^k$-continuity.
\end{prop}

\begin{proof}
Without loss of generality, we can take the two patches  $\bS$ and $\tilde{\bS}$ in Proposition \ref{prop:Gamma} and consider the map
\begin{equation}
\vect{\rho}(u,v) =
\begin{pmatrix}
\Delta(v) u+1\\
v
\end{pmatrix},
\end{equation}
where $\Delta(v)$ is given by \eqref{eq:Gammav}. 
The above function $\vect{\rho}$ is a $C^k$ reparametrization between the two domains of $\bS$ and $\tilde{\bS}$.
Moreover, exploiting relation \eqref{eq:cross-der}, it is immediate to verify that
\begin{equation}
\dr{u}\bS(u,v)|_{(0,v)} = \dr{u}\left(\tilde{\bS}\circ \vect{\rho}\right)(u,v)|_{(0,v)}, \qquad r=0,\dots,k.
\end{equation}
The above observations entail that the two considered patches join along their common boundary with $G^k$-continuity.
\end{proof}

To conclude this section we provide a proof of Proposition \ref{prop:Gamma}.
Preliminarily, the following Lemma is stated as an independent result, since it will be later recalled in Section \ref{sec:augmented_extraordinary}.

\begin{lem}\label{prop:der_scaling}
Let $\bS$ be an augmented surface patch of the form \eqref{eq:S}. Then its cross-boundary derivatives satisfy the following relations
\begin{equation}\label{eq:der_scaling}
\begin{aligned}
\dr{u} \bS(u,v) |_{(\bar{u},v)}&= \dr{x} \bS\Big(x,y,\bd,\be\Big)|_{(\bar{u},v)} \delta_{0,0}^r(v),
\,&\bar{u}=0,1,\\[1ex]
\dr{v} \bS(u,v)|_{(u,\bar{v})} &= \dr{y} \bS\Big(x,y,\bd,\be\Big)|_{(u,\bar{v})} \epsilon_{0,0}^r(u),
\, &\bar{v}=0,1,
\end{aligned}
\end{equation}
where $\delta_{0,0}$ and $\epsilon_{0,0}$ are the local parametrization functions defined in \eqref{eq:dh}--\eqref{eq:eh}.
\end{lem}

\begin{proof}
The result immediately follows by using the chain rule, the relation \eqref{eq:xy}
and the fact that
$\delta_{i,0}^{(r)}(v)=\epsilon_{0,j}^{(r)}(u)=0$, $r=1,\dots,k$, $u,v=0,1$ and $i,j=-\frac{w}{2}+1,\dots,\frac{w}{2}-1$.
\end{proof}

We observe that the above Lemma is based on the fact that the local parametrization functions \eqref{eq:dh}--\eqref{eq:eh} have vanishing derivatives at the endpoints of their interval of definition.
As we will see shortly, this property also plays a prominent role in proving Proposition \ref{prop:Gamma}.
An interesting consequence is that, in order to prescribe a correct set of local parametrization functions, it is not sufficient to require that these functions interpolate the edge parameter intervals. In fact, although this strategy may still allow us to interpolate a network of independently parameterized curves, without the aforementioned constraint on the derivatives, we would not be able to ensure that the composite surface is globally $G^k$ continuous.

\begin{proof}[Proof of Proposition \ref{prop:Gamma}]
Without loss of generality, we can assume that the local reference system $uv$ of the two patches is oriented as illustrated in Figure \ref{fig:grid}. Therefore, the common boundary of $\bS$ and $\tilde{\bS}$ corresponds to $(0,v)$ for $\bS$ and $(1,v)$ for $\tilde{\bS}$. We also refer to the figure for the labeling of all the relevant quantities involved.
To prove the statement, we shall verify that \eqref{eq:cross-der}--\eqref{eq:Gammav}
hold for any arbitrary $v\in [0,1]$.

By differentiating formulae \eqref{eq:S}--\eqref{eq:biv_fund} in the cross-boundary direction $u$, we obtain
\begin{equation}\label{eq:derS_r}
\dr{u}\bS(u,v) =
\sum_{i=-\frac{w}{2}+1}^{\frac{w}{2}} \sum_{j=-\frac{w}{2}+1}^{\frac{w}{2}} \bp_{i,j}
\sum_{q=0}^r \binom{r}{q}
\, \dq{u}\psi_i\Big(u \delta_{0,0}(v),\bd(v)\Big)
\, \drq{u} \psi_j\Big(v \epsilon_{0,0}(u),\be(u)\Big) ,
\end{equation}
where
\begin{equation}\label{eq:beS}
\bd(v)=\left(\delta_{-\frac{w}{2}+1,0}(v),\dots,\delta_{0,0}(u),\dots,\delta_{\frac{w}{2}-1,0}(v)\right)\qquad \text{and} \qquad \be(u)=\left(\epsilon_{0,-\frac{w}{2}+1}(u),\dots,\epsilon_{0,0}(u),\dots,\epsilon_{0,\frac{w}{2}-1}(u)\right).
\end{equation}
From the definition of the local parametrization functions in \eqref{eq:dh}--\eqref{eq:eh}, we have
$\epsilon_{h,0}'(0)=0$, $h = -\frac{w}{2}+1,\dots,\frac{w}{2}-1$ and therefore
\begin{equation}
\left.\frac{\partial}{\partial u} \psi_j\Big(v \epsilon_{0,0}(u),\be(u)\Big) \right|_{(0,v)}= \left.\sum_{h=-\frac{w}{2}+1}^{\frac{w}{2}-1} \frac{\partial \psi_j}{\partial \epsilon_{h,0}} \frac{\partial \epsilon_{h,0}}{\partial u}\right|_{(0,v)} = 0.
\end{equation}
Moreover, since $\epsilon_{h,0}^{(r)}(u)$ vanishes at $u=0$, for all $r=1,\dots,k$, by iterating the differentiation process (cf.\ Fa\`{a} di Bruno's law) it can be easily verified that
%
\begin{equation}
\left.\drq{u} \psi_j\Big(v \epsilon_{0,0}(u),\be(u)\Big) \right|_{(0,v)}= 0, \qquad q=0,\dots,r-1, \qquad r=1,\dots,k.
\end{equation}
%


In addition, recalling that the fundamental functions \class have continuity $C^k$ and support width $w$, there holds
\begin{equation}
\left.\dr{u}\psi_{\frac{w}{2}}\Big(u \delta_{0,0}(v),\bd(v)\Big)\right|_{(0,v)}=0, \qquad r=1,\dots,k.
\end{equation}

Using the last two identities above, at any boundary point equation \eqref{eq:derS_r} reduces to
\begin{equation}\stepcounter{equation}\label{eq:S_bound0}
\left.\dr{u}\bS(u,v)\right|_{(0,v)} =
\sum_{i=-\frac{w}{2}+1}^{\frac{w}{2}-1} \sum_{j=-\frac{w}{2}+1}^{\frac{w}{2}} \bp_{i,j}\,
\psi_j\Big(v e_{0,0},\be(0)\Big) \, \left.\dr{u} \psi_i\Big(u \delta_{0,0}(v),\bd(v)\Big)\right|_{(0,v)}.
\end{equation}

We now turn to considering the neighboring surface patch $\tilde{\bS}$.
Denoted $\tilde{\bd}(v)$ and $\tilde{\be}(u)$ the local parameter vectors for $\tilde{\bS}$, we have
\begin{equation}\label{eq:beStilde}
\tilde{\bd}(v)=\left(\delta_{-\frac{w}{2},0}(v),\dots,\delta_{-1,0}(u),\dots,\delta_{\frac{w}{2}-2,0}(v)\right), \qquad \tilde{\be}(u)=\left(\epsilon_{-1,-\frac{w}{2}+1}(u),\dots,\epsilon_{-1,0}(u),\dots,\epsilon_{-1,\frac{w}{2}-1}(u)\right)
\end{equation}
and thus

\begin{equation}\label{eq:dertildeS_r}
\dr{u}\tilde{\bS}(u,v) =
\sum_{i=-\frac{w}{2}}^{\frac{w}{2}-1} \sum_{j=-\frac{w}{2}+1}^{\frac{w}{2}} \bp_{i,j}
\sum_{q=0}^r {r \choose q}
\, \dq{u}\psi_i\Big(u \delta_{-1,0}(v),\tilde{\bd}(v)\Big)
\, \drq{u} \psi_j\Big(v \epsilon_{-1,0}(u),\tilde{\be}(u)\Big).
\end{equation}
As before, it can be easily verified that
\begin{equation}
\left.\drq{u} \psi_j\Big(v \epsilon_{-1,0}(u),\tilde{\be}(u)\Big)\right|_{(1,v)}=0, \qquad q=0,\dots,r-1, \qquad r=1,\dots,k.
\end{equation}

Moreover, from the compact support of the fundamental functions follows that
\begin{equation}
\left.\dr{u}\psi_{-\frac{w}{2}}\Big(u \delta_{-1,0}(v),\tilde{\bd}(v)\Big)\right|_{(1,v)}=0,\qquad r=1,\dots,k,
\end{equation}
and, observing that $\be(0)=\tilde{\be}(1)$, we obtain
\begin{equation}\label{eq:S_bound1}
\left.\dr{u}\tilde{\bS}(u,v)\right|_{(1,v)}=
\sum_{i=-\frac{w}{2}+1}^{\frac{w}{2}-1} \sum_{j=-\frac{w}{2}+1}^{\frac{w}{2}} \bp_{i,j}
\, \psi_j\Big(v e_{0,0},\be(0)\Big)\, \left.\dr{u} \psi_i\Big(u \delta_{-1,0}(v),\tilde{\bd}(v)\Big)\right|_{(1,v)}.
\end{equation}
Now, in view of \eqref{eq:S_bound0} and \eqref{eq:S_bound1}, it only remains to show that the derivatives of order $r=1,\dots,k$ of the functions $\psi_i$ agree after the scaling $\Delta(v)$ in \eqref{eq:Gammav}.

Comparing the expressions in \eqref{eq:beS} and \eqref{eq:beStilde} we can see that the two vectors $\bd(v)$ and $\tilde{\bd}(v)$ are one a \qtext{shifted} version of the other and therefore \qtext{overlap} almost everywhere, with the exception of the first element in $\tilde{\bd}$ and the last one in $\bd$ (Figure \ref{fig:prop_Gamma_a} schematizes the situation for a class of fundamental functions having support width $w=4$).
\begin{figure}\label{fig:prop_Gamma}
\centering
\subfigure[]{\includegraphics[width=0.4\textwidth]{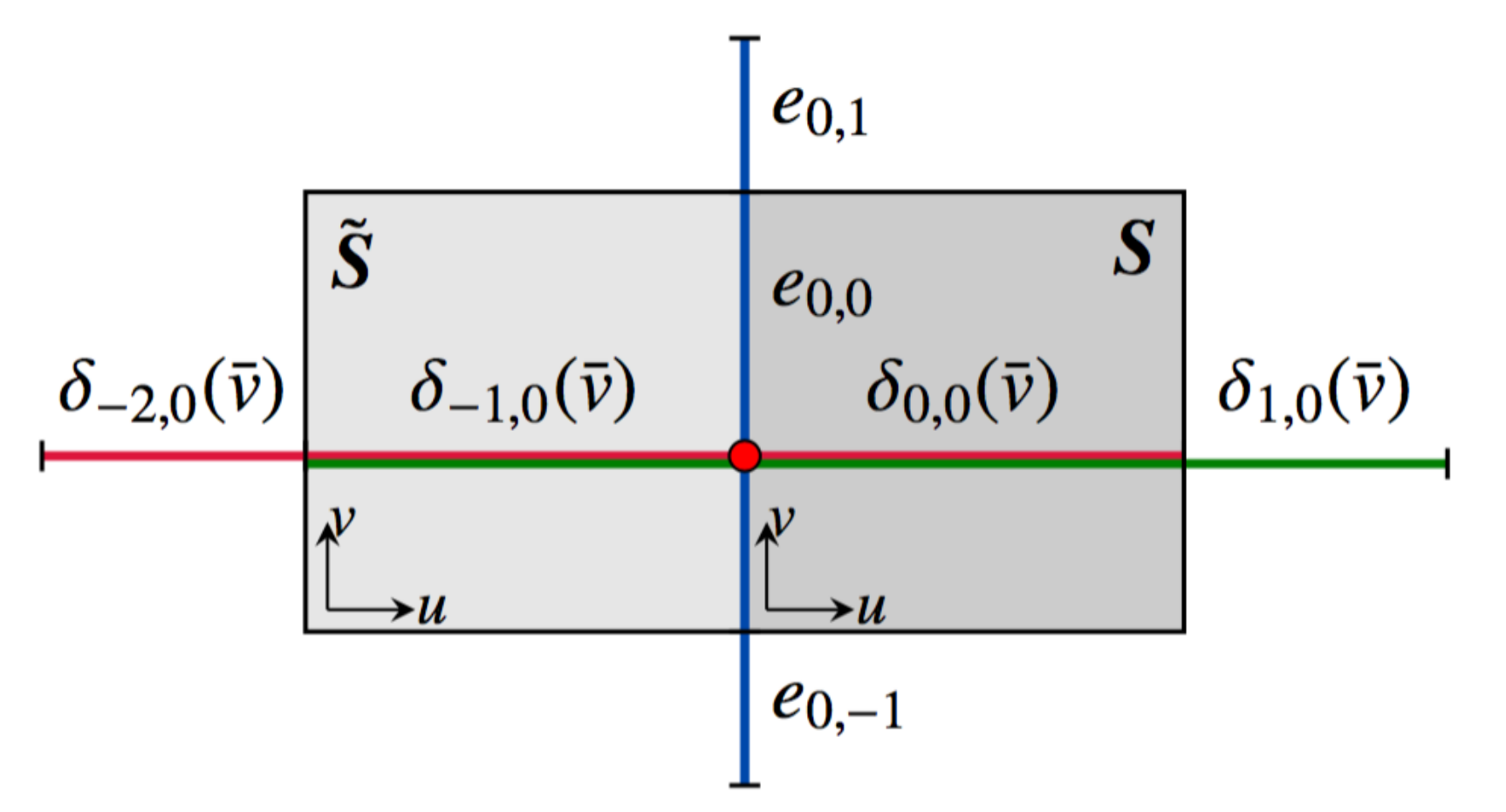}\label{fig:prop_Gamma_a}}
\hspace{1cm}
\subfigure[]{\includegraphics[width=0.45\textwidth]{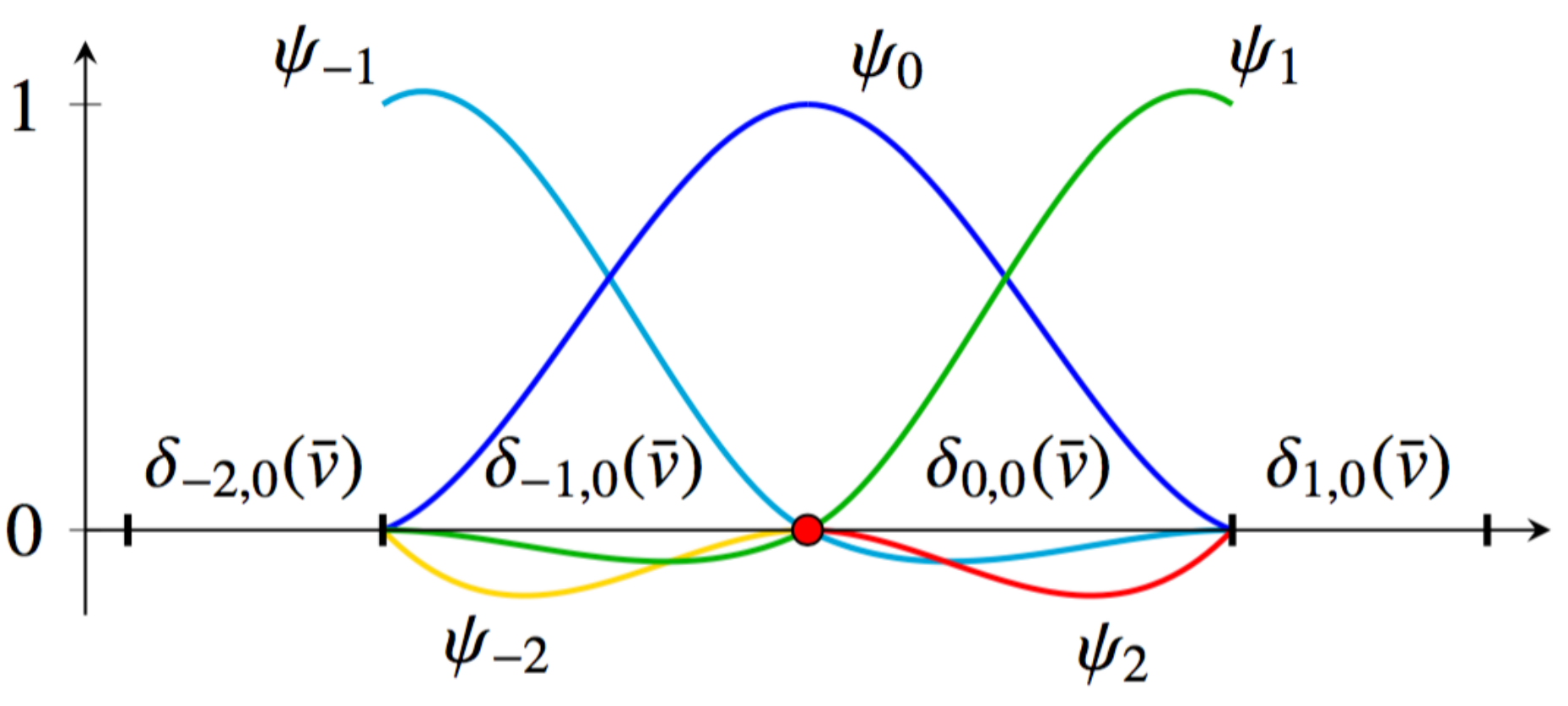}\label{fig:prop_Gamma_b}}
\caption{A class of fundamental functions with $w=4$.
\subref{fig:prop_Gamma_a} Local parameter vectors at the boundary point $(0,\bar{v})$ and $(1,\bar{v})$ respectively for the patches $\bS$ and $\tilde{\bS}$ involved in the proof of Proposition \ref{prop:Gamma}; \subref{fig:prop_Gamma_b} Fundamental functions defined on the corresponding parameter intervals in the cross-boundary direction.}
\end{figure}
These two different intervals are uninfluential to the value and derivatives of the
fundamental functions at a boundary point (see Figure \ref{fig:prop_Gamma_b}).
As a consequence, at such a point, the fundamental functions defined on $\bd(v)$ and $\tilde{\bd}(v)$ agree together with their derivatives up to order $k$, namely
\begin{equation}
\left.\dr{x}\psi_i(x,\bd(v))\right|_{x=0} = \left.\dr{x}\psi_i(x,\tilde{\bd}(v))\right|_{x=\delta_{-1,0}(v)}, \qquad i=-\frac{w}{2}-1,\dots,\frac{w}{2}+1, \qquad r=0,\dots,k.
\end{equation}
The last relation and \eqref{eq:der_scaling} entail that
\begin{equation}\label{eq:Gk}
\left.\dr{u}\bS(u,v)\right|_{(0,v)} = \left( \frac{\delta_{0,0}(v)}{\delta_{-1,0}(v)}\right)^r \left.\dr{u}\tilde{\bS}(u,v)\right|_{(1,v)},
\end{equation}
which concludes the proof.
\end{proof}
\subsection{Examples}
\label{sec:examples_regular}
\begin{figure}[t]
\centering
\subfigure[Input mesh]
{\includegraphics[width=0.95\textwidth/4]{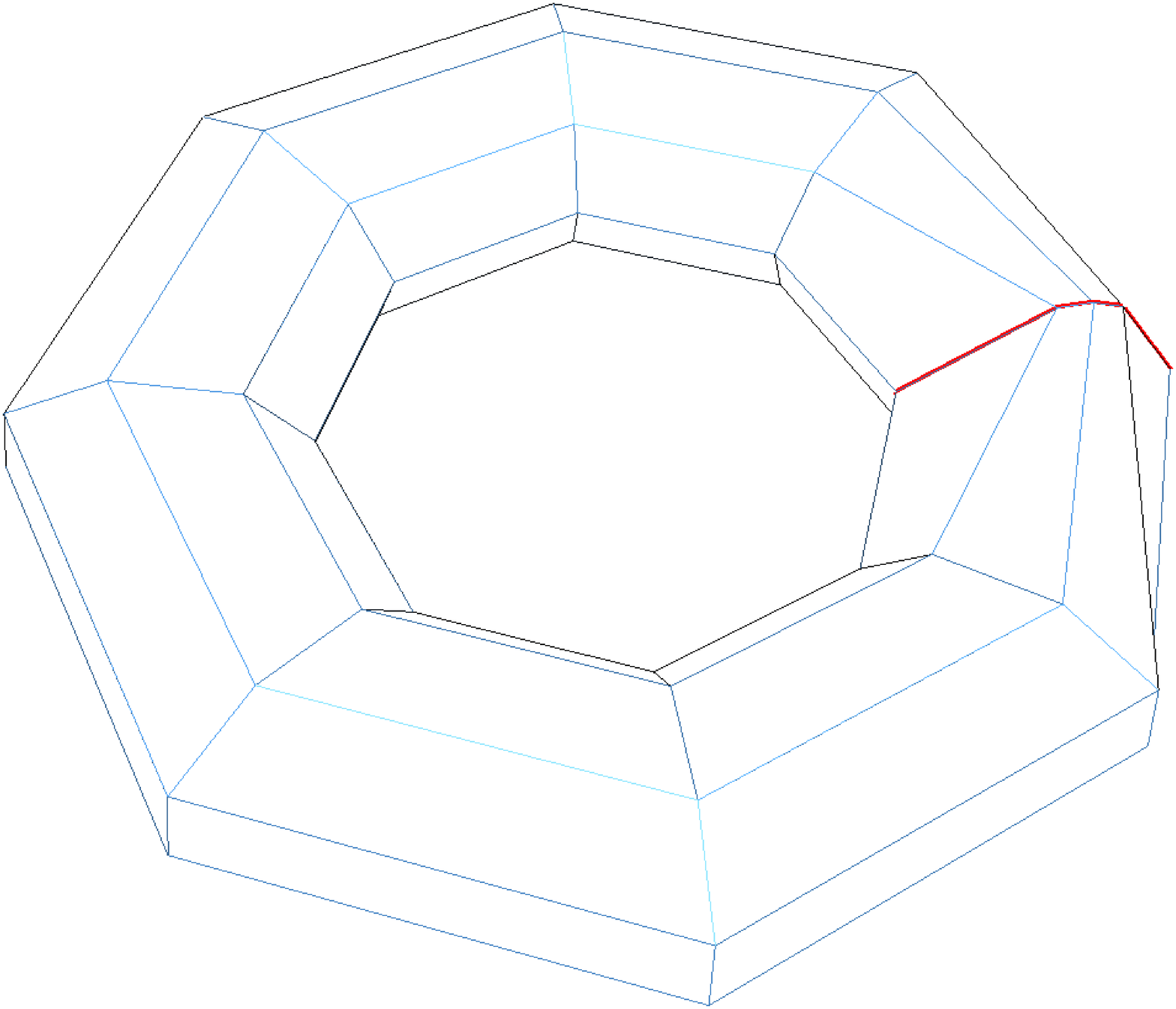}\label{fig:ex1_1}}
\hfill
\subfigure[Tensor product surface]
{\includegraphics[width=0.95\textwidth/4]{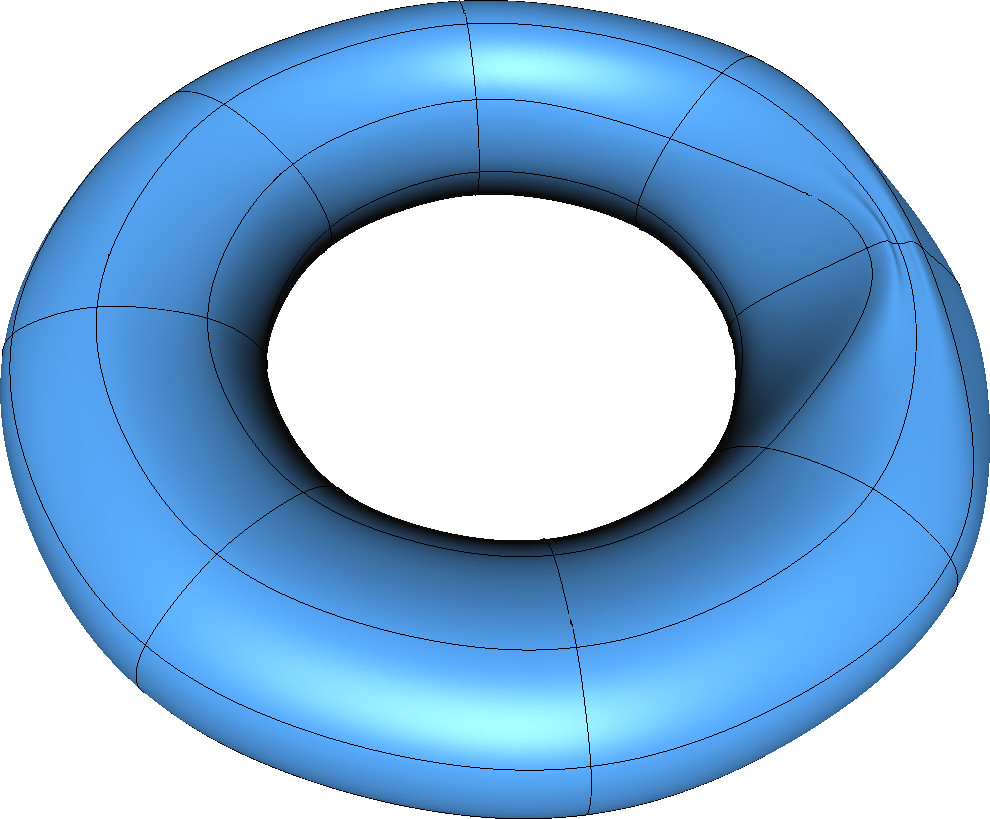}\label{fig:ex1_2}}
\hfill
\subfigure[Zoom of \ref{fig:ex1_2}]
{\includegraphics[width=0.8\textwidth/4]{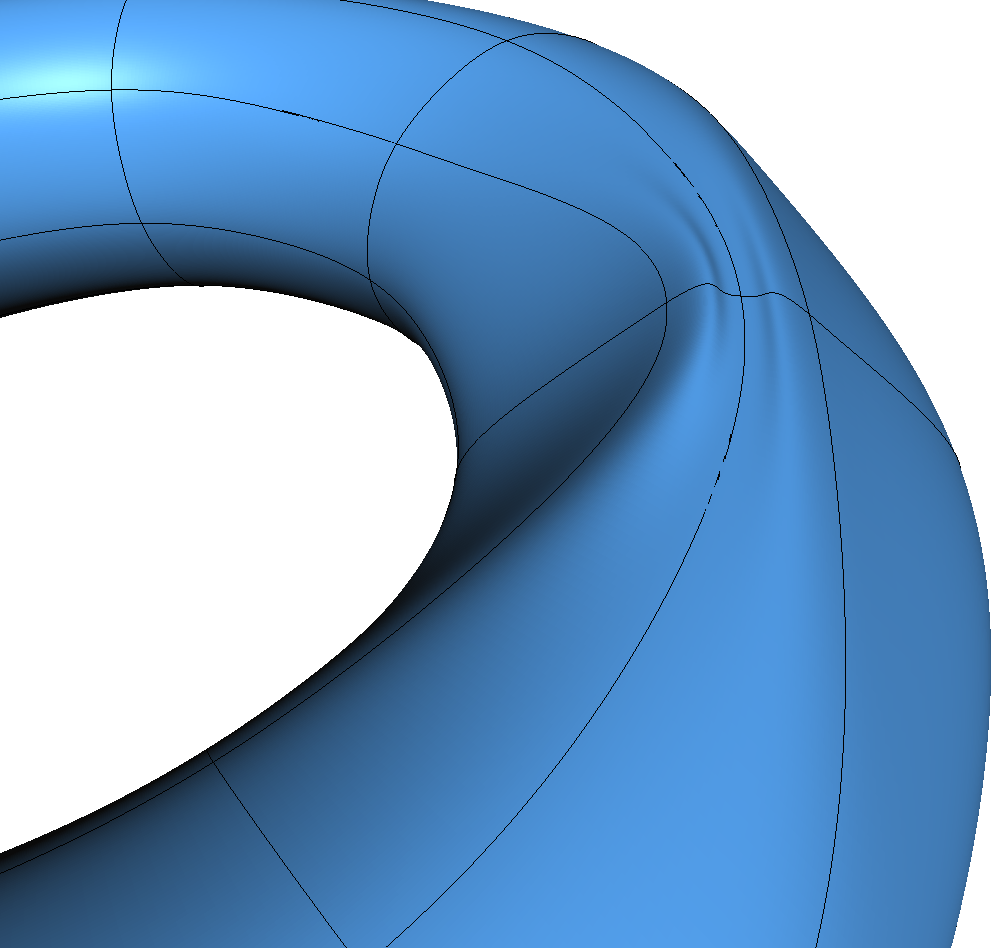}\label{fig:ex1_3}}
\hfill
\subfigure[Mean curvature of \ref{fig:ex1_2}]
{\includegraphics[width=0.8\textwidth/4]{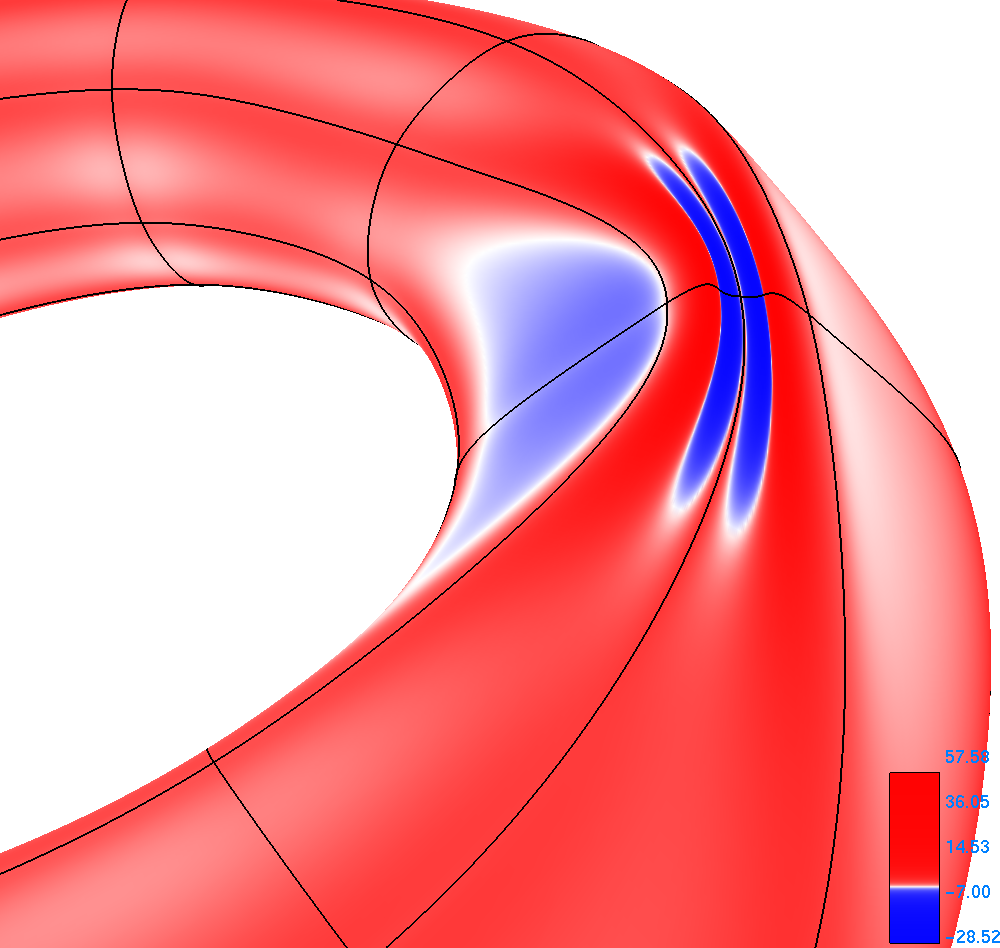}\label{fig:ex1_4}}
\\
\subfigure[Section polyline and corresponding section curves]
{\hspace{0.0cm}\includegraphics[width=0.95\textwidth/4]{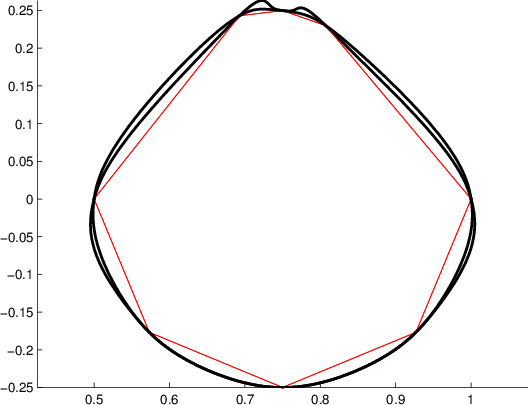}\label{fig:ex1_5}}
\hfill
\subfigure[Augmented surface]
{\hspace{0.0cm}\includegraphics[width=0.95\textwidth/4]{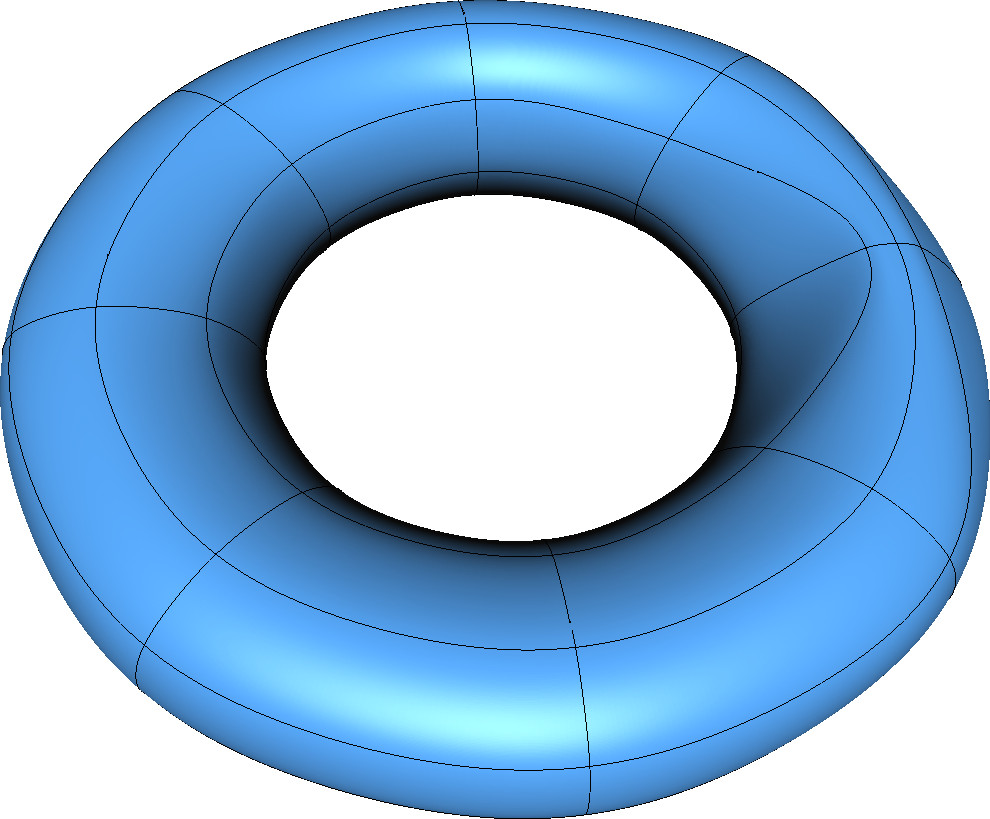}\label{fig:ex1_6}}
\hfill
\subfigure[Zoom of \ref{fig:ex1_6}]
{\includegraphics[width=0.8\textwidth/4]{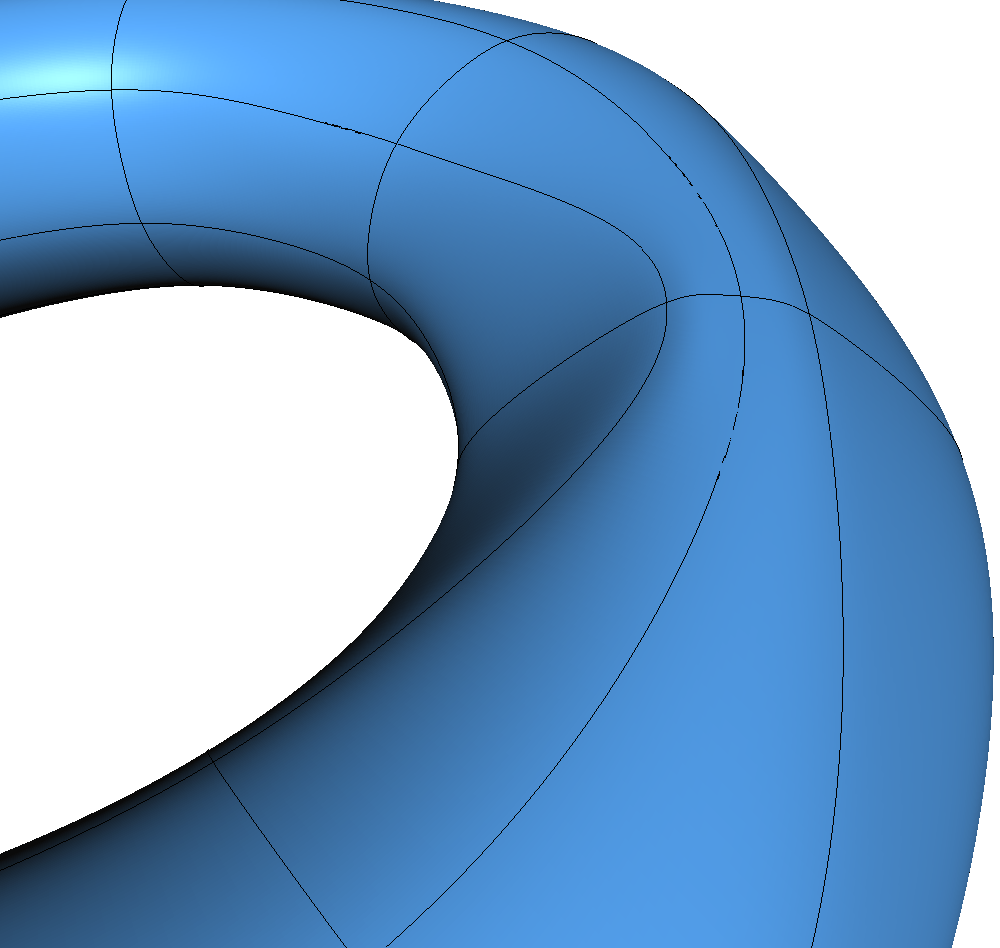}\label{fig:ex1_7}}
\hfill
\subfigure[Mean curvature of \ref{fig:ex1_6}]
{\includegraphics[width=0.8\textwidth/4]{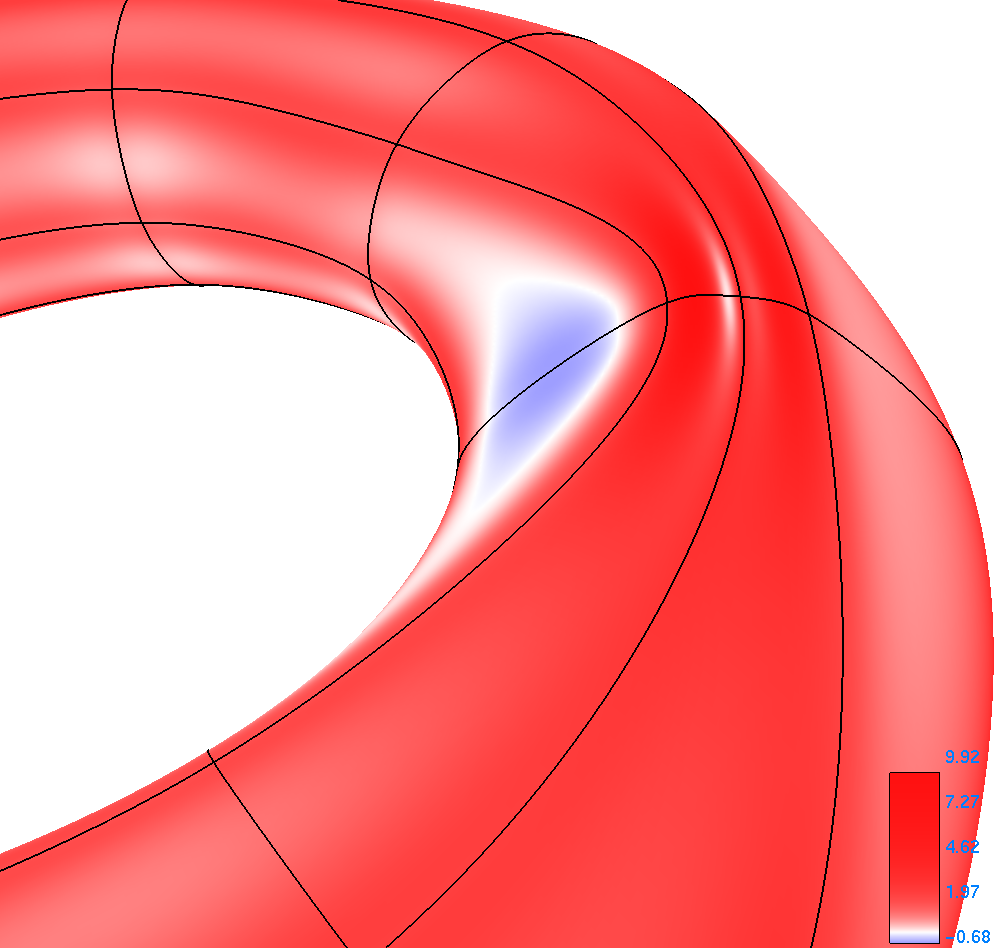}\label{fig:ex1_8}}
\caption{Comparison between tensor product and augmented surfaces built upon univariate splines of class $D^5C^2P^2S^4$. Figure \ref{fig:ex1_5} shows the interpolatory curves of the considered class generated from the red-colored section polyline in \ref{fig:ex1_1} with augmented and tensor product parameterizations.}\label{fig:ex1}
\end{figure}

\begin{figure}[t]
\centering
\subfigure[Initial mesh]
{\includegraphics[width=0.9\textwidth/4,trim=0.5cm 1cm 0.5cm 0cm, clip=true]{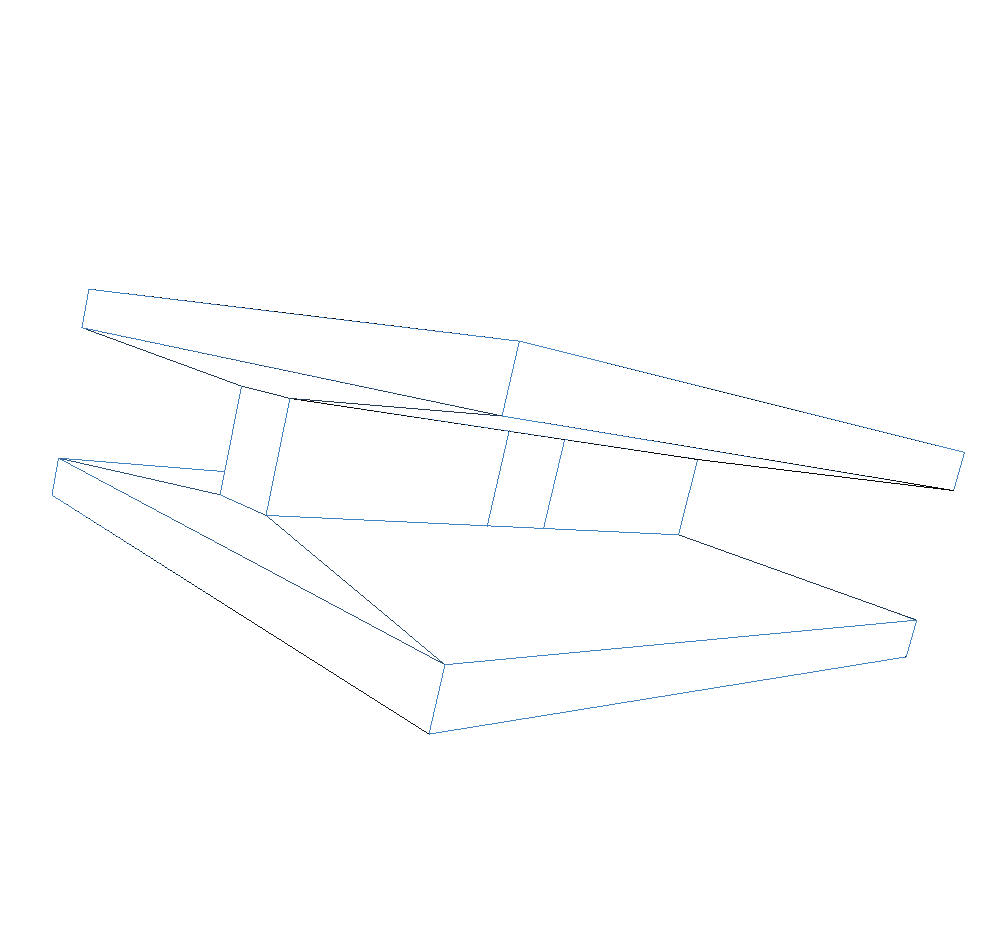}
\includegraphics[width=0.9\textwidth/4,trim=0.5cm 1cm 0.5cm 0cm, clip=true]{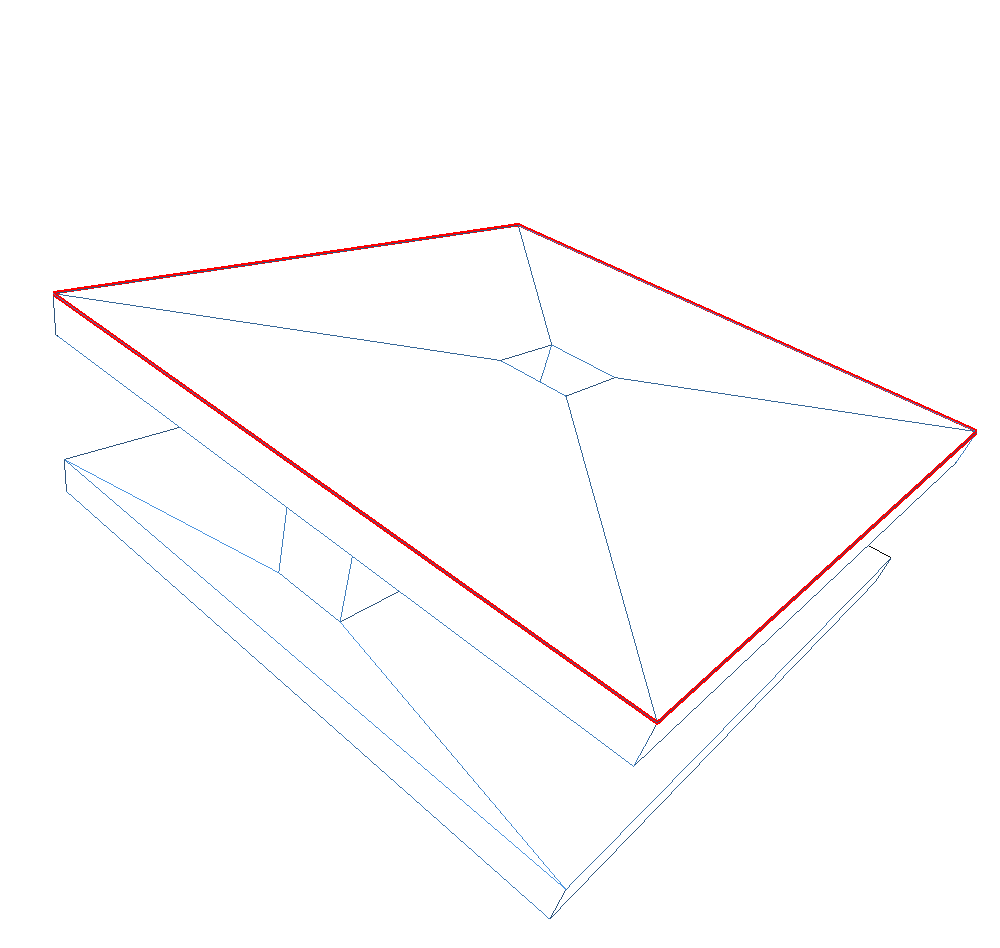}
\label{fig:ex2b_mesh}}
\hspace{0.5cm}
\subfigure[Section polyline and corresponding section curves]{
\includegraphics[width=0.9\textwidth/4]{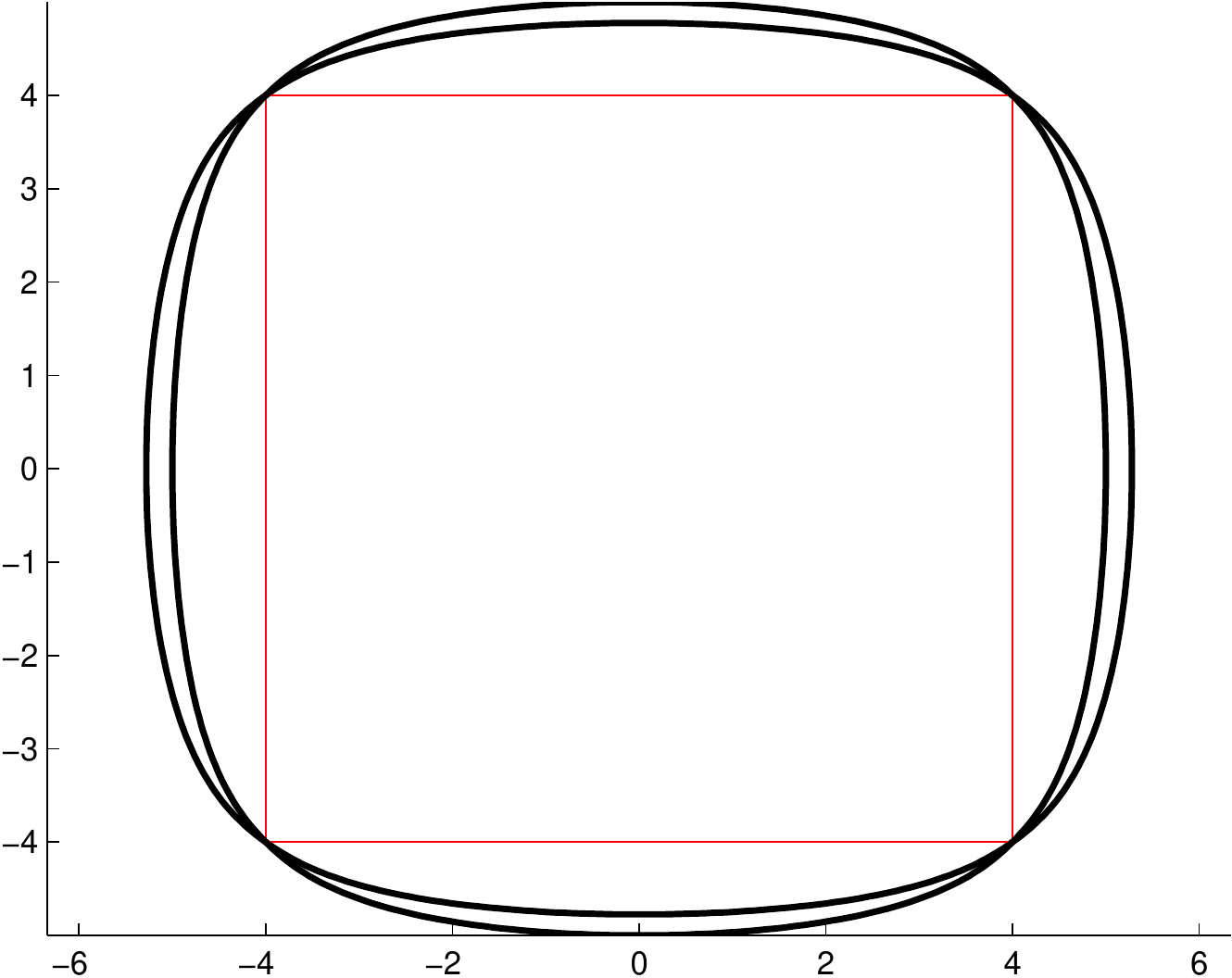}
\includegraphics[width=0.9\textwidth/4]{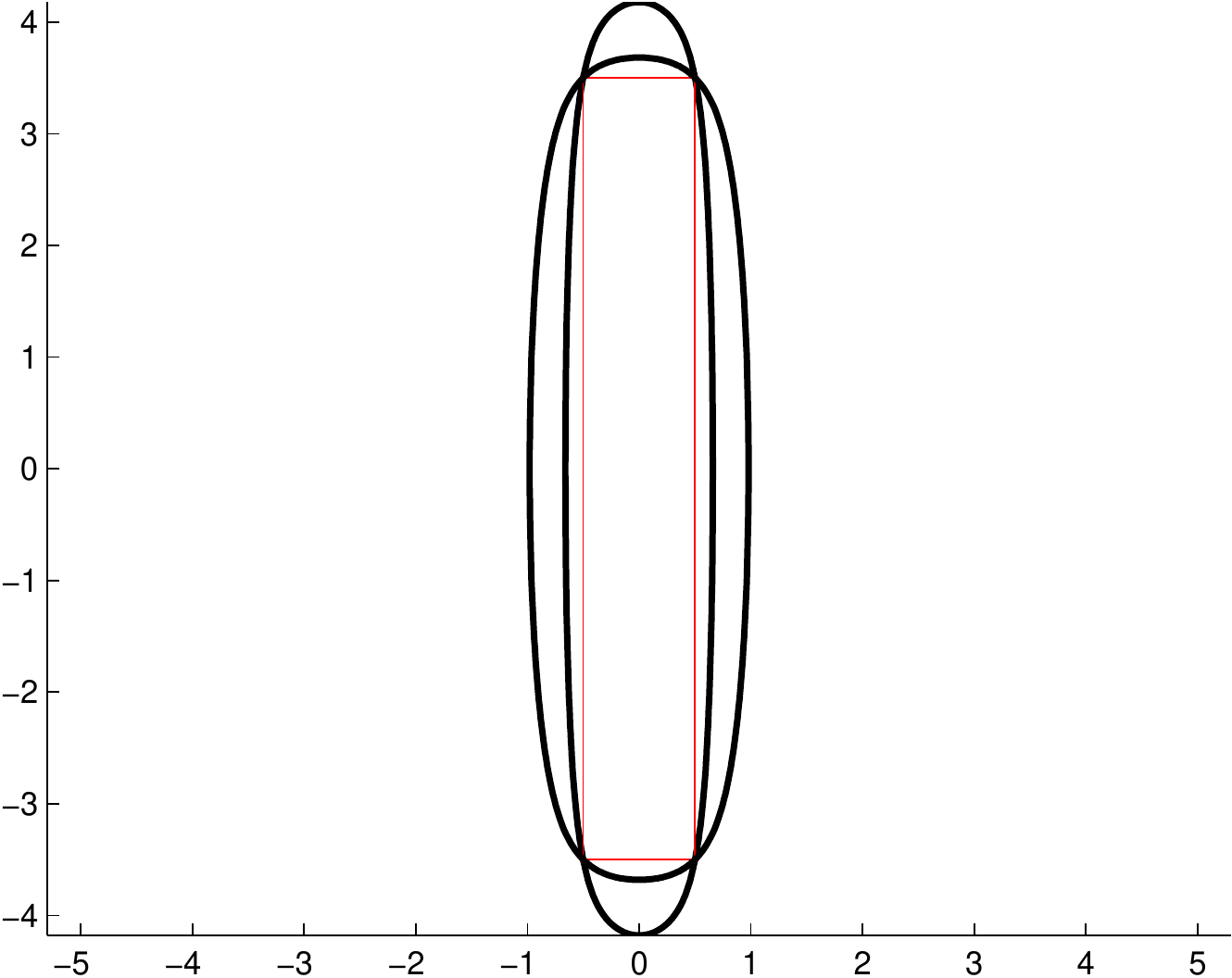}
\label{fig:ex2b_section}
}
\\
\subfigure[Augmented surface]{
\includegraphics[width=0.9\textwidth/4,trim=0.5cm 1cm 0.5cm 0cm, clip=true]{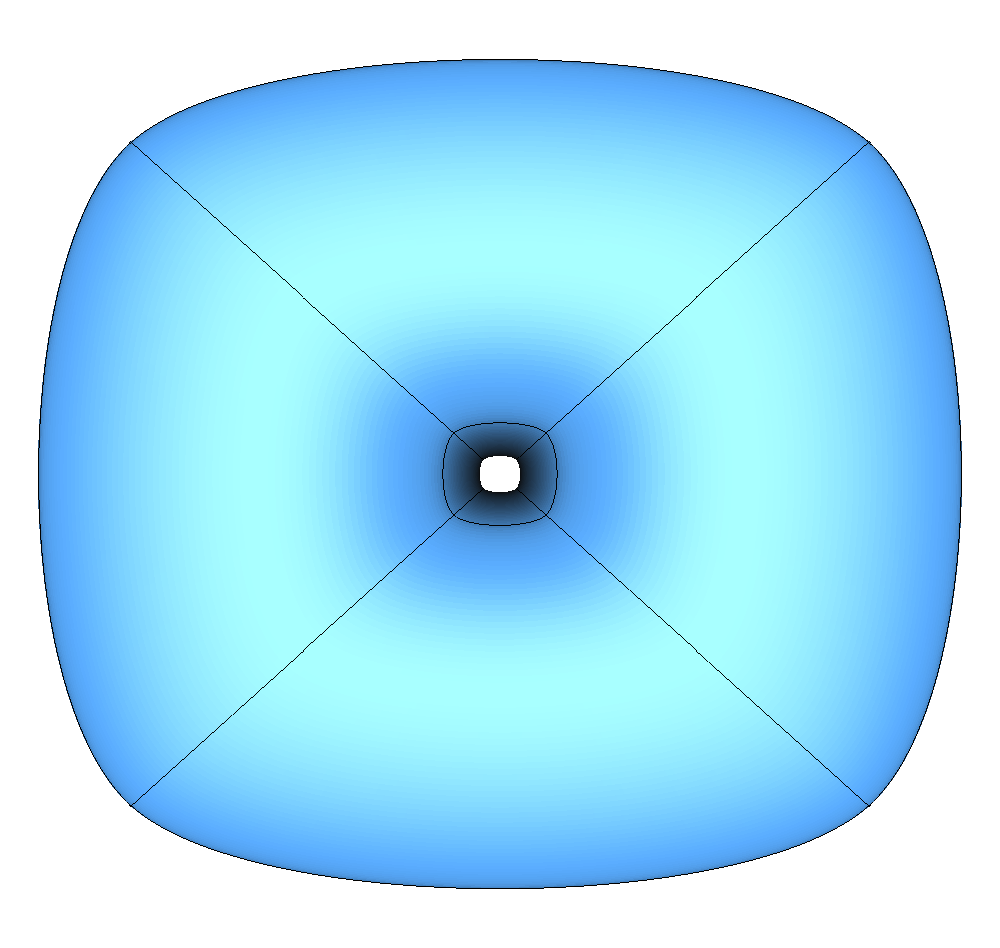}
\hfill
\includegraphics[width=0.9\textwidth/4,trim=0.5cm 0cm 0.5cm 0cm, clip=true]{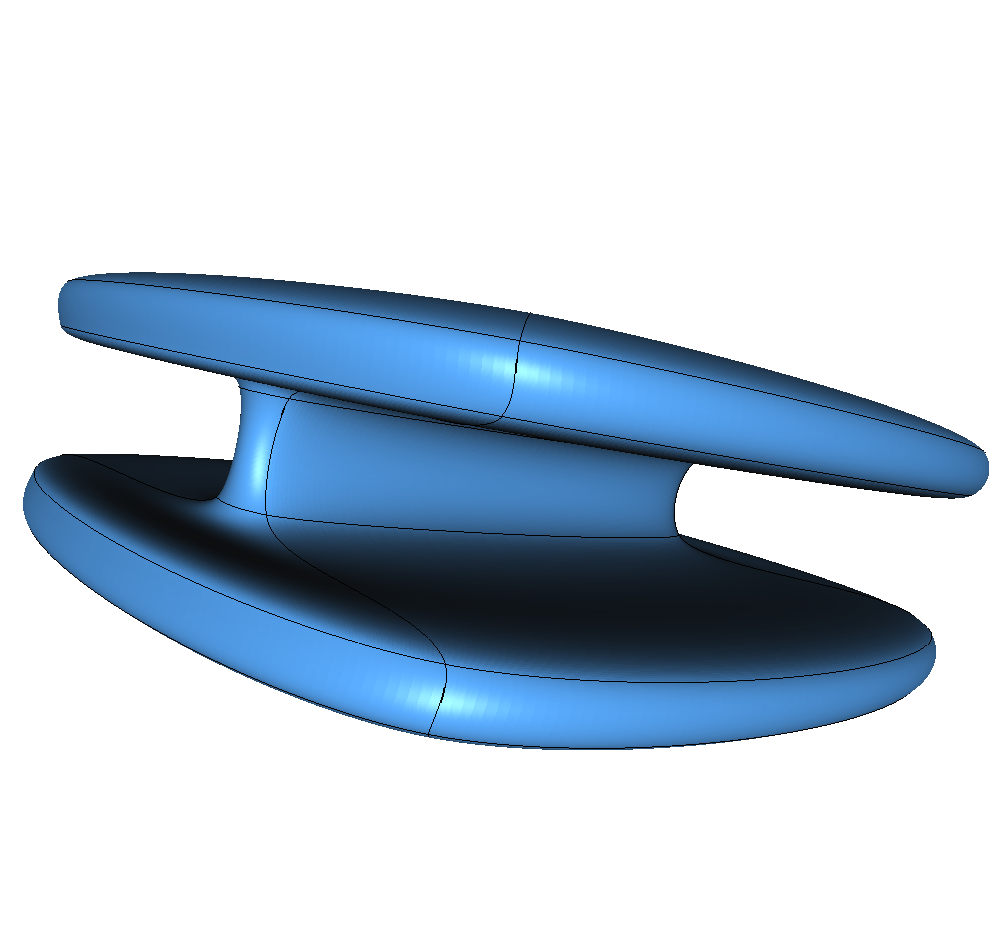}\label{fig:ex2b_1}
}
\hspace{0.5cm}
\subfigure[Tensor product surface]{
\includegraphics[width=0.9\textwidth/4,trim=0.5cm 0cm 0.5cm 0cm, clip=true]{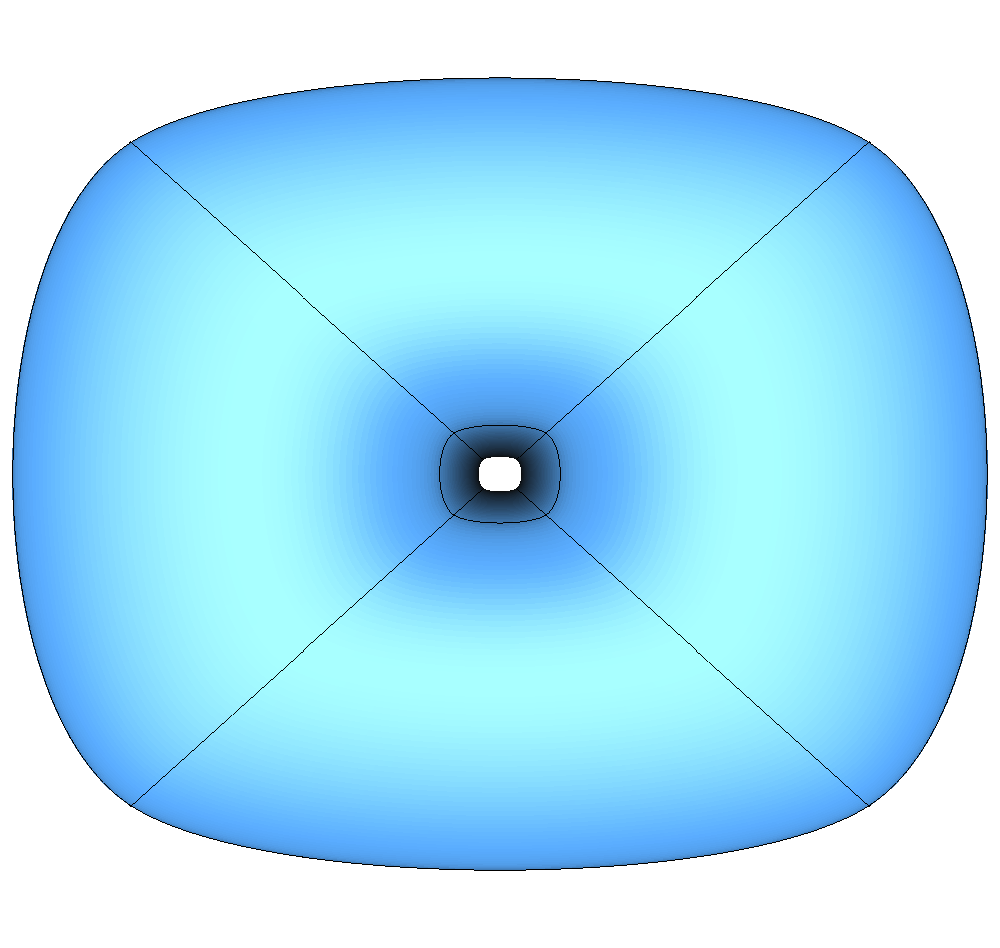}
\hfill
\includegraphics[width=0.9\textwidth/4,trim=0.0cm 0cm 0.5cm 0cm, clip=true]{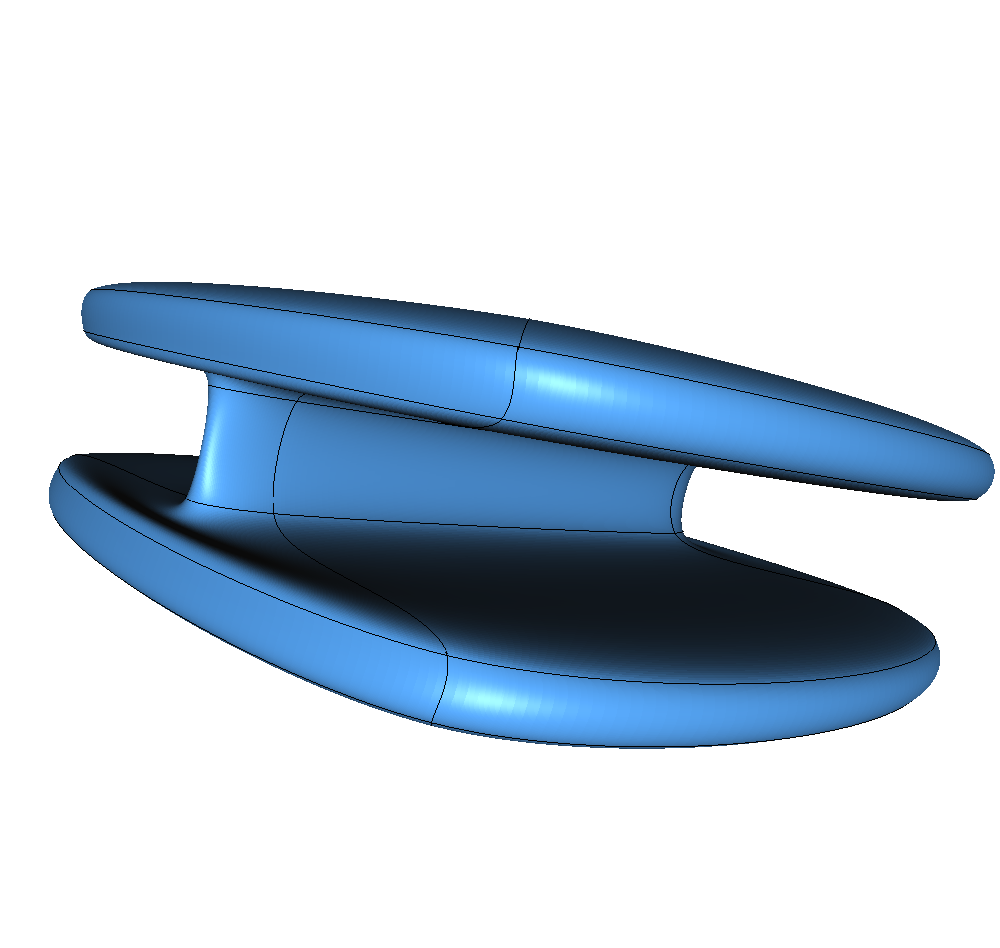}\label{fig:ex2b_2}}
\caption{Comparison between augmented and tensor product surface. \subref{fig:ex2b_mesh} Input mesh; \subref{fig:ex2b_section} The section polyline highlighted in \subref{fig:ex2b_mesh} and section curves of class $D^5C^2P^2S^4$, with the augmented (centripetal) and mean parametrizations. \subref{fig:ex2b_1}-\subref{fig:ex2b_2} Corresponding augmented and tensor product surfaces.}\label{fig:ex2b}
\end{figure}

\begin{figure}[t]
\centering
\renewcommand*{\thesubfigure}{}
\subfigure[Initial mesh]
{\includegraphics[width=0.9\textwidth/4, trim=0cm 2cm 0cm 2cm, clip=true]{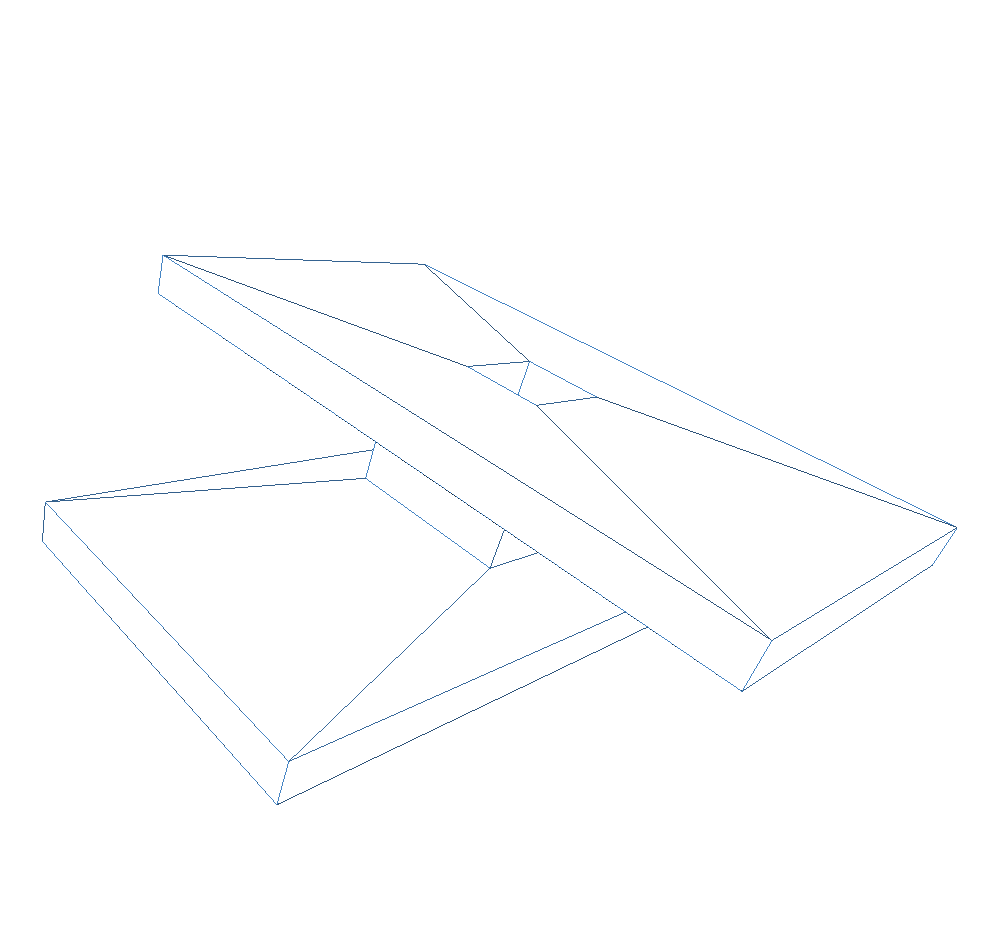}
\hspace{0.5cm}
\includegraphics[width=0.9\textwidth/4,trim=0cm 2cm 0cm 2cm, clip=true]{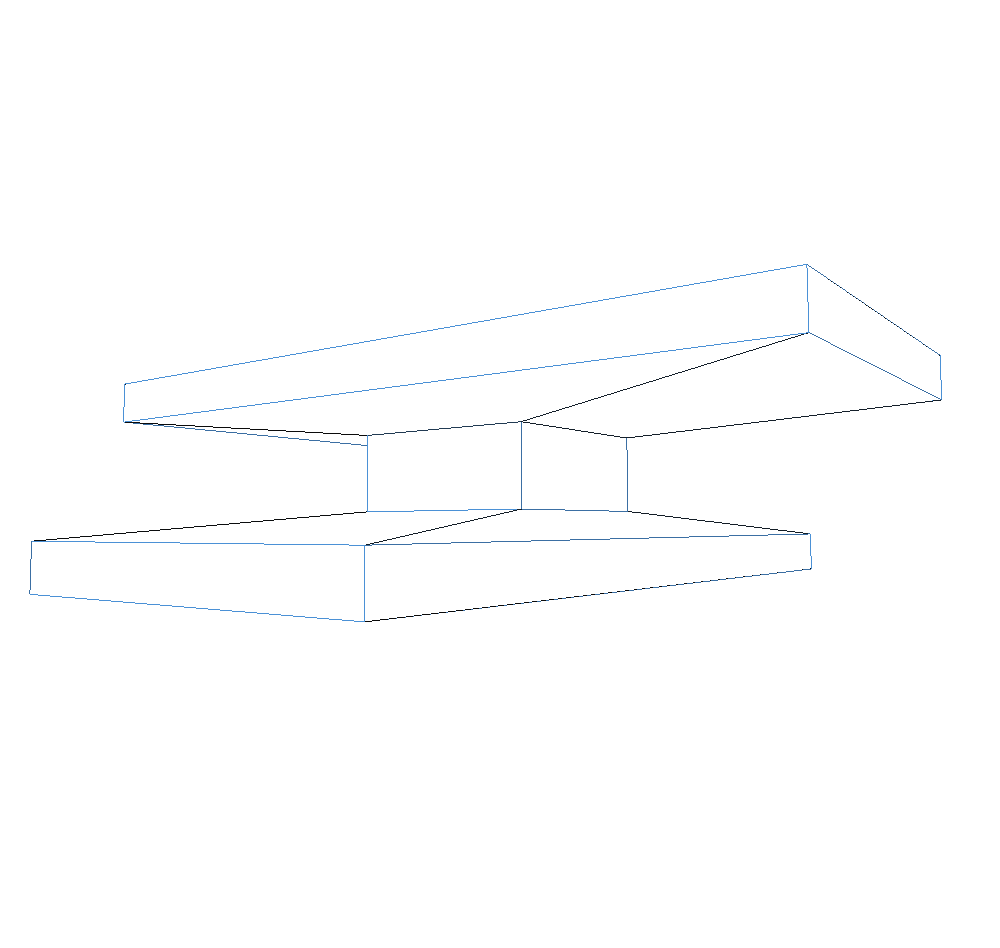}\label{fig:ex2_mesh}}
\\
\subfigure[$D^3C^1P^2S^4$]
{\includegraphics[width=0.9\textwidth/4,trim=0cm 2cm 0cm 2cm, clip=true]{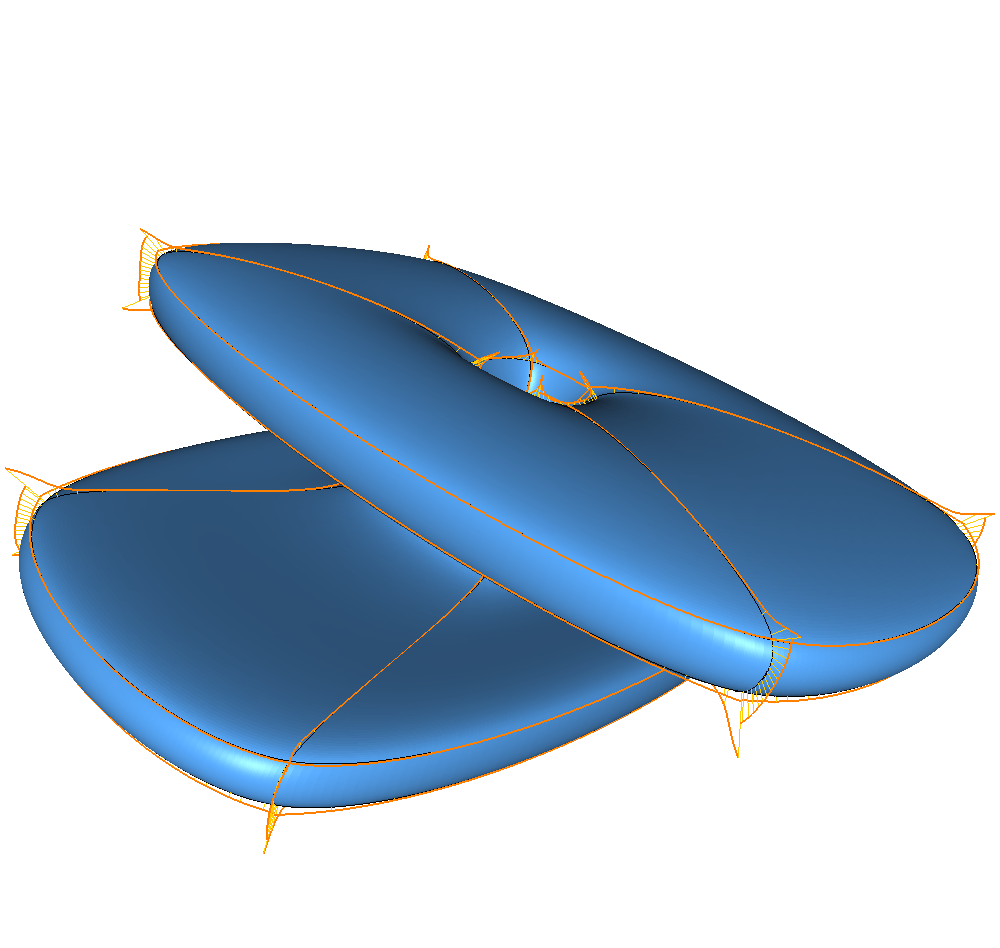}\label{fig:ex2_1}}
\hfill
\subfigure[$D^5C^2P^2S^4$]
{\includegraphics[width=0.9\textwidth/4,trim=0cm 2cm 0cm 2cm, clip=true]{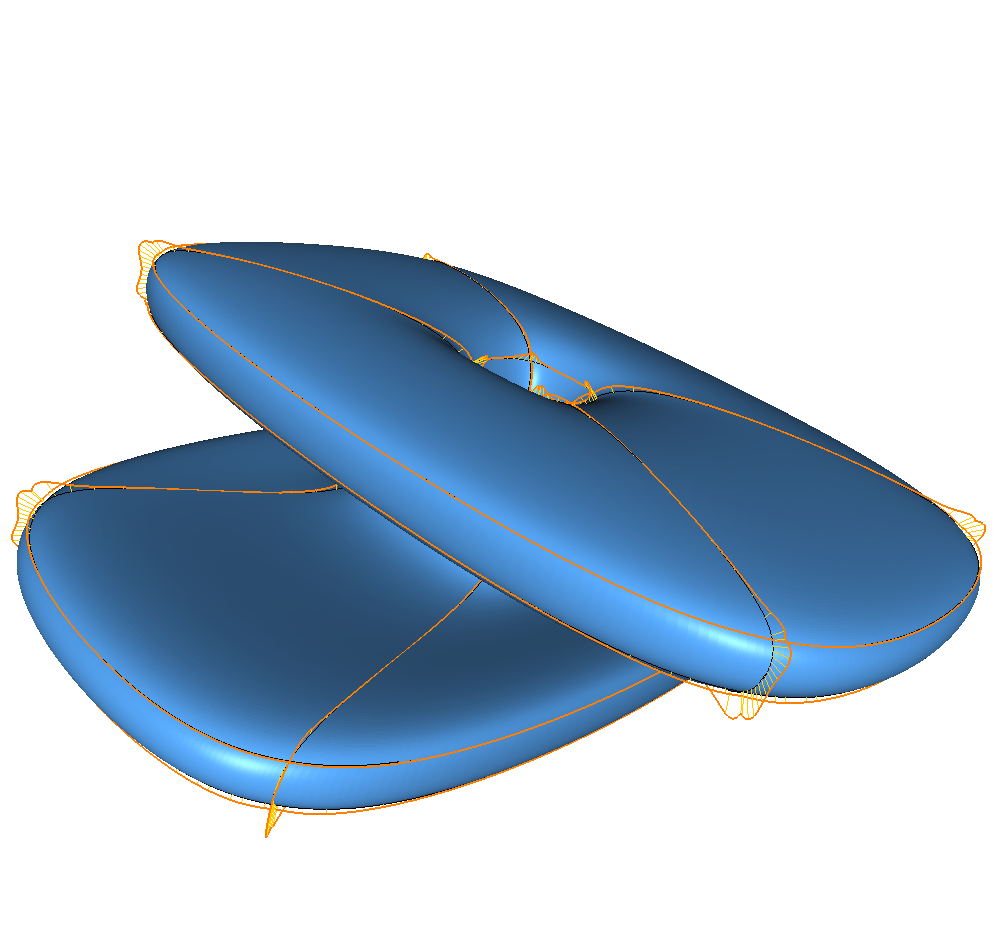}\label{fig:ex2_2}}
\hfill
\subfigure[$D^4C^2P^3S^6$]
{\includegraphics[width=0.9\textwidth/4,trim=0cm 2cm 0cm 2cm, clip=true]{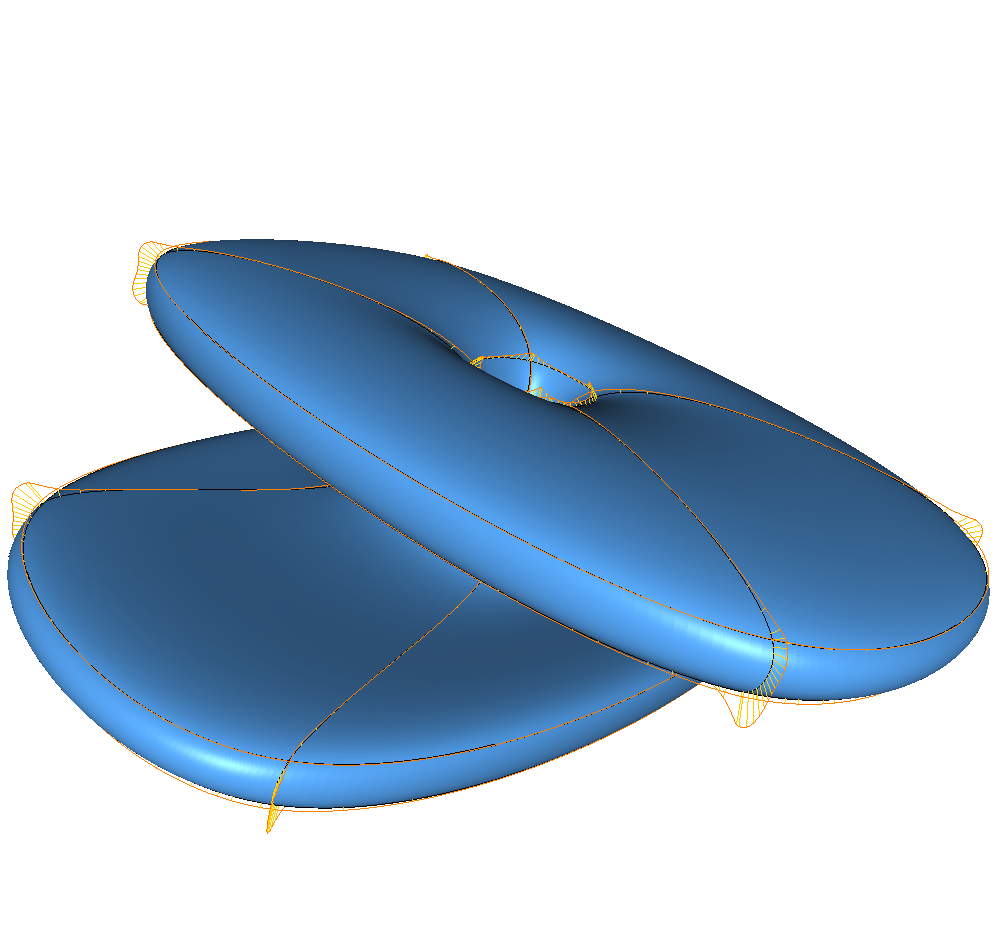}\label{fig:ex2_3}}
\hfill
\subfigure[$D^6C^3P^3S^6$]
{\raisebox{-0mm}{\includegraphics[width=0.9\textwidth/4,trim=0cm 2cm 0cm 2cm, clip=true]{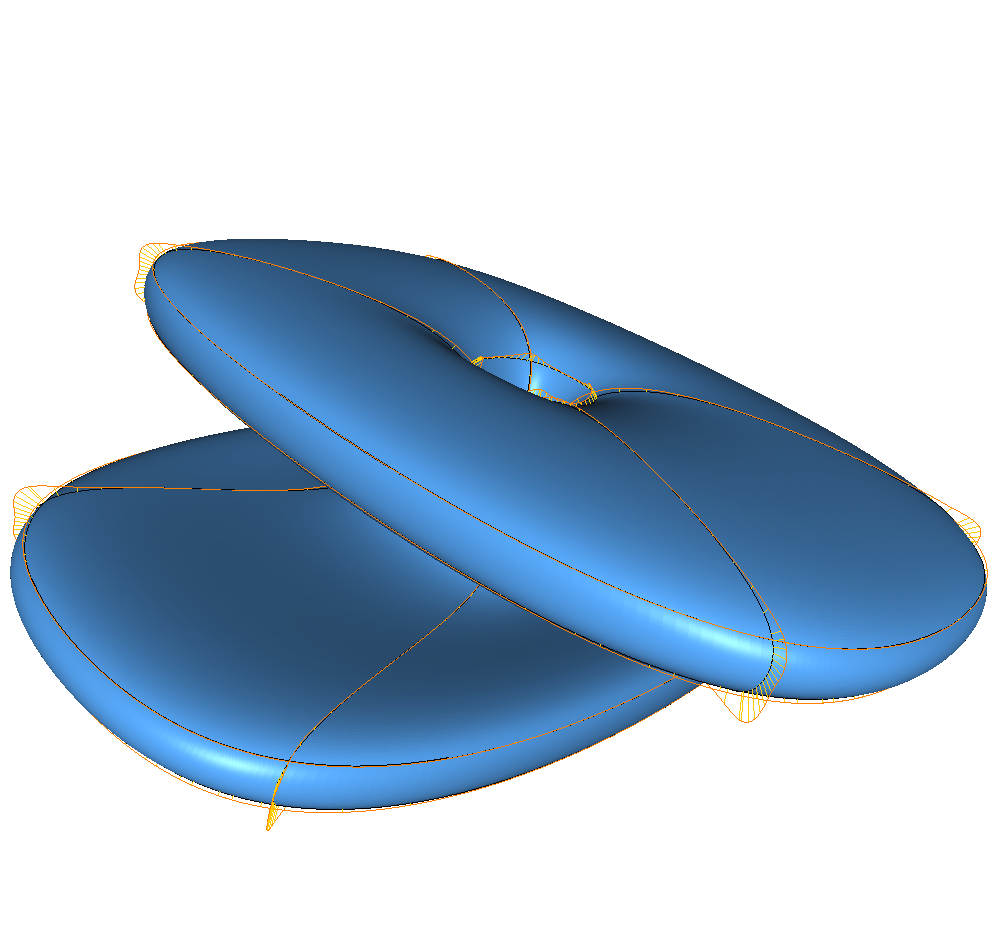}\label{fig:ex2_4}}}
\\
\subfigure[$D^3C^1P^2S^4$]
{\includegraphics[width=0.9\textwidth/4,trim=0cm 0cm 0cm 2cm, clip=true]{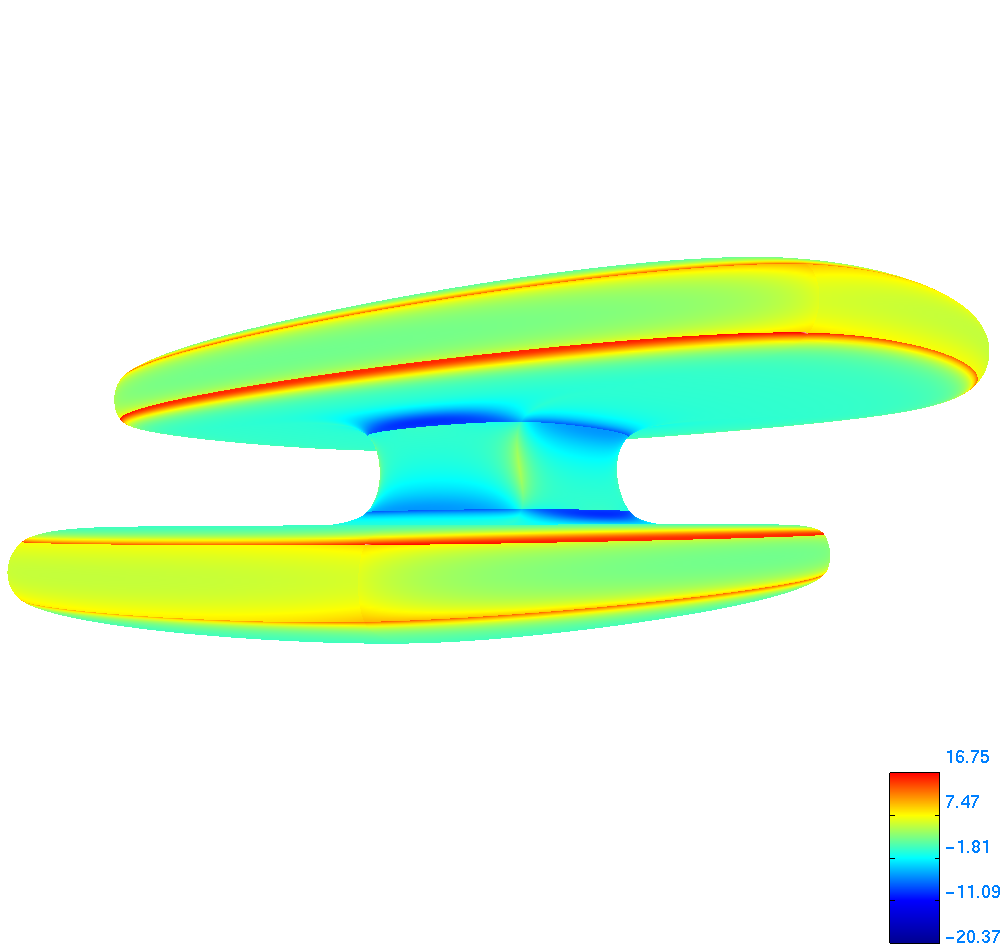}\label{fig:ex2_9}}
\hfill
\subfigure[$D^5C^2P^2S^4$]
{\includegraphics[width=0.9\textwidth/4,trim=0cm 0cm 0cm 2cm, clip=true]{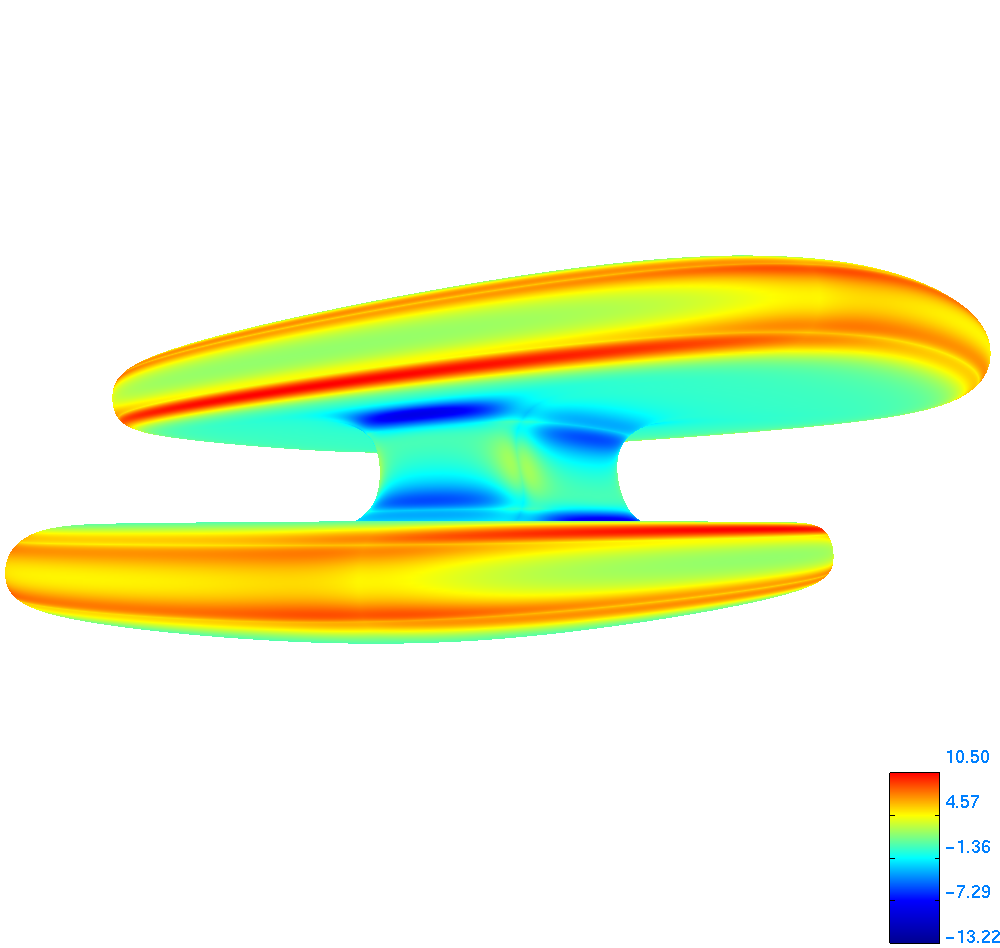}\label{fig:ex2_10}}
\hfill
\subfigure[$D^4C^2P^3S^6$]
{\includegraphics[width=0.9\textwidth/4,trim=0cm 0cm 0cm 2cm, clip=true]{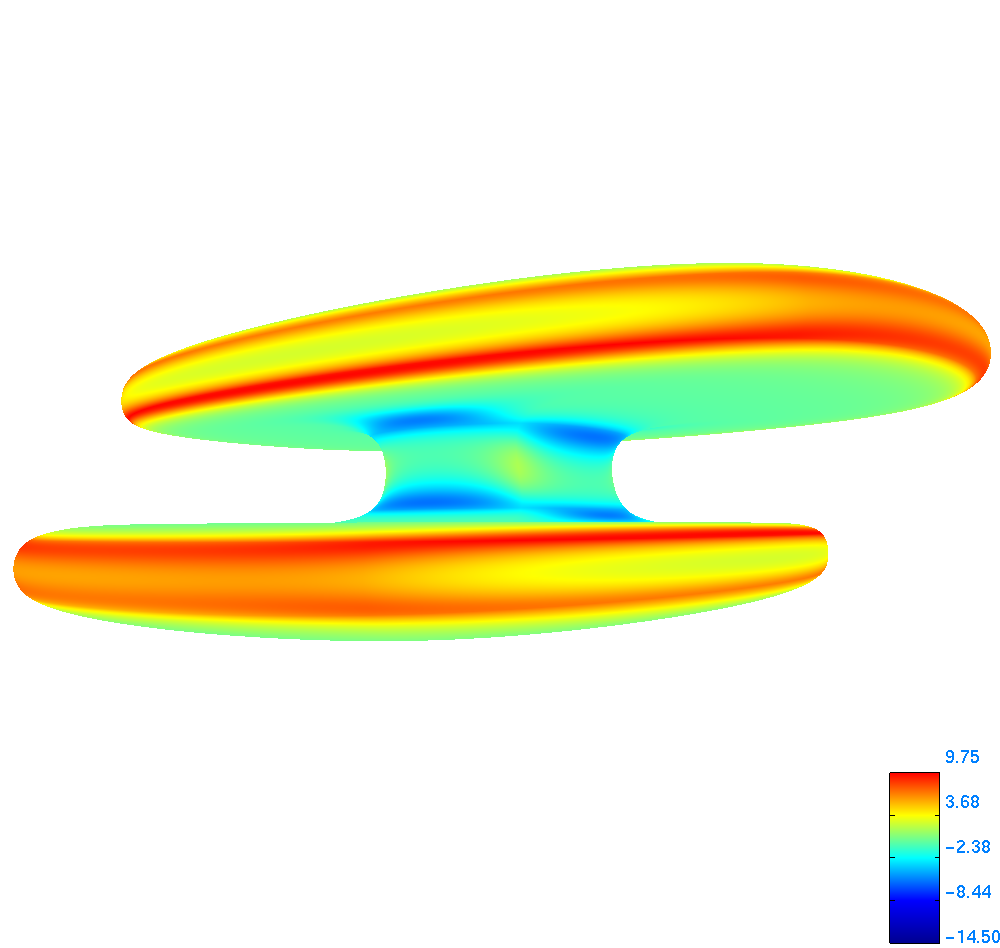}\label{fig:ex2_11}}
\hfill
\subfigure[$D^6C^3P^3S^6$]
{\includegraphics[width=0.9\textwidth/4,trim=0cm 0cm 0cm 2cm, clip=true]{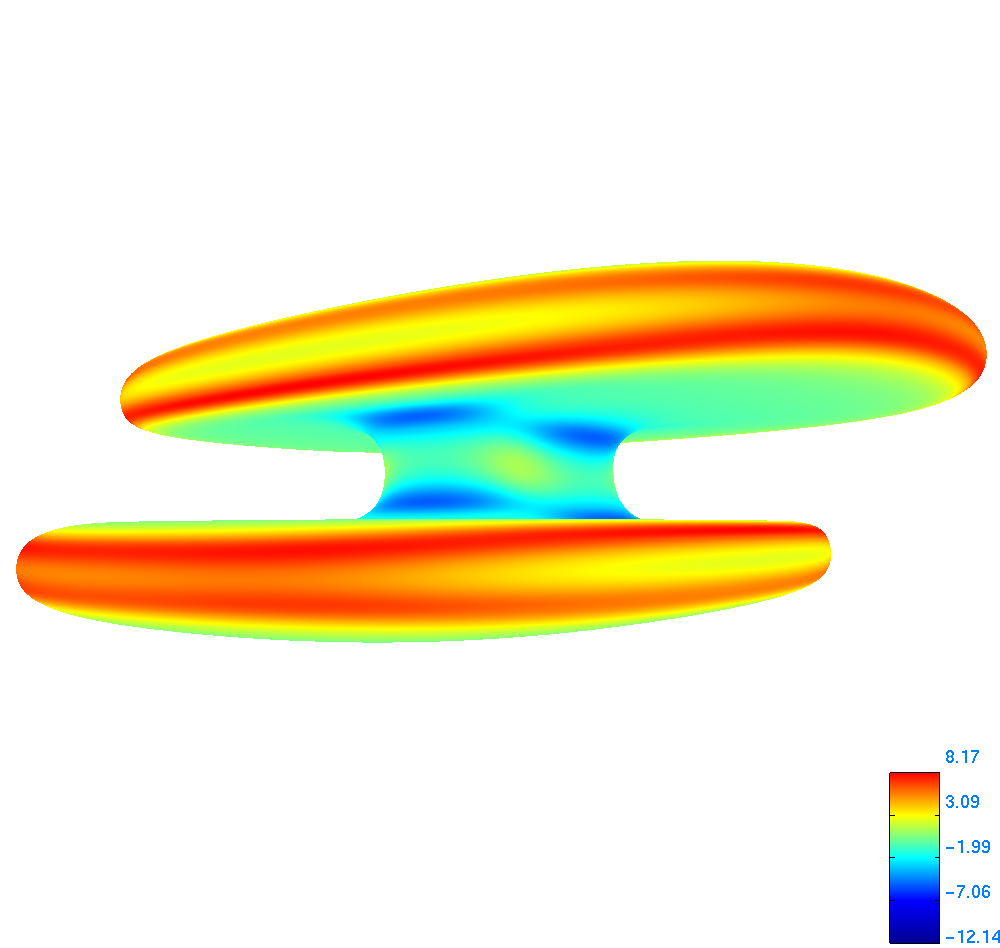}\label{fig:ex2_12}}\\
\caption{Local interpolatory surfaces with augmented parametrization built upon different classes of univariate splines $D^gC^kP^mS^w$ ($g$, $k$, $m$ and $w$ are the degree, order of continuity, maximum degree of polynomials that are reproduced and support width).  1st row: Input mesh depicted from two different points of view; 2nd row: Surfaces and curvature combs of section curves superimposed; 3rd row: Mean curvature.}\label{fig:ex2}
\end{figure}

This section covers some examples of augmented surfaces, which will allow us to highlight several important features of the proposed method. In particular:
\begin{itemize}[leftmargin=*]
\item Augmented surfaces have significant better quality with respect to tensor product ones.
 In fact, when the data points are unevenly distributed, tensor product surfaces show unwanted interpolation artifacts. According to our experiments, these artifacts are not present in the augmented counterpart.
\item For each class $D^gC^kP^mS^w$ of local, univariate, non-uniform splines in \cite{BCR13a,ABC13}, we get an augmented surface with the same smoothness and support width, and where the section curves are univariate splines in the given class. Thus we have a large family of surfaces with different properties, each of which, thanks to the augmented parametrization, is a good interpolant to the input mesh.
\item Even in the case where a tensor product surface has no artifacts, its augmented counterpart can better approximate the input mesh.
    \end{itemize}

\noindent
In the following we discuss the above points in more detail, also with the help of graphical illustrations.

Before delving into details, we recall that the method for computation of the edge parameter intervals greatly impacts the surface shape and its quality.
Therefore, to allow a fair comparison, all our examples are based on the same technique. Comparing different approaches of curve parametrization is not the scope of this paper. The reader may refer to \cite{Fang2013} for a recent study on the topic.

For the augmented surfaces, the edge parameter intervals are computed through the centripetal parametrization as illustrated in the previous section.
Moreover, tensor product surfaces are parameterized by averaging the centripetal parameter interval values, in such a way that each isocurve can have the same (non-uniform) parameter sequence. Note that, in this case, a different parameter sequence is created for each domain direction. For brevity, we call this approach a \emph{mean parametrization}.

Figure \ref{fig:ex1} shows two surfaces obtained with the augmented and mean parameterizations, which interpolate
a \qtext{modified} torus mesh. The uneven length of the mesh edges makes it a challenging task to find an interpolating surface of good quality.
In fact, the tensor product surface shows noticeable ripples both in the shaded display and in the mean curvature graph.
In contrast, these artifacts are not encountered in the augmented surface, which exhibits a \qtext{fair} behavior and a more pleasant curvature graph.

Focusing on the red-colored section polyline in Figure \ref{fig:ex1_1} will help us in understanding the different behavior of the two surfaces. Figure \ref{fig:ex1_5} depicts the corresponding section curves obtained when the parameter intervals are generated by the augmented and mean parametrizations.
The averaging step required by the mean parametrization entails that the parameter intervals assigned to the edges of the considered polyline are very different than the initial centripetal parametrization and have little relationship with the geometry of the interpolation data. This is responsible for the undesired ripples.
Oppositely, in the augmented surface the section polyline vertices are interpolated
at the centripetal parameters and the resulting section curve is free of shape artifacts.

Figure \ref{fig:ex3} (for now restricted to cases \ref{fig:ex3_1} and \ref{fig:ex3_2}) is another example where the mean parametrization fails compared to the augmented one. The tensor product surface with mean parametrization self intersects, whereas no artifact is present in the augmented surface.
This example also illustrates the ease with which local interpolation allows for managing open surfaces.
In particular, the mesh was extrapolated in a linear way over the boundary to get an additional layer of faces and vertices. With these additional data, the boundary patches were straightforwardly computed by formula \eqref{eq:S}.

The example in Figure \ref{fig:ex2b} emphasizes that an augmented surface may be preferable compared
to a tensor product surface even when the latter does not present unwanted artifacts.
The red-colored section polyline in Figure \ref{fig:ex2b_mesh} is a regular square. Such symmetry in the data is reflected in the corresponding section curve of the augmented surface (Figure \ref{fig:ex2b_section}).
Conversely, this is not the case with the tensor product surface, as a side effect of the underlying mean parametrization.
Therefore, even if both surfaces are $G^2$-continuous and free of interpolation artifacts, the augmented one represents more faithfully the input data.

Figure \ref{fig:ex2} shows local interpolatory surfaces with augmented parametrization based upon various classes $D^gC^kP^mS^w$.
Regardless of the different properties of the underlying fundamental functions, all the displayed surfaces are fair and closely resemble the shape of the input mesh.
The difference between one surface and another is made apparent by the curvature comb of the section curves and by the mean curvature graph.
In view of Proposition \ref{prop:Ck}, the displayed surfaces are $G^k$ continuous with $k=1,2,3$, according to the continuity of the class \class.

\afterpage{\clearpage}

\section{Local interpolation of meshes with extraordinary vertices by augmented surface patches}
\label{sec:augmented_extraordinary}
In this section we discuss how to generate local interpolatory surfaces of high quality when the input mesh contains extraordinary vertices.

To this aim, we observe that a mesh containing extraordinary vertices can be partitioned into \emph{regular} and \emph{extraordinary regions}.
The regular regions are formed by all the mesh faces where the local method described in Section \ref{sec:regular_case} can be applied, whereas the remaining faces form the extraordinary regions. The generation of an augmented surface patch of the form \eqref{eq:S} requires that we can uniquely determine a surrounding rectangular grid of mesh vertices, whose size depends on the support of the patch.
In particular, we need a $w\times w$ vertex grid, when the underlying fundamental functions belong to class \class.
We say that a patch is \emph{regular} when it can be generated by formula \eqref{eq:S}, whereas it is \emph{extraordinary} otherwise.
Figure \ref{fig:network_holes} illustrates the regular and extraordinary regions of a sample mesh for $w=4,6$. As shown in the figure, there can also be more extraordinary vertices close by. In this case, if the support width is greater than 4, an extraordinary patch does not necessarily contain an extraordinary vertex.

As a result of the construction of the regular augmented patches wherever possible, we obtain a surface with \qtext{holes} surrounding the extraordinary vertices. For example, in the cases of support width 4 or 6, around each isolated extraordinary vertex we will have a hole corresponding to one ring or two rings of mesh faces respectively. Moreover, the boundaries of the augmented regular patches will form a network of open curves of class \class.

Therefore, the problem to be addressed is how to patch the extraordinary portions of the mesh with sufficient smoothness and how to properly manage the join between regular and extraordinary patches, taking into account the underlying augmented parametrization.

At this point, it is easy to understand that the entire method is not intended for meshes mainly composed of extraordinary vertices. In fact, the regular portion of the composite surface serves as a \qtext{guide} for the generation of the extraordinary patches that must join to it.

\begin{figure}[t]
\centering
\vspace{0.2cm}
\subfigure[w=4]{\includegraphics[width=0.35\textwidth,trim=0.5cm 2.0cm 0.5cm 1.0cm,clip=true]{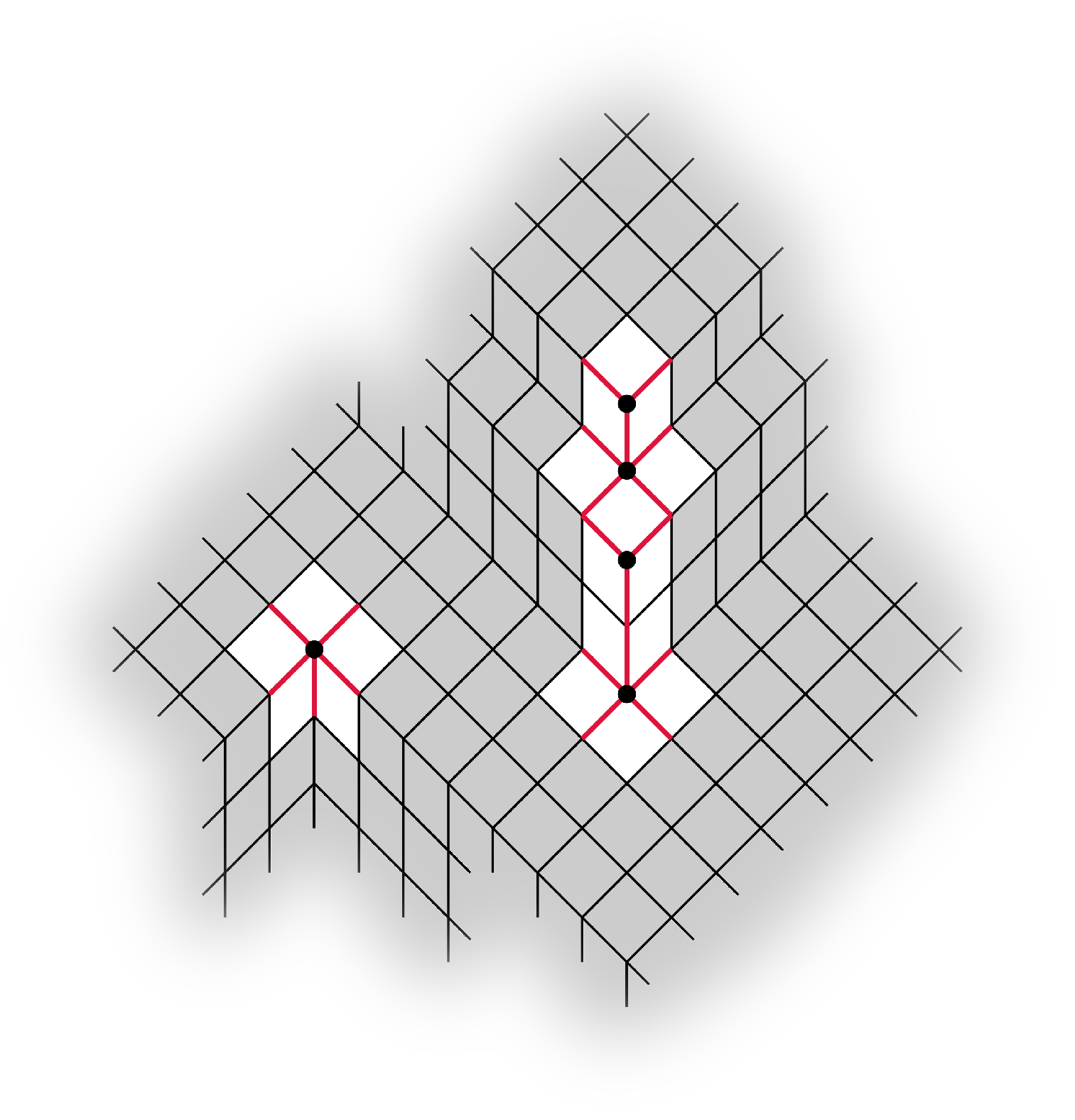}\label{fig:network_holes4}}
\hspace{1cm}
\subfigure[w=6]{\includegraphics[width=0.35\textwidth,trim=0.5cm 2.0cm 0.5cm 1.0cm,clip=true]{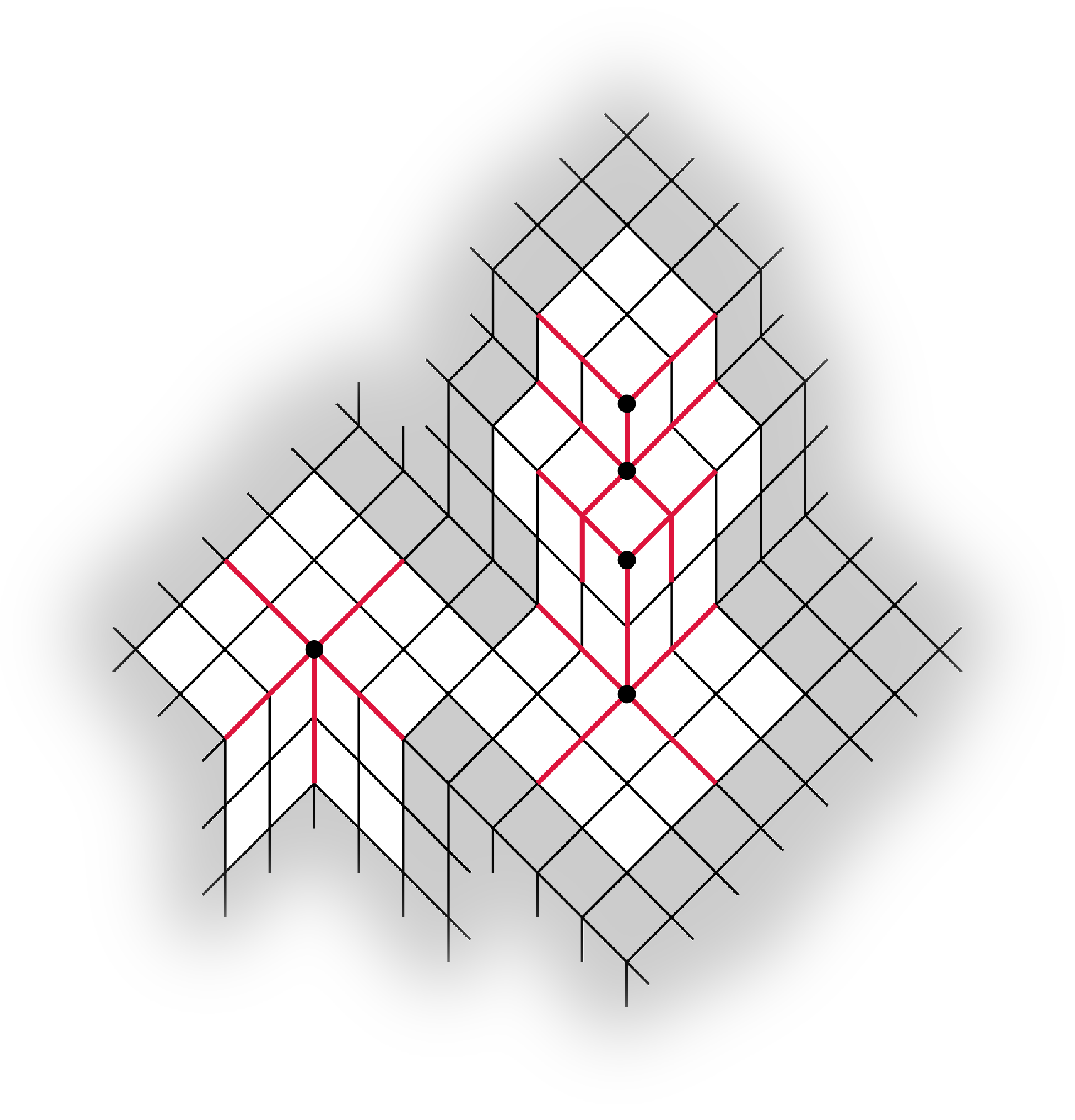}\label{fig:network_holes6}}
\caption{
Mesh containing several extraordinary vertices. The regular regions are highlighted in gray for support width $w=4$ and $w=6$.}
\label{fig:network_holes}
\end{figure}

Our construction for the extraordinary patches is based on the well-known Coons-Gregory scheme, which we will tweak in order to handle an augmented parametrization.
To explain the basic idea, we start by considering the $G^1$ Coons-Gregory patches \cite{Gregory1974} (see also the classical text books \cite{Hoschek93,Farin2002a}), which can be more familiar to the majority of readers.
Successively, we also address the $G^2$ form of these patches, initially proposed in \cite{MW91} and later developed in \cite{Hermann1996}.
As a result, depending on the method used, an extraordinary patch will join with $G^1$ or $G^2$ continuity to the neighboring patches, that can be either regular or extraordinary.
At the same time, as we have seen in the previous section, the regular portions of the surface will have arbitrary smoothness $G^k$ away from the boundary between the regular and extraordinary regions.

The input data for a Coons-Gregory patch are 4 boundary curves and the related cross-boundary derivative fields (in short cross-derivatives) of the first and, possibly, second order.
Each of the 4 boundaries may join either an extraordinary patch or a regular one.
In the former situation, the boundary curve and cross-derivatives need to be prescribed in an appropriate way and how to do so is an independent issue. Therefore, from now we assume that this information is given. We will return on the problem of determining the missing boundary information in Section \ref{sec:network}, where we discuss how these data have been generated in our prototype system.

In the latter case, where the boundary joins a regular patch of the form \eqref{eq:S}, the input data are sampled from such a patch and we need to take into account that the cross-derivatives vary according to the augmented parametrization. More precisely, the transversal derivatives of an augmented patch depend on the local parametrization functions and, according to Lemma \ref{prop:Gamma}, the mapping from the $uv$ domain to the local variables $x,y$ is different at each boundary point.
The two following subsections are devoted to illustrating how a Coons-Gregory patch needs to be defined in order to correctly interpolate the data sampled from an augmented regular patch.

Before delving into details, we introduce the setting and notation.
As in Section \ref{sec:regular_case}, we assume that a parameter interval value is assigned to every mesh edge, both in the regular regions and in the extraordinary ones.
This is reasonable, since an automatic parametrization method generally yields a parameter interval in both cases. For instance, if we wish to use the centripetal parametrization, formula \eqref{eq:d_e} applies regardless of whether an edge belongs in a regular region or not. Consequently, the parameter intervals associated with the edges of an extraordinary face may not form a rectangle. This yields an augmented parametrization for the extraordinary patches as well.
We denote an augmented extraordinary patch $\bS^*$, to distinguish it from the regular patches so far denoted by $\bS$.

Figure \ref{fig:coons} illustrates the labeling of the relevant quantities needed to construct an augmented Coons-Gregory patch. We will adopt a simplified notation with respect to the preceding part of the paper.
The face vertices to be interpolated are denoted by $\bp_i$, $i=1,\dots,4$ and sometimes we shall call these points the \emph{corners}. The boundary curves are denoted by $\bgamma_i$, $i=1,\dots,4$ and $\bchi_i,\bxi_i$ represent the respective first and second order cross-derivative fields. Moreover, $d_0$, $d_1$, $e_0$ $e_1$ are the edge parameter intervals.
\begin{figure}[t]
\centering

\begin{tikzpicture}

\tikzstyle{knot} = [shape=circle,draw=black,fill=white,minimum size=4pt,inner sep=0pt]
\tikzstyle{edge} = [draw=black,dashed]
\tikzstyle{curv} = [draw=black,thick]
\tikzstyle{cd} = [-stealth,black,semithick]

\newcommand{\ScaleFactor}{2.5}

\path ($\ScaleFactor*($(0,0)+(0,0)$)$) coordinate (p0);
\path ($\ScaleFactor*($(1,0)+(0.25,0)$)$) coordinate (p1);
\path ($\ScaleFactor*($(1,1)+(0.75,0)$)$) coordinate (p2);
\path ($\ScaleFactor*($(0,1)+(0.5,0)$)$) coordinate (p3);

\node[anchor=north] at (p0) {$\bp_0$};
\node[anchor=north] at (p1) {$\bp_1$};
\node[anchor=south west] at (p2) {$\bp_2$};
\node[anchor=south] at (p3) {$\bp_3$};

\draw[edge] (p0)--(p1) node[midway,below] {$d_0$};
\draw[edge] (p1)--(p2) node[midway,right] {$e_1$};
\draw[edge] (p3)--(p2) node[midway,below] {$d_1$};
\draw[edge] (p0)--(p3) node[midway,right] {$e_0$};

\node[knot] at (p0) {};
\node[knot] at (p1) {};
\node[knot] at (p2) {};
\node[knot] at (p3) {}; 

\draw[curv,out=45,in=135,looseness=0.5] (p0) to node[midway,above] {$\bgamma_0$} (p1);
\draw[curv,out=90,in=135,looseness=0.5] (p1) to node[midway,left,xshift=-2pt,yshift=-4pt] {$\bgamma_1$} (p2);
\draw[curv,out=45,in=135,looseness=0.5] (p3) to node[midway,above] {$\bgamma_2$} (p2);
\draw[curv,out=90,in=135,looseness=0.5] (p0) to node[midway,left,xshift=-5pt,yshift=-10pt] {$\bgamma_3$} (p3);

\path ($(p1)-\ScaleFactor*(0.35,0.1)$) coordinate (cd0);
\draw[cd,out=80,in=240,looseness=0.75] (cd0) to node[at end,anchor=south,inner sep=0pt] {$\bchi_0{\color{black},\,}\bxi_0$} ($(cd0)+\ScaleFactor*(0.13,0.4)$);
\path ($(p2)-\ScaleFactor*(0.4,0.2)$) coordinate (cd1);
\draw[cd,out=0,in=135,looseness=0.5] (cd1) to node[at end,anchor=south west,inner sep=1.5pt] {$\bchi_1{\color{black},\,}\bxi_1$} ($(cd1)+\ScaleFactor*(0.4,-0.15)$);
\path ($(p2)-\ScaleFactor*(0.45,0.07)$) coordinate (cd2);
\draw[cd,out=70,in=225,looseness=0.75] (cd2) to node[at end,anchor=south,inner sep=0pt] {$\bchi_2{\color{black},\,}\bxi_2$} ($(cd2)+\ScaleFactor*(0.2,0.35)$);
\path ($(p3)-\ScaleFactor*(0.4,0.3)$) coordinate (cd3);
\draw[cd,out=45,in=180,looseness=0.5] (cd3) to node[at end,anchor=west,inner sep=0pt] {$\bchi_3{\color{black},\,}\bxi_3$} ($(cd3)+\ScaleFactor*(0.4,0.15)$);

\path ($\ScaleFactor*(0.4,0)$) coordinate (du);
\path ($\ScaleFactor*(0,0.4)$) coordinate (dv);
\path ($(p0)-\ScaleFactor*(0.35,0.35)$) coordinate (O);
\draw[-stealth] (O)--($(O)+(du)$) node[anchor=west] {$u$};
\draw[-stealth] (O)--($(O)+(dv)$) node[anchor=south] {$v$};


\end{tikzpicture}
\caption{Schematic representation of the quantities needed for the definition of an augmented Coons-Gregory patch (the second-order cross-derivative fields $\bxi_i$ are needed for the biquintically blended patch only).}
\label{fig:coons}
\end{figure}
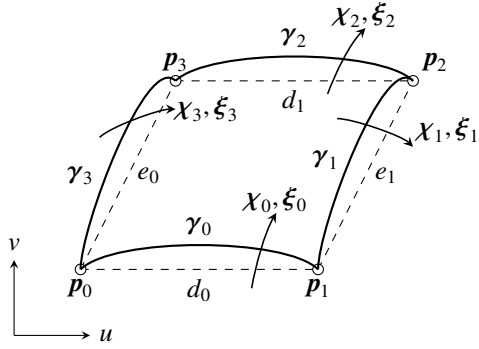

With an extraordinary patch, we associate the parametric domain $[0,1]^2$ and
two local parametrization functions $\delta(v)$ and $\epsilon(u)$, $u,v\in[0,1]$, following the same definition given in Section \ref{sec:regular_case}.
In the current notation, the formulae \eqref{eq:dh}--\eqref{eq:eh} read as follows: $\delta(v)$ is the polynomial $\delta:[0,1]\rightarrow[d_0,d_1]$ of degree $2k+1$ such that $\delta(0)=d_0$, $\delta(1)=d_1$ and $\delta^{(r)}(0)=\delta^{(r)}(1)=0$, $r=1,\dots,k$. Analogously, $\epsilon(u)$ is the polynomial $\epsilon:[0,1]\rightarrow[e_0,e_1]$ of degree $2k+1$ such that $\epsilon(0)=e_0$, $\epsilon(1)=e_1$ and $\epsilon^{(r)}(0)=\epsilon^{(r)}(1)=0$, $r=1,\dots,k$.

We describe the boundary curves in terms of the local variables
\begin{equation}\label{eq:boundary_coord}
   x_0 = u d_0, \qquad
   x_1 = u d_1, \qquad
   y_0 = v e_0, \qquad
   y_1 = v e_1,
  \end{equation}
  in such a way that each boundary segment $\bgamma_i$, is parameterized on the corresponding edge parameter interval.
  This means that $\bgamma_0(x_0)=\bp_0$ when $x_0=0$ and $\bgamma_0(x_0)=\bp_1$ when $x_0=d_0$, with similar interpolation conditions for the other three curves $\bgamma_1(y_1)$, $\bgamma_2(x_1)$, $\bgamma_3(y_0)$.

\subsection{Augmented bicubically blended Coons patches with Gregory correction}
\label{sec:coons3}
For the construction of a $G^1$ Coons-Gregory patch, we shall use as blending functions
the cubic Hermite basis on $[0,1]$, arranged in the vector
\begin{equation}\label{eq:Hvec3}
\bH(u)=
\left(
 -1,
 2u^3-3u^2+1,
 -2u^3+3u^2,
 u^3-2u^2+u,
 u^3-u^2
\right)^T.
\end{equation}
We define the augmented Coons-Gregory patch $\bS^*$ according to
\begin{equation}\label{eq:coons3}
\bS^*(u,v) = - \bH(u)^T \matr{M}_3(u,v) \bH(v),
\end{equation}
where
\begin{equation}\label{eq:M3}
\matr{M}_3(u,v)=
\left(
  \begin{array}{*5{>{\displaystyle}c}}
    0 & \bgamma_0(x_0) & \bgamma_2(x_1) & \epsilon(u)\, \bchi_0(x_0) & \epsilon(u)\, \bchi_2(x_1) \\
    \bgamma_3(y_0) & \bp_0 & \bp_3 & e_0\, \bgamma_3'(0) & e_0\, \bgamma_3'(e_0) \\
    \bgamma_1(y_1) & \bp_1 & \bp_2 & e_1\, \bgamma_1'(0) & e_1\, \bgamma_1'(e_1) \\
    \delta(v)\, \bchi_3(y_0) & d_0\, \bgamma_0'(0) & d_1\, \bgamma_2'(0) & \multicolumn{2}{c}{\multirow{2}{*}{$\matr{\Omega}_{1,1}$}} \\
    \delta(v)\, \bchi_1(y_1) & d_0\, \bgamma_0'(d_0) & d_1\, \bgamma_2'(d_1) & {} & {} \\
  \end{array}
\right),
\end{equation}
for any $(u,v)\in [0,1]^2$, being $\epsilon$ and $\delta$ the related local parametrization functions and $x_0,x_1,y_0,y_1$ determined by \eqref{eq:boundary_coord}.

The choice of expressing the cross-derivative fields in terms of the local variables $x_i$ and $y_i$, $i=0,1$, facilitates the construction of the augmented Coons-Gregory patch in the case where the boundary information is not available from an adjacent regular patch, and thus needs to be heuristically estimated.
In fact, it is more natural to specify the missing fields in the local variables, as we will do in Section \ref{sec:der_fields}.

As a consequence of the augmented parametrization, a suitable scaling factor is associated with certain entries of the patch matrix $\matr{M}_3$ to guarantee that the correct derivatives are interpolated both in the boundary and cross-boundary direction.
More precisely, the derivatives of the boundary curves, $\bgamma_i'$, $i=1,\dots,4$, are mapped to the $uv$ domain multiplying their value by the length of the related parameter interval $d_i$ or $e_i$, $i=0,1$.
In addition, in view of Proposition \eqref{prop:der_scaling}, the scaling factor needed to express the first derivatives $\bchi_0(x_0),\bchi_1(y_1),\bchi_2(x_1),\bchi_3(y_0)$ in the $uv$ domain
corresponds to the value of the appropriate local parametrization function at the evaluation point $(u,v)$.
Therefore, whereas the scaling factor is constant for the derivatives of the boundary curves, for the cross-derivatives it changes at each $(u,v)$.
Finally, we shall also take into account that,
when $\bchi_i$ is drawn from a regular augmented patch $\bS$, then the sampled value of $\bchi_i$ needs to be divided by the appropriate local parametrization function of $\bS$ in order to generate the value of $\bchi$ to be inserted in \eqref{eq:M3}.

To complete the definition of the matrix \eqref{eq:M3}, we need to provide the twist vectors matrix $\matr{\Omega}_{1,1}$, possibly with \emph{Gregory correction} for twist incompatibility. This is given by
  \begin{equation}\label{eq:Omega11}
      \matr{\Omega}_{1,1}\coloneqq
      \left(
        \begin{array}{*2{>{\displaystyle}c}}
          d_0 e_0 \, \frac{u\bchi_3'(0)+v\bchi_0'(0)}{u+v} & d_1 e_0 \, \frac{u\bchi_3'(e_0)+(1-v)\bchi_2'(0)}{u+(1-v)} \\
          d_0 e_1 \, \frac{(1-u)\bchi_1'(0)+v\bchi_0'(d_0)}{(1-u)+v} &
          d_1 e_1 \, \frac{(1-u)\bchi_1'(e_1)+(1-v)\bchi_2'(d_1)}{(1-u)+(1-v)} \\
        \end{array}
      \right),
      \end{equation}
where the scaling factors of type $d_ie_j$, $i,j=0,1$ preceding the rational terms are determined by the same argument used above.
It can be shown through direct verification that the patch $\bS^*$ interpolates the corner points, the boundary curves and the cross-boundary first derivatives. Therefore the resulting composite surface is $G^1$ continuous at the boundaries between regular and extraordinary regions and in the interior of the latter.
This type of continuity is sufficient when the regular regions are constructed from $C^1$ fundamental functions, like in the case of Catmull-Rom splines.
Otherwise, it may be desirable to use the patching scheme with higher continuity which we are going to introduce in the next subsection.

\subsection{Augmented biquintically blended Coons patches with Gregory correction}
\label{sec:coons5}

In order to construct a $G^2$ Coons-Gregory patch we exploit as blending functions the quintic Hermite basis, arranged in the vector
\begin{equation}\label{eq:Hvec5}
\begin{aligned}
\bH(u)=
& \left(
 -1,
 -6u^5+15u^4-10u^3+1,
 6u^5-15u^4+10u^3,
 -3u^5+8u^4-6u^3+u, \right. \\
& \left.
 -3u^5+7u^4-4u^3,
 -\frac{1}{2}u^5+\frac{3}{2}u^4-\frac{3}{2}u^3+\frac{1}{2}u^2,
 \frac{1}{2}u^5-u^4+\frac{1}{2}u^3
 \right)^T,
\end{aligned}
\end{equation}
and we compute the value of the augmented patch $\bS^*$ according to
\begin{equation}\label{eq:coons5}
\bS^*(u,v) = - \bH(u)^T \matr{M}_5(u,v) \bH(v).
\end{equation}
The patch matrix is now
\begin{equation}
\matr{M}_5(u,v) =
\left(
  \begin{array}{*7{>{\displaystyle}c}}
  \multicolumn{5}{c|}{\multirow{5}{*}{$\matr{M_3(u,v)$}}} & \epsilon^2(u) \, \bxi_0(x_0) & \epsilon^2(u)\, \bxi_2(x_1) \\
  {} & {} & {} & {} & \multicolumn{1}{c|}{} & e_0^2\, \bgamma_3''(0) & e_0^2\, \bgamma_3''(e_0) \\
  {} & {} & {} & {} & \multicolumn{1}{c|}{} & e_1^2\, \bgamma_1''(0) & e_1^2\, \bgamma_1''(e_1) \\
  {} & {} & {} & {} & \multicolumn{1}{c|}{} & \multicolumn{2}{c}{\multirow{2}{*}{$\matr{\Omega}_{1,2}$}} \\
  {} & {} & {} & {\phantom{\epsilon(u)\, \bchi_0(x_0)}} & \multicolumn{1}{c|}{\phantom{\epsilon(u)\, \bchi_2(x_1)}} & {} & {} \\
  \cline{1-5}\noalign{\vskip 0.5ex}
  \delta^2(v)\, \bxi_3(y_0) & d_0^2\, \bgamma_0''(0) & d_1^2\, \bgamma_2''(0) & \multicolumn{2}{c}{\multirow{2}{*}{$\matr{\Omega}_{2,1}$}} & \multicolumn{2}{c}{\multirow{2}{*}{$\matr{\Omega}_{2,2}$}} \\
  \delta^2(v)\, \bxi_1(y_1) & d_0^2\, \bgamma_0''(d_0) & d_1^2\, \bgamma_2''(d_1) & {} & {} & {} & {} \\
  \end{array}
\right),
\end{equation}
where $\matr{M}_3$ is given by \eqref{eq:M3} and $\bxi_0(x_0),\bxi_1(y_1),\bxi_2(x_1)$, $\bxi_3(y_0)$ are the cross-boundary second derivatives.
The mixed derivatives matrices, possibly with Gregory correction for twist incompatibility, are defined as follows:
      \begin{align}
      \matr{\Omega}_{1,1} &=
      \left(
        \begin{array}{*2{>{\displaystyle}c}}
          d_0 e_0 \, \frac{u^2\bchi_3'(0)+v^2\bchi_0'(0)}{u^2+v^2} &
          d_1 e_0 \, \frac{u^2\bchi_3'(e_0)+(1-v)^2\bchi_2'(0)}{u^2+(1-v)^2} \\
          d_0 e_1 \, \frac{(1-u)^2\bchi_1'(0)+v^2\bchi_0'(d_0)}{(1-u)^2+v^2} &
          d_1 e_1 \, \frac{(1-u)^2\bchi_1'(e_1)+(1-v)^2\bchi_2'(d_1)}{(1-u)^2+(1-v)^2} \\
        \end{array}
      \right),
      \\
      \matr{\Omega}_{1,2} &=
      \left(
        \begin{array}{*2{>{\displaystyle}c}}
          d_0 e_0^2 \, \frac{u^2\bchi_3''(0)+v^2\bxi_0'(0)}{u^2+v^2} &
          d_1 e_0^2 \, \frac{u^2\bchi_3''(e_0)+(1-v)^2\bxi_2'(0)}{u^2+(1-v)^2} \\
          d_0 e_1^2 \, \frac{(1-u)^2\bchi_1''(0)+v^2\bxi_0'(d_0)}{(1-u)^2+v^2} &
          d_1 e_1^2 \, \frac{(1-u)^2\bchi_1''(e_1)+(1-v)^2\bxi_2'(d_1)}{(1-u)^2+(1-v)^2} \\
        \end{array}
      \right),
      \\
      \matr{\Omega}_{2,1} &=
      \left(
        \begin{array}{*2{>{\displaystyle}c}}
          d_0^2 e_0 \, \frac{u^2\bxi_3'(0)+v^2\bchi_0''(0)}{u^2+v^2} &
          d_1^2 e_0 \, \frac{u^2\bxi_3'(e_0)+(1-v)^2\bchi_2''(0)}{u^2+(1-v)^2} \\
          d_0^2 e_1 \, \frac{(1-u)^2\bxi_1'(0)+v^2\bchi_0''(d_0)}{(1-u)^2+v^2} &
          d_1^2 e_1 \, \frac{(1-u)^2\bxi_1'(e_1)+(1-v)^2\bchi_2''(d_1)}{(1-u)^2+(1-v)^2} \\
        \end{array}
      \right),
      \\
      \matr{\Omega}_{2,2} &=
      \left(
        \begin{array}{*2{>{\displaystyle}c}}
          d_0^2 e_0^2 \, \frac{u^2\bxi_3''(0)+v^2\bxi_0''(0)}{u^2+v^2} &
          d_1^2 e_0^2 \, \frac{u^2\bxi_3''(e_0)+(1-v)^2\bxi_2''(0)}{u^2+(1-v)^2} \\
          d_0^2 e_1^2 \, \frac{(1-u)^2\bxi_1''(0)+v^2\bxi_0''(d_0)}{(1-u)^2+v^2} &
          d_1^2 e_1^2 \, \frac{(1-u)^2\bxi_1''(e_1)+(1-v)^2\bxi_2''(d_1)}{(1-u)^2+(1-v)^2} \\
        \end{array}
      \right).
      \end{align}

The interpretation of the construction and of the scaling factors that appear in $\matr{M}_5$ is conceptually similar to what we have seen in the preceding subsection and therefore no additional comment is needed.
As before, the terms involving second order derivatives require appropriate scaling to map the local variables $x,y$ to the $uv$ domain and again the proper scaling factor can be determined from relation \eqref{eq:der_scaling}.
It can be shown through direct verification that the patch $\bS^*$ interpolates the corner points, the boundary curves and the cross-boundary first and second derivatives. Therefore it guarantees that the resulting composite surface is $G^2$ continuous at the boundaries between regular and extraordinary regions and in the interior of the latter.

\subsection{Mesh regions with extraordinary vertices and \texorpdfstring{$G^1/G^2$}{G1/G2} compatibility conditions}
\label{sec:network}

The construction of the Coons-Gregory patches requires boundary curves and cross-derivative fields.
When this information cannot be sampled from an existing regular patch, then it shall be drawn from the mesh itself, which represents the only available data.

Once the boundary curves have been computed, the cross-derivative fields can be constructed by suitable interpolation of their values at the corners according to standard procedures (see \cite{Farin2002a} and \cite{Hermann1996} for the $G^1$ and $G^2$ case respectively). Therefore, this section is mainly devoted to discussing the construction of the boundary curves. To make the paper self-contained, in Subsection \ref{sec:der_fields} we only briefly review how the cross-derivatives can be generated based on existing approaches.

Methods to create fair curve networks from arbitrary meshes were suggested in \cite{Moreton-Sequin94}, and, to the best of our knowledge, this is the only available reference on the subject.
Our case is different, however, since much of the curve network is created using the local interpolating univariate splines \class and there only remains to define isolated segments of the network.

For meshes that contain extraordinary vertices, we extend our definition of a section polyline, given is Section \ref{sec:regular_case}. A section polyline is a sequence of adjacent edges of the mesh characterized in one of the following ways: i) it is closed and all its vertices are regular or ii) it is open, the first and last points are extraordinary vertices and all the remaining vertices are regular. Accordingly, a section curve is one that interpolates the vertices of a section polyline.

Wherever possible, we wish to construct the patch boundary segments in such a way that, globally, the section curves belong to class \class. We observe that, if $w>4$, some curve segments can be directly determined by locally interpolating the corresponding section polyline.
For example,
with reference to Figure \ref{fig:network_holes}, only the red-colored edges cannot be naturally associated with a local spline interpolant \class and therefore must be treated in an ad hoc manner.

When they can not be otherwise determined, we construct the boundary segments by polynomial interpolation of endpoints and endpoint derivatives.
The latter shall be sampled from existing curve segments when available, in such a way that the corresponding section curve has globally the correct continuity. Otherwise they need to be heuristically estimated.
In doing so, we shall take into account that the derivatives at one vertex should satisfy $G^k$, $k=1,2$ compatibility conditions (depending on the continuity of the Coons-Gregory patch), meaning that locally the resulting network of section curves can be embedded into a $C^k$ surface. In particular, $G^1$ compatibility entails that all the curve tangents lie on the same plane, which is the tangent plane. Conditions for $G^2$ compatibility are more complex and do not have a direct geometric interpretation \cite{Hermann1996,HPS2012}.

Our approach for generating $G^1$ and $G^2$ compatible derivatives is based on least squares polynomial approximation.
We point out that this is not the first time that a similar idea is used.
For example, in \cite{ABCetal2013}, least squares polynomial approximation was exploited to tweak locally the derivatives
of Catmull-Clark subdivision surfaces in the neighborhood of extraordinary vertices.

The outline of the procedure is as follows. First, we generate a proper set of points in the vicinity of the vertex where we want to estimate the derivatives, with the criterion that they should serve as a \qtext{guide} for the shape of the surface that we want to fit. Then we compute a least squares polynomial that approximates these points and we sample its derivatives along proper directions. This method guarantees $G^2$ compatibility and, in our tests, has always generated visually pleasing surfaces in the vicinity of the extraordinary vertices.
The following subsection describes the approach in greater detail.
One can reasonably expect that even better results can be produced by more elaborate techniques and further investigation is an interesting topic for future research.

\subsubsection{Derivatives generation}
\label{sec:der_comp}
Let $\bp_0$ be a vertex of valence $n$ where we want to compute a set of $G^1$ or $G^2$ compatible derivatives. In this section, we denote by $\bp_i$, $i=1,\dots,n$, the endpoints of the edges emanating from $\bp_0$, by $d_i$ the parameter intervals of edges $\overline{\bp_0 \bp_i}$ and by $\bff_i$ the vectors $\bp_i-\bp_0$.

Next, we associate with each edge $\overline{\bp_0\bp_i}$ a vector $\bTau_{\bp_0,\bp_i}$ defined as
\begin{equation}\label{eq:bessel_gen}
\bTau_{\bp_0,\bp_i}=\frac{\alpha_i}{d_i} \bff_i - \frac{1-\alpha_i}{\bar{d}_i} \,\overline{\bff}_i,
\end{equation}
where
\begin{equation}\label{eq:bessel_part}
\alpha_i = \frac{\bar{d}_i}{d_{i}+\bar{d}_i}, \qquad
\bar{d}_i = - \sum_{\substack{j=1\\ j \neq i}}^n \cos\left(\frac{2 \pi (j-i)}{n}\right) d_{j}, \qquad
\overline{\bff}_i = \sum_{\substack{j=1\\ j \neq i}}^n \left|\cos\left(\frac{2 \pi (j-i)}{n}\right)\right| \bff_j.
\end{equation}

When $n=4$, equation \eqref{eq:bessel_gen} reduces to the well-known Bessel estimate for computing an approximation of the first derivative of a parametric curve \cite[Section 9.8]{Farin2002a}.
As a consequence, for general valence $n \neq 4$, $\bTau_{\bp_0,\bp_i}$ represents a heuristic estimate of the first derivative at $\bp_0$ of the curve segment between $\bp_0$ and $\bp_i$.
In particular, we can observe that, when $n$ is even and the points $\bp_i$ have rotational symmetry with respect to $\bp_0$, then $\bTau_{\bp_0,\bp_i}$ corresponds to the Bessel formula applied to the three points $\bp_{i+\frac{n}{2}}, \bp_0, \bp_i$, which are intuitively associated with a curve passing through $\bp_0$.\\

At this point, if only $G^1$ compatibility is required, we can simply get an appropriate set of derivatives at $\bp_0$ by projecting the vectors $\bTau_{\bp_0,\bp_i}$, $i=1,\dots,n$ on a common plane.

We will now proceed to determine a $G^2$ compatible set of derivatives.
In particular, let $\btau_{\bp_0,\bp_i}^{(1)}$ and $\btau_{\bp_0,\bp_i}^{(2)}$ be the first and the second derivatives of the curve segment associated with the edge $\overline{\bp_0\bp_i}$.
Our strategy is to construct a polynomial $\bP$ that interpolates $\bp_0$ and approximates in a least-squares sense a suitable set of points $\bq_j, j=1,\dots,2n$ around $\bp_0$ and set $\btau_{\bp_0,\bp_i}^{(1)}$ and $\btau_{\bp_0,\bp_i}^{(2)}$ as the derivatives of such polynomial along proper directions.
The approximation points $\bq_j$ are chosen so that the polynomial $\bP$ will have a reasonable shape
in a small neighborhood of $\bp_0$.
In particular, for each $i=1,\dots,n$, $\bq_i$ and $\bq_{n+i}$ are respectively the values at parameters $\frac{d_i}{4}$ and $\frac{d_i}{2}$ of the cubic polynomial $\blambda$ such that $\blambda(0)=\bp_0$, $\blambda'(0)=\bTau_{\bp_0,\bp_i}$, $\blambda(d_i)=\bp_i$, $\blambda'(d_i)=\bTau_{\bp_i,\bp_0}$. Their expressions is given explicitly by
\begin{equation}
\begin{aligned}
\bq_i &= \blambda\left(\frac{d_i}{4}\right) = \frac{1}{64}\left(54\bp_0+10\bp_i+3d_i\left(3\bTau_{\bp_0,\bp_i}-\bTau_{\bp_i,\bp_0}\right)\right),\\
\bq_{n+i} &= \blambda\left(\frac{d_i}{2}\right) = \frac{1}{8}\left(4\bp_0+4\bp_i+d_i\left(\bTau_{\bp_0,\bp_i}-\bTau_{\bp_i,\bp_0}\right)\right).
\end{aligned}
\end{equation}

Note that the above vector $\bTau_{\bp_i,\bp_0}$ represents a derivative at $\bp_i$, and thus it shall be sampled from the adjacent segment of the section curve passing through $\bp_i$ and $\bp_0$, when this is available, or otherwise computed from formulae \eqref{eq:bessel_gen}--\eqref{eq:bessel_part}.

We exploit a bivariate polynomial $\bP$ of degree $3$ or $2$ respectively in the case $n\geq 5$ or $n=3,4$.
The coefficients of $\bP$ are determined componentwise by minimizing the expression
\begin{equation}
\sum_{j=1}^{2n} \left(\bP(x_j,y_j)-\bq_j\right)^2,
\end{equation}
where the parametric coordinates $(x_j,y_j)$ associated with the point $\bq_j$ are given by
\begin{equation}\label{eq:rxy}
(x_j,y_j) = r_j\,(\cos\eta_i, \sin\eta_i), \quad j=i,n+i,
\end{equation}
and the angles $\eta_i, i=1,\dots,n$ are obtained by mapping onto the $xy$ plane the spatial configuration formed by the angles $\zeta_i\coloneqq\widehat{\bTau_{\bp_0,\bp_i},\bTau_{\bp_0,\bp_{i+1}}}, i=1,\dots,n$, namely
\begin{equation}
\eta_1=0, \qquad \eta_i=\eta_{i-1}+\zeta_i \frac{2 \pi}{\sum_{j=1}^n \zeta_j}, \quad i=2,\dots,n.
\end{equation}
The value of $r_j$ in \eqref{eq:rxy} is a free parameter and can be exploited to locally tune the shape of the final surface.
When the edge parameter intervals $d_i$ are computed according to \eqref{eq:d_e}, then a possible choice (which we have used in the proposed examples) is
\begin{equation}
r_j=\norm{\bq_j-\bp_0}_2^\alpha.
\end{equation}
Finally, we can determine $\btau_{\bp_0,\bp_i}^{(1)}$ and $\btau^{(2)}_{\bp_0,\bp_i}$, $i=1,\dots,n$ as the first and the second derivatives of $\bP$ at $\bp_0$ in the direction determined by $\eta_i$, namely
\begin{equation}\label{eq:dir_der}
\begin{aligned}
\btau^{(2)}_{\bp_0,\bp_i}&=\frac{\partial \bP}{\partial x} (0,0) \cos\eta_i+ \frac{\partial \bP}{\partial y} (0,0) \sin\eta_i,\\
\btau^{(2)}_{\bp_0,\bp_i}&=\frac{\partial^2 \bP}{\partial x^2} (0,0) \cos^2\eta_i + 2 \frac{\partial^2 \bP}{\partial x \partial y} (0,0) \cos\eta_i \sin\eta_i + \frac{\partial^2 \bP}{\partial y^2} (0,0) \sin^2\eta_i.
\end{aligned}
\end{equation}



\subsubsection{Construction of the cross-boundary derivative fields}
\label{sec:der_fields}

We address the problem of suitably prescribing the first and second-order derivative fields, $\bchi_i$ and $\bxi_i$ respectively, across a boundary curve $\bgamma_i$. Let us suppose that $x\in [0,d_{i}]$ is the local variable which describes $\bgamma_i$. Then we define the cross-boundary derivative fields as
\begin{equation}\label{eq:chi}
\bchi_i(x) = a(x)\bgamma_i'(x) + b(x)\br(x),
\end{equation}
and
\begin{equation}\label{eq:tau}
\bxi_i(x) = a^2(x)\bgamma_i''(x) +s(x) \bgamma_i'(x) + t(x) \br(x) + 2 a(x) b(x) \br'(x)+ b^2(x) \bw(x),
\end{equation}
for suitable polynomials $a,b,s,t$,  $\br$, and $\bw$.
The above definitions ensure respectively $G^1$ and $G^2$ continuity across $\bgamma_i$ \cite[Theorem C]{Hermann1996}.

We observe that, if a regular surface patch \eqref{eq:S} has the
uniform parametrization, then both $\bchi_i$ and $\bxi_i$ have the same degree $g$ as the relative class of fundamental functions \class.
Similarly, also in the extraordinary case we take them to be polynomials of degree $g$.

To determine $\bchi_i$, we observe that $\bgamma_i'$ has degree $g-1$, and thus $a$ has to be at most linear and, by analogy, we can choose $b$ of degree $1$.
As a consequence, $\br$ has degree at most $g-1$.
In particular, our numerical experiments have shown that satisfactory results can be obtained defining $\br$ by either linear or quadratic interpolation.
For the linear case, we interpolate the vectors $\br_0=\bgamma_i'(0) \times \bn$ and $\br_1=\bgamma_i'(d_{i}) \times \bar{\bn}$, where $\bn$ and $\bar{\bn}$ are the normals at the endpoints of $\bgamma_i$.
In the quadratic case, an additional interpolation vector is estimated as $\br_{\text{m}}=\bgamma_i'\left(\frac{d_{i}}{2}\right) \times \bn_{\text{m}}$, where $\bn_{\text{m}}$ is the average of the normals to the two faces sharing the edge associated with $\bgamma_i$.
The values of $a$ and $b$ at the endpoints of $\bgamma_i$ are fixed by the requirement that
$\bchi_i(0)$ and $\bchi_i(d_{i})$ be equal to the first derivatives of $\bgamma_{i-1}$ and $\bgamma_{i+1}$ at $\bgamma_i(0)$ and $\bgamma_i(d_i)$ respectively. Hence $a$ and $b$ are uniquely determined by linear interpolation of their values at the endpoints.

For the construction of $\bxi_i$ we proceed similarly as above. Therefore, we take $s$ and $t$ to be polynomials of degree 1.
Moreover, to keep the number of degrees of freedom as low as possible and because its geometric meaning is not always clear, we set the degree of $\bw$ to $1$. Intuitively, $\bw$ represents how much the surface deviates from its tangent plane.
Thus, to complete its definition, we specify its value at the endpoints of $\bgamma_i$ by using the approach suggested in \cite{Hermann1996}, \ie we set
\begin{align}
\bw(0)&=(\mu^2 \kappa_1 + \nu^2 \kappa_2)\bn,\\
\bw(d_{i})&=(\bar{\mu}^2 \bar{\kappa}_1 + \bar{\nu}^2 \bar{\kappa}_2)\bar{\bn},
\end{align}
where $\kappa_1, \kappa_2$ and $\bar{\kappa}_1, \bar{\kappa}_2$ are the principal curvatures at $\bgamma_i(0)$
and $\bgamma_i(d_{i})$ respectively
and $(\mu,\nu)$ and $(\bar{\mu},\bar{\nu})$ are the coordinates of $\br(0)$ and $\br(d_{i})$ in the local coordinate systems of the principal directions $(\vect{K}_1,\vect{K}_2)$ and $(\overline{\vect{K}}_1, \overline{\vect{K}}_2)$.
Finally, $s$ and $t$ are uniquely determined by linear interpolation of their values at the endpoints of $\gamma_i$, that are fixed by the requirement that $\bxi_i(0)$ and $\bxi_i(d_{i})$ be equal to the second derivatives of $\bgamma_{i-1}$ and $\bgamma_{i+1}$ respectively at $\bgamma_i(0)$ and $\bgamma_i(d_{i})$.
\subsection{Examples}
\label{sec:examples_extraordinary}
In this section we discuss some examples of augmented surfaces that interpolate meshes with extraordinary vertices.
The displayed surfaces are based upon the class of univariate splines $D^5C^2P^2S^4$, which offers a good tradeoff between high continuity (the corresponding surfaces are $G^2$ continuous) and small support width (and thus higher computational efficiency).

Following the same outline of the regular case examples (cf.\ Section \ref{sec:examples_regular}), we start by comparing surfaces generated through the augmented and mean parameterizations.
For meshes with extraordinary vertices, the mean parametrization can be generalized along the same lines of \cite{Cashman-Augsdorfer-Dodgson-Sabin2009}. In particular, one parameter interval is assigned to every edge belonging in the same edge ribbon. For every ribbon, this parameter interval is the average of the intervals of all edges in the ribbon. This strategy results in the configuration illustrated in Figure \ref{fig:ex4_4} and obviously generalizes the analogous setting of parameters used in the regular case.

Figure \ref{fig:ex3} depicts a composite surface made of all regular patches and a similar surface that contains an extraordinary vertex, corresponding to the meshes in Figures \ref{fig:ex4_1} and \ref{fig:ex4_2}. In both cases,
the mean parametrization gives rise to an unwanted folding and self intersection of the surface. On the contrary, the augmented surfaces are free of interpolation artifacts and approximate faithfully the input mesh. The different result of the two parameterizations is readily apparent from the shaded display and is further emphasized by the curvature graph.

\begin{figure}[t]
\centering
\subfigure[]
{\raisebox{0.5cm}{\includegraphics[width=0.85\textwidth/4,trim={0 0 0 10cm},clip]{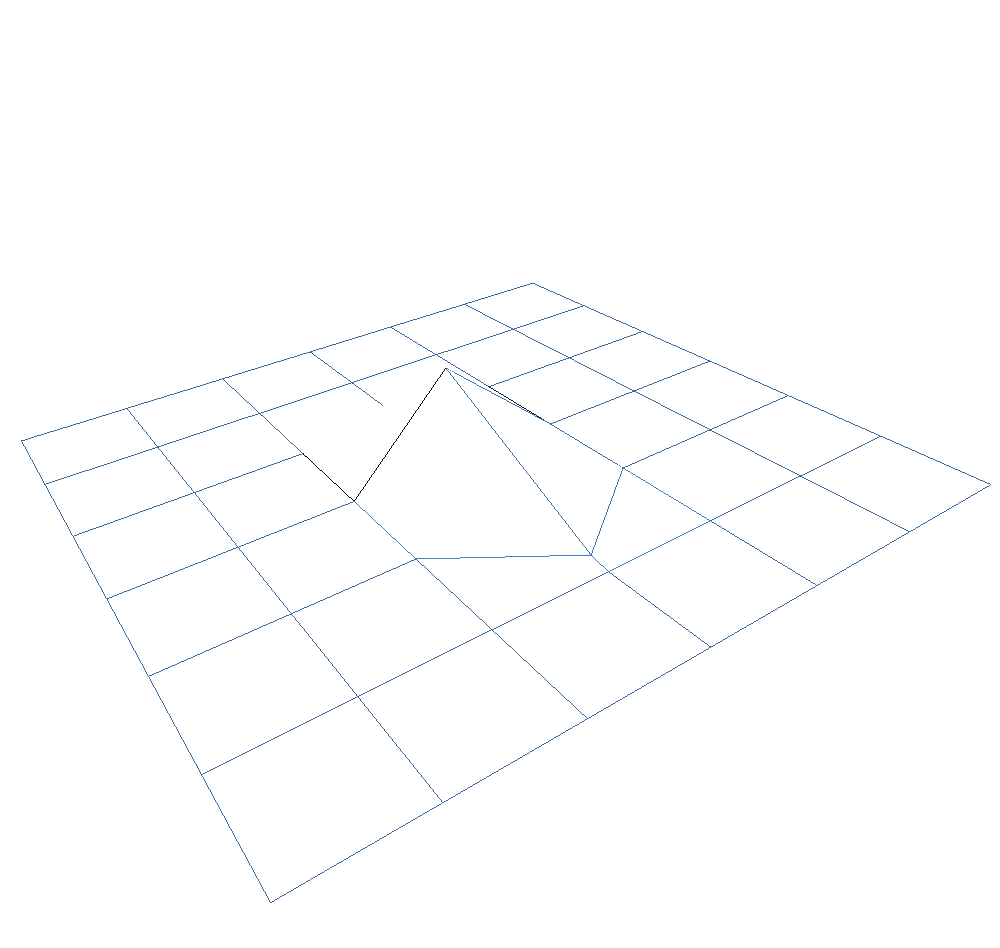}\label{fig:ex4_1}}}
\hfill
\subfigure[]
{\raisebox{0.5cm}{\includegraphics[width=0.85\textwidth/4,trim={0 0 0 10cm},clip]{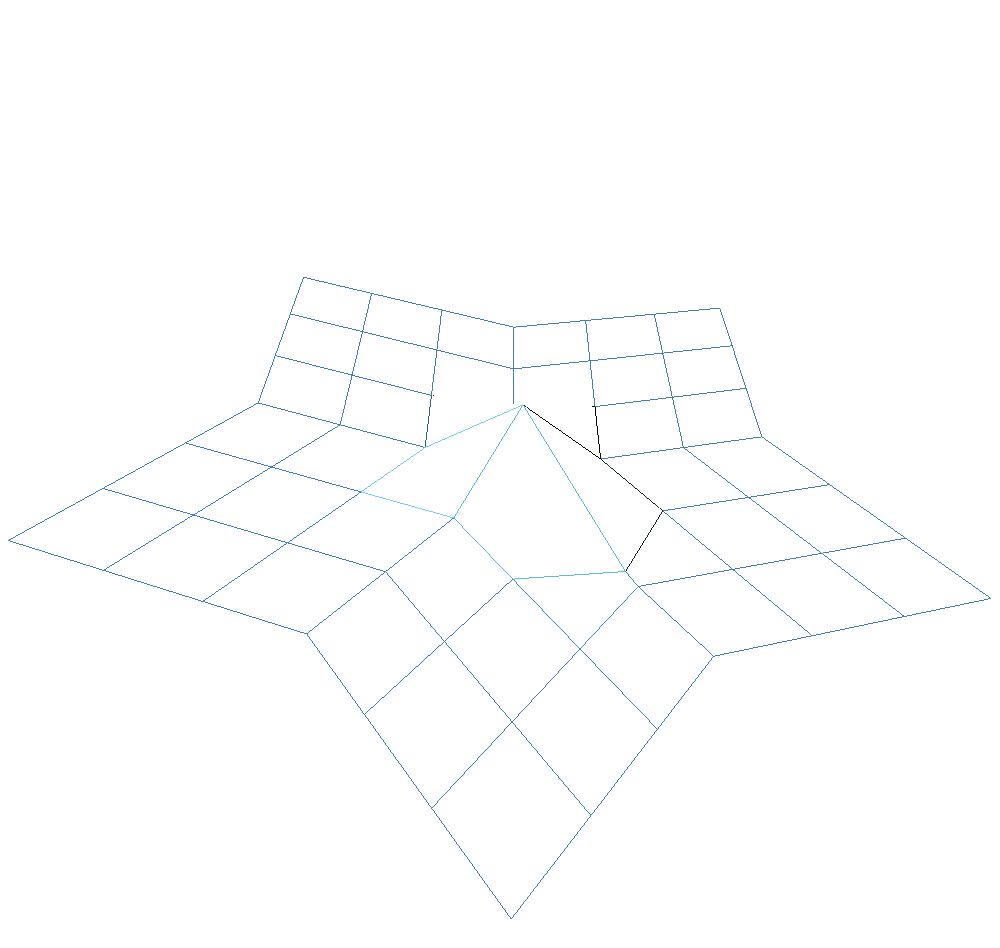}\label{fig:ex4_2}}}
\hfill
\subfigure[]
{

\begin{tikzpicture}[scale=0.65]

\tikzstyle{rv} = [shape=circle,draw=black,fill=white,minimum size=4pt,inner sep=0pt]
\tikzstyle{ev} = [shape=circle,draw=black,fill=black,minimum size=4pt,inner sep=0pt]
\tikzstyle{edge} = [draw=black,dashed]
\tikzstyle{radial} = [draw=black,solid]
\tikzstyle{eedge} = [draw=black,thick]
\tikzstyle{curv} = [draw=colspline,thick]
\tikzstyle{ecurv} = [draw=red,thick]

\newcommand{\Edge}{1.5}
\newcommand{\AngleN}{360/4}

\newcommand{\Sc}{1.2}
\newcommand{\Scc}{2}
\newcommand{\Sccc}{0.2}

\path (0,0) coordinate (V);
\path ($\Edge*({cos(0*\AngleN)},{sin(0*\AngleN)})$) coordinate (p1);
\path ($\Sc*\Edge*({cos(0*\AngleN)},{sin(0*\AngleN)})$) coordinate (t1);
\path ($\Scc*\Edge*({cos(0*\AngleN)},{sin(0*\AngleN)})$) coordinate (s1);
\path ($\Edge*({cos(1*\AngleN)},{sin(1*\AngleN)})$) coordinate (p2);
\path ($\Sc*\Edge*({cos(1*\AngleN)},{sin(1*\AngleN)})$) coordinate (t2);
\path ($\Scc*\Edge*({cos(1*\AngleN)},{sin(1*\AngleN)})$) coordinate (s2);
\path ($\Edge*({cos(2*\AngleN)},{sin(2*\AngleN)})$) coordinate (p3);
\path ($\Sc*\Edge*({cos(2*\AngleN)},{sin(2*\AngleN)})$) coordinate (t3);
\path ($\Scc*\Edge*({cos(2*\AngleN)},{sin(2*\AngleN)})$) coordinate (s3);
\path ($\Edge*({cos(3*\AngleN)},{sin(3*\AngleN)})$) coordinate (p4);
\path ($\Sc*\Edge*({cos(3*\AngleN)},{sin(3*\AngleN)})$) coordinate (t4);
\path ($\Scc*\Edge*({cos(3*\AngleN)},{sin(3*\AngleN)})$) coordinate (s4);
\path ($1.5*\Edge*({cos(1/2*\AngleN)},{sin(1/2*\AngleN)})$) coordinate (f1);
\path ($1.5*\Edge*({cos(3/2*\AngleN)},{sin(3/2*\AngleN)})$) coordinate (f2);
\path ($1.5*\Edge*({cos(5/2*\AngleN)},{sin(5/2*\AngleN)})$) coordinate (f3);
\path ($1.5*\Edge*({cos(7/2*\AngleN)},{sin(7/2*\AngleN)})$) coordinate (f4);
\path ($(f1)+\Sccc*($(f1)-(p1)$)$) coordinate (ss1);
\path ($(f1)+\Sccc*($(s1)-(p1)$)$) coordinate (ssss1);
\path ($(f2)+\Sccc*($(f2)-(p2)$)$) coordinate (ss2);
\path ($(f2)+\Sccc*($(s2)-(p2)$)$) coordinate (ssss2);
\path ($(f3)+\Sccc*($(f3)-(p3)$)$) coordinate (ss3);
\path ($(f3)+\Sccc*($(s3)-(p3)$)$) coordinate (ssss3);
\path ($(f4)+\Sccc*($(f4)-(p4)$)$) coordinate (ss4);
\path ($(f4)+\Sccc*($(s4)-(p4)$)$) coordinate (ssss4);

\draw[radial] (V)--(p1) node[midway,below,xshift=0.2cm,yshift=0.05cm] {$a$};
\draw[radial] (V)--(p2) node[midway,right,xshift=-0.05cm] {$b$};
\draw[radial] (V)--(p3) node[midway,below,xshift=0pt,yshift=0.05cm] {$c$};
\draw[radial] (V)--(p4) node[midway,right,xshift=-0.05cm] {$d$};
\draw[radial] (p1)--(t1) {};
\draw[radial] (p2)--(t2) {};
\draw[radial] (p3)--(t3) {};
\draw[radial] (p4)--(t4) {};

\draw[radial] (p1)--(f1) node[midway,right,xshift=-0.05cm]{$b$};
\draw[radial] (p2)--(f2) node[midway,below,xshift=0pt,yshift=0.05cm]{$c$};
\draw[radial] (p3)--(f3) node[midway,right,xshift=-0.05cm]{$d$};
\draw[radial] (p4)--(f4) node[midway,below,xshift=0.2cm,yshift=0.05cm]{$a$};

\draw[radial] (f1)--(p2) node[midway,below,xshift=0.2cm,yshift=0.05cm]{$a$};
\draw[radial] (f2)--(p3) node[midway,right,xshift=-0.05cm]{$b$};
\draw[radial] (f3)--(p4) node[midway,below,xshift=0pt,yshift=0.05cm]{$c$};
\draw[radial] (f4)--(p1) node[midway,right,xshift=-0.05cm]{$d$};

\draw[radial] (f1)--(ss1) {};
\draw[radial] (f2)--(ss2) {};
\draw[radial] (f3)--(ss3) {};
\draw[radial] (f4)--(ss4) {};
\draw[radial] (f1)--(ssss1) {};
\draw[radial] (f2)--(ssss2) {};
\draw[radial] (f3)--(ssss3) {};
\draw[radial] (f4)--(ssss4) {};


\node[ev] at (V) {};
\node[rv] at (p1) {};
\node[rv] at (p2) {};
\node[rv] at (p3) {};
\node[rv] at (p4) {};
\node[rv] at (f1) {};
\node[rv] at (f2) {};
\node[rv] at (f3) {};
\node[rv] at (f4) {};





\end{tikzpicture}\label{fig:ex4_3}}
\hfill
\subfigure[]
{

\begin{tikzpicture}[scale=0.65]

\tikzstyle{rv} = [shape=circle,draw=black,fill=white,minimum size=4pt,inner sep=0pt]
\tikzstyle{ev} = [shape=circle,draw=black,fill=black,minimum size=4pt,inner sep=0pt]
\tikzstyle{edge} = [draw=black,dashed]
\tikzstyle{radial} = [draw=black,solid]
\tikzstyle{eedge} = [draw=black,thick]
\tikzstyle{curv} = [draw=colspline,thick]
\tikzstyle{ecurv} = [draw=red,thick]

\newcommand{\Edge}{1.5}
\newcommand{\AngleN}{360/5}

\newcommand{\Sc}{1.2}
\newcommand{\Scc}{2}
\newcommand{\Sccc}{0.2}

\path (0,0) coordinate (V);
\path ($\Edge*({cos(0*\AngleN)},{sin(0*\AngleN)})$) coordinate (p1);
\path ($\Sc*\Edge*({cos(0*\AngleN)},{sin(0*\AngleN)})$) coordinate (t1);
\path ($\Scc*\Edge*({cos(0*\AngleN)},{sin(0*\AngleN)})$) coordinate (s1);
\path ($\Edge*({cos(1*\AngleN)},{sin(1*\AngleN)})$) coordinate (p2);
\path ($\Sc*\Edge*({cos(1*\AngleN)},{sin(1*\AngleN)})$) coordinate (t2);
\path ($\Scc*\Edge*({cos(1*\AngleN)},{sin(1*\AngleN)})$) coordinate (s2);
\path ($\Edge*({cos(2*\AngleN)},{sin(2*\AngleN)})$) coordinate (p3);
\path ($\Sc*\Edge*({cos(2*\AngleN)},{sin(2*\AngleN)})$) coordinate (t3);
\path ($\Scc*\Edge*({cos(2*\AngleN)},{sin(2*\AngleN)})$) coordinate (s3);
\path ($\Edge*({cos(3*\AngleN)},{sin(3*\AngleN)})$) coordinate (p4);
\path ($\Sc*\Edge*({cos(3*\AngleN)},{sin(3*\AngleN)})$) coordinate (t4);
\path ($\Scc*\Edge*({cos(3*\AngleN)},{sin(3*\AngleN)})$) coordinate (s4);
\path ($\Edge*({cos(4*\AngleN)},{sin(4*\AngleN)})$) coordinate (p5);
\path ($\Sc*\Edge*({cos(4*\AngleN)},{sin(4*\AngleN)})$) coordinate (t5);
\path ($\Scc*\Edge*({cos(4*\AngleN)},{sin(4*\AngleN)})$) coordinate (s5);
\path ($1.5*\Edge*({cos(1/2*\AngleN)},{sin(1/2*\AngleN)})$) coordinate (f1);
\path ($1.5*\Edge*({cos(3/2*\AngleN)},{sin(3/2*\AngleN)})$) coordinate (f2);
\path ($1.5*\Edge*({cos(5/2*\AngleN)},{sin(5/2*\AngleN)})$) coordinate (f3);
\path ($1.5*\Edge*({cos(7/2*\AngleN)},{sin(7/2*\AngleN)})$) coordinate (f4);
\path ($1.5*\Edge*({cos(9/2*\AngleN)},{sin(9/2*\AngleN)})$) coordinate (f5);
\path ($(f1)+\Sccc*($(f1)-(p1)$)$) coordinate (ss1);
\path ($(f1)+\Sccc*($(s1)-(p1)$)$) coordinate (ssss1);
\path ($(f2)+\Sccc*($(f2)-(p2)$)$) coordinate (ss2);
\path ($(f2)+\Sccc*($(s2)-(p2)$)$) coordinate (ssss2);
\path ($(f3)+\Sccc*($(f3)-(p3)$)$) coordinate (ss3);
\path ($(f3)+\Sccc*($(s3)-(p3)$)$) coordinate (ssss3);
\path ($(f4)+\Sccc*($(f4)-(p4)$)$) coordinate (ss4);
\path ($(f4)+\Sccc*($(s4)-(p4)$)$) coordinate (ssss4);
\path ($(f5)+\Sccc*($(f5)-(p1)$)$) coordinate (ss5);
\path ($(f5)+\Sccc*($(f5)-(p5)$)$) coordinate (ssss5);



\draw[radial] (p1)--(t1) {};
\draw[radial] (p2)--(t2) {};
\draw[radial] (p3)--(t3) {};
\draw[radial] (p4)--(t4) {};
\draw[radial] (p5)--(t5) {};


\draw[radial] (V)--(p1) node[midway,below,xshift=0pt,yshift=0.07cm] {$a$};
\draw[radial] (V)--(p2) node[midway,right,xshift=-0.1cm,yshift=-0.1cm] {$b$};
\draw[radial] (V)--(p3) node[midway,below,xshift=-0.1cm,yshift=0.1cm] {$c$};
\draw[radial] (V)--(p4) node[midway,below,xshift=0pt,yshift=0.1cm] {$d$};
\draw[radial] (V)--(p5) node[midway,left,xshift=0.12cm,yshift=-0.2cm] {$e$};

\draw[radial] (p1)--(f1) node[midway,right,xshift=-0.1cm,yshift=-0.1cm]{$b$};
\draw[radial] (p2)--(f2) node[midway,below,xshift=0pt,yshift=0.05cm]{$c$};
\draw[radial] (p3)--(f3) node[midway,below,xshift=0.1cm,,yshift=0.15cm]{$d$};
\draw[radial] (p4)--(f4) node[midway,left,xshift=0.15cm,yshift=-0.2cm]{$e$};
\draw[radial] (p5)--(f5) node[midway,below,xshift=0cm,yshift=0.07cm]{$a$};

\draw[radial] (f1)--(p2) node[midway,below,xshift=0pt,yshift=0.07cm]{$a$};
\draw[radial] (f2)--(p3) node[midway,right,xshift=-0.1cm,yshift=-0.1cm]{$b$};
\draw[radial] (f3)--(p4) node[midway,below,xshift=-0.1cm,yshift=0.1cm]{$c$};
\draw[radial] (f4)--(p5) node[midway,below,xshift=0pt,,yshift=0.1cm]{$d$};
\draw[radial] (f5)--(p1) node[midway,left,xshift=0.1cm,yshift=-0.1cm]{$e$};

\draw[radial] (f1)--(ss1) {};
\draw[radial] (f2)--(ss2) {};
\draw[radial] (f3)--(ss3) {};
\draw[radial] (f4)--(ss4) {};
\draw[radial] (f5)--(ss5) {};
\draw[radial] (f1)--(ssss1) {};
\draw[radial] (f2)--(ssss2) {};
\draw[radial] (f3)--(ssss3) {};
\draw[radial] (f4)--(ssss4) {};
\draw[radial] (f5)--(ssss5) {};


\node[ev] at (V) {};
\node[rv] at (p1) {};
\node[rv] at (p2) {};
\node[rv] at (p3) {};
\node[rv] at (p4) {};
\node[rv] at (p5) {};
\node[rv] at (f1) {};
\node[rv] at (f2) {};
\node[rv] at (f3) {};
\node[rv] at (f4) {};
\node[rv] at (f5) {};





\end{tikzpicture}\label{fig:ex4_4}}
\caption{\subref{fig:ex4_1} A regular mesh and \subref{fig:ex4_2} a mesh with an extraordinary vertex of valence 5. \subref{fig:ex4_3}--\subref{fig:ex4_4} Corresponding mean parametrization in the neighborhood of the central vertex: each letter indicates a different edge parameter interval value.} \label{fig:ex4}
\end{figure}
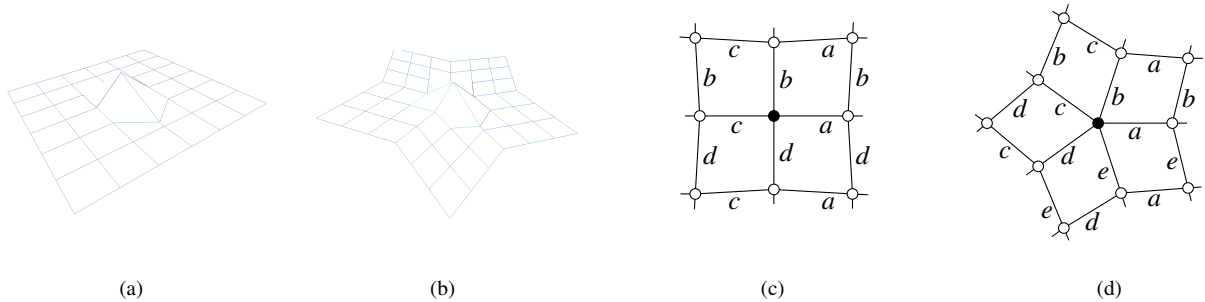

\begin{figure}[t]
\centering
\subfigure[]
{\includegraphics[width=0.85\textwidth/4]{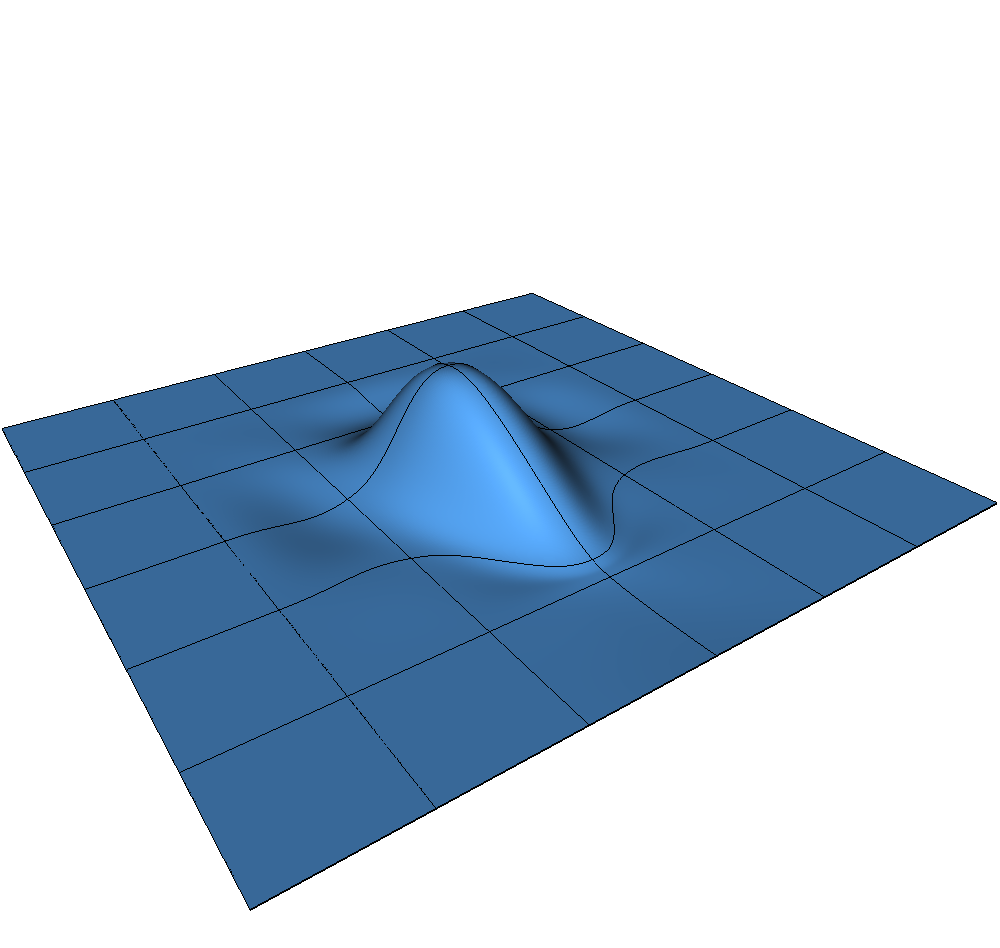}\label{fig:ex3_1}}
\hfill
\subfigure[]
{\includegraphics[width=0.85\textwidth/4]{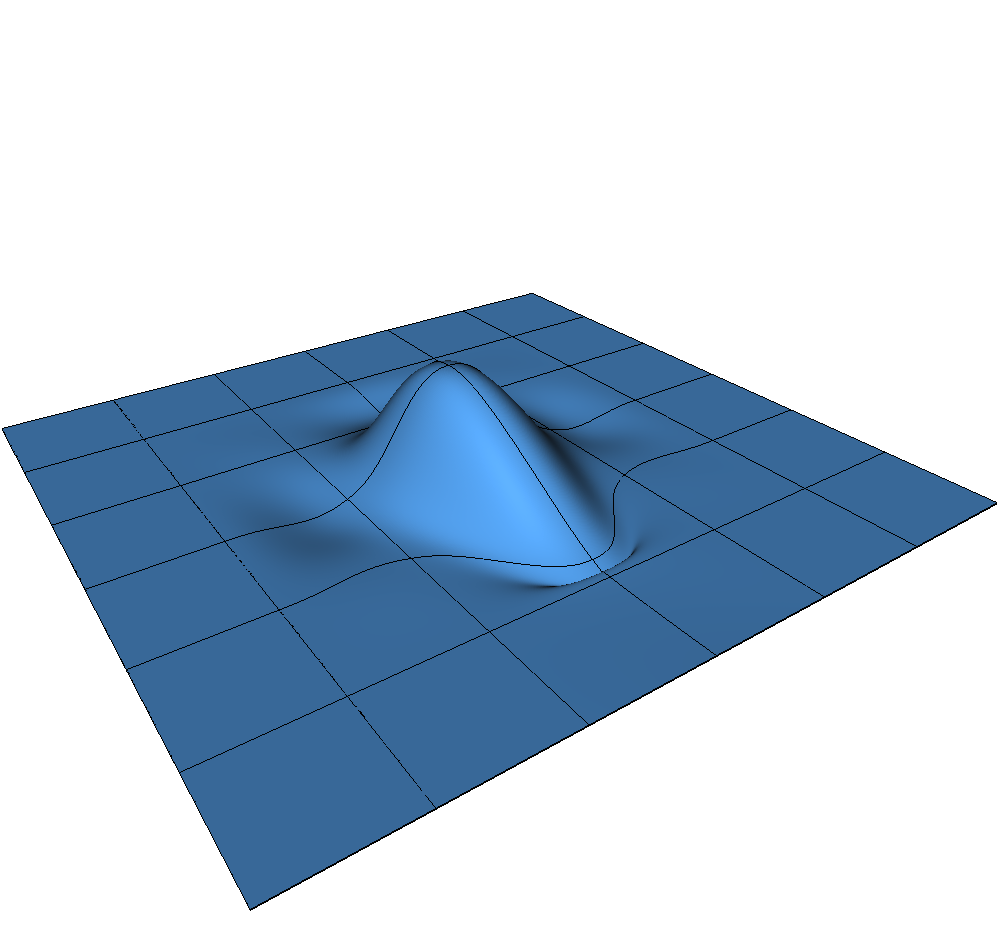}\label{fig:ex3_2}}
\hfill
\subfigure[]
{\includegraphics[width=0.85\textwidth/4]{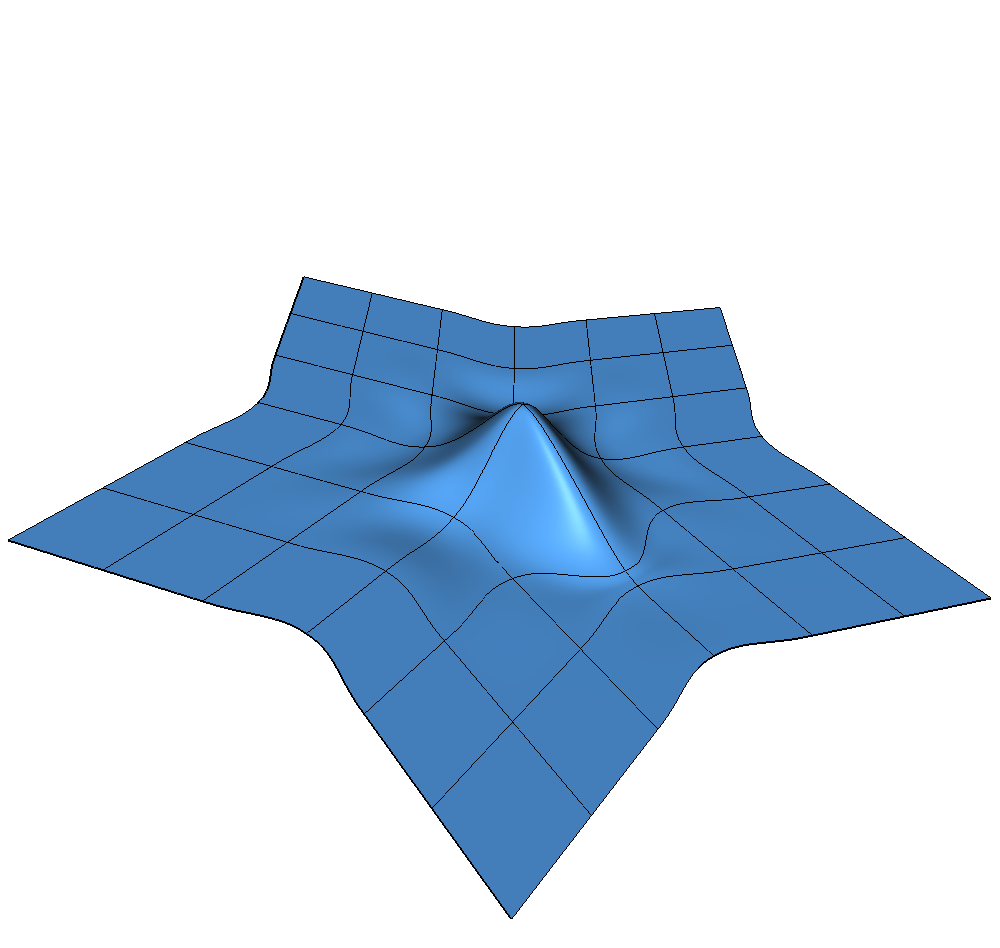}\label{fig:ex3_3}}
\hfill
\subfigure[]
{\raisebox{-0mm}{\includegraphics[width=0.85\textwidth/4]{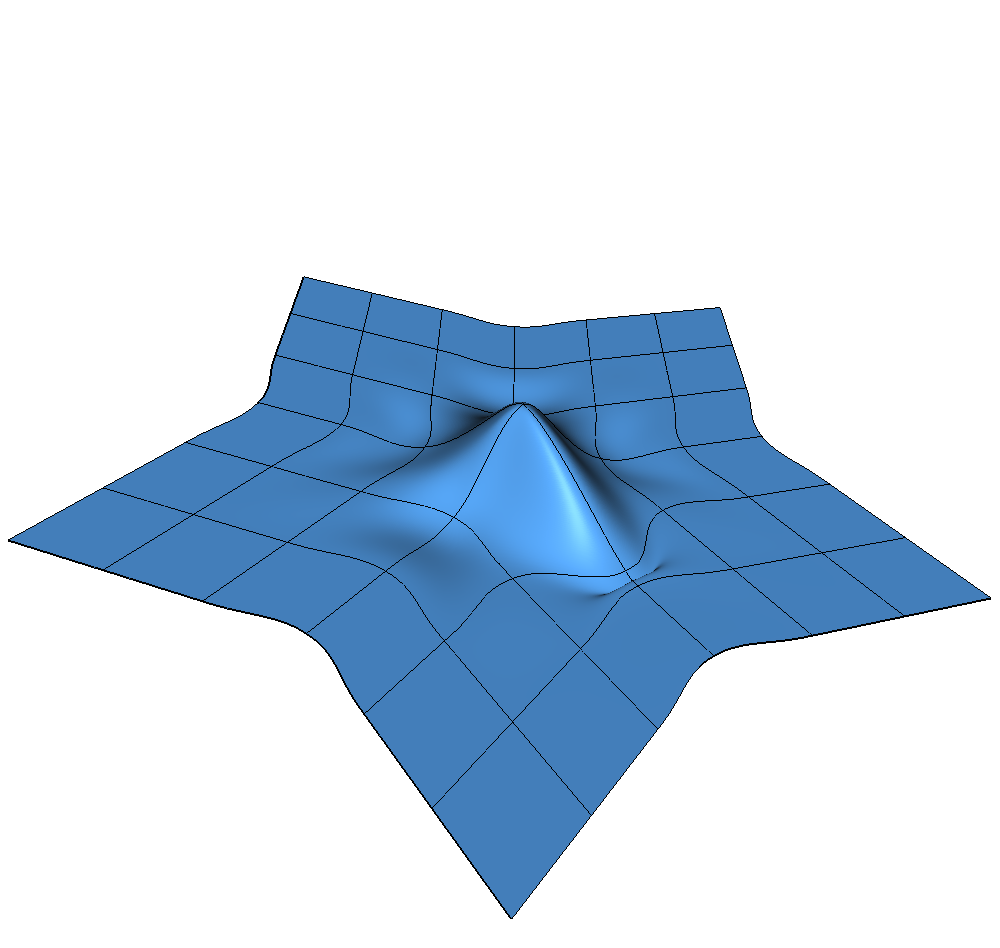}\label{fig:ex3_4}}}
\\
\subfigure[]
{\includegraphics[width=0.85\textwidth/4]{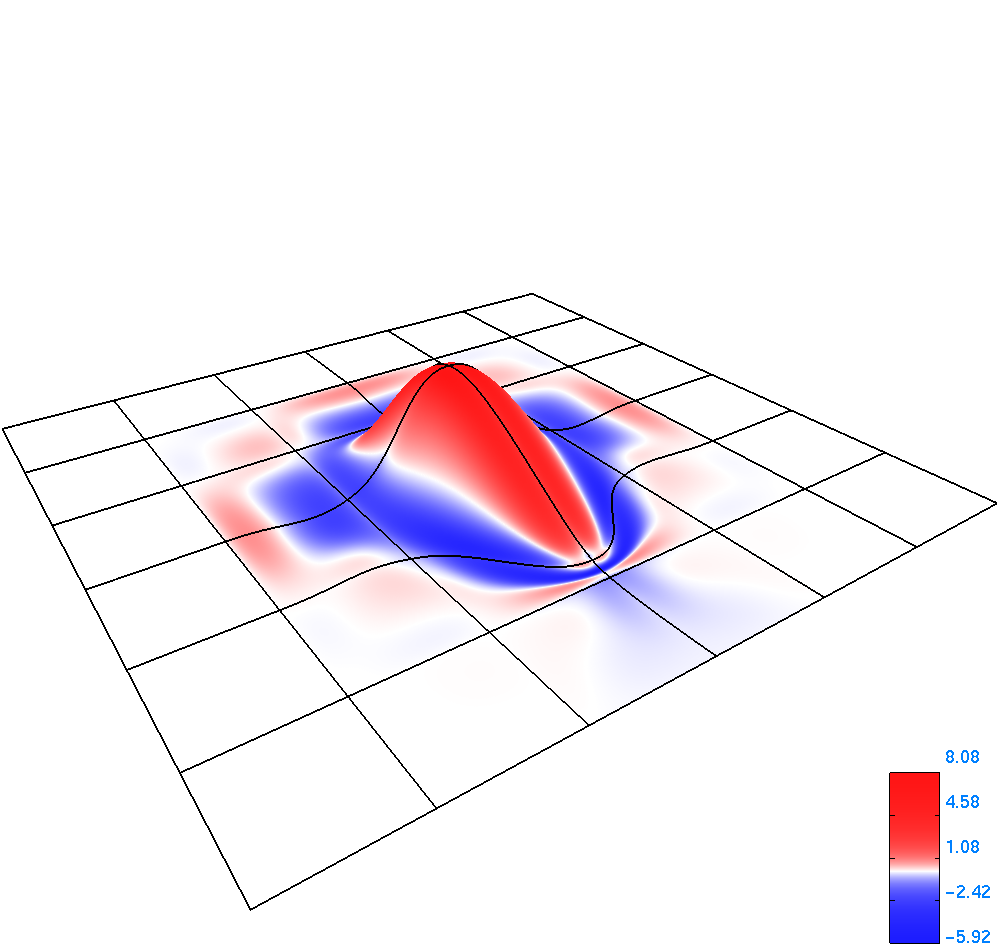}\label{fig:ex3_5}}
\hfill
\subfigure[]
{\includegraphics[width=0.85\textwidth/4]{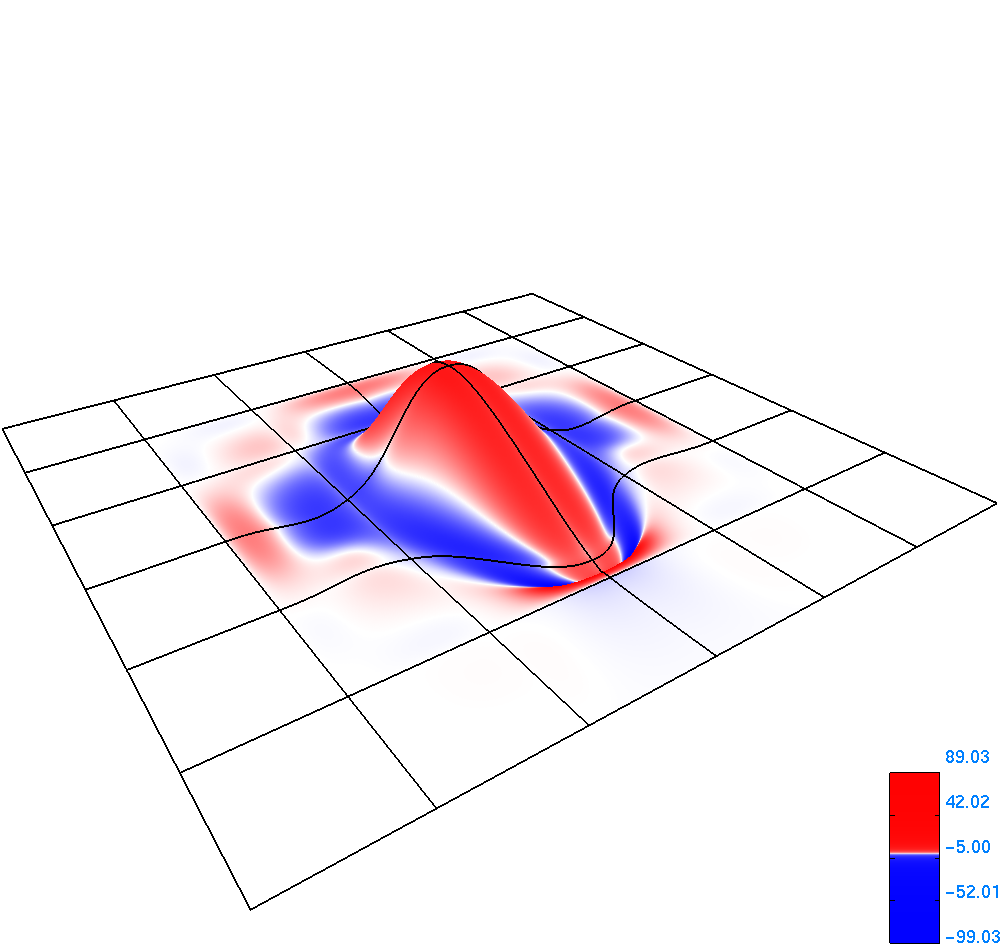}\label{fig:ex3_6}}
\hfill
\subfigure[]
{\includegraphics[width=0.85\textwidth/4]{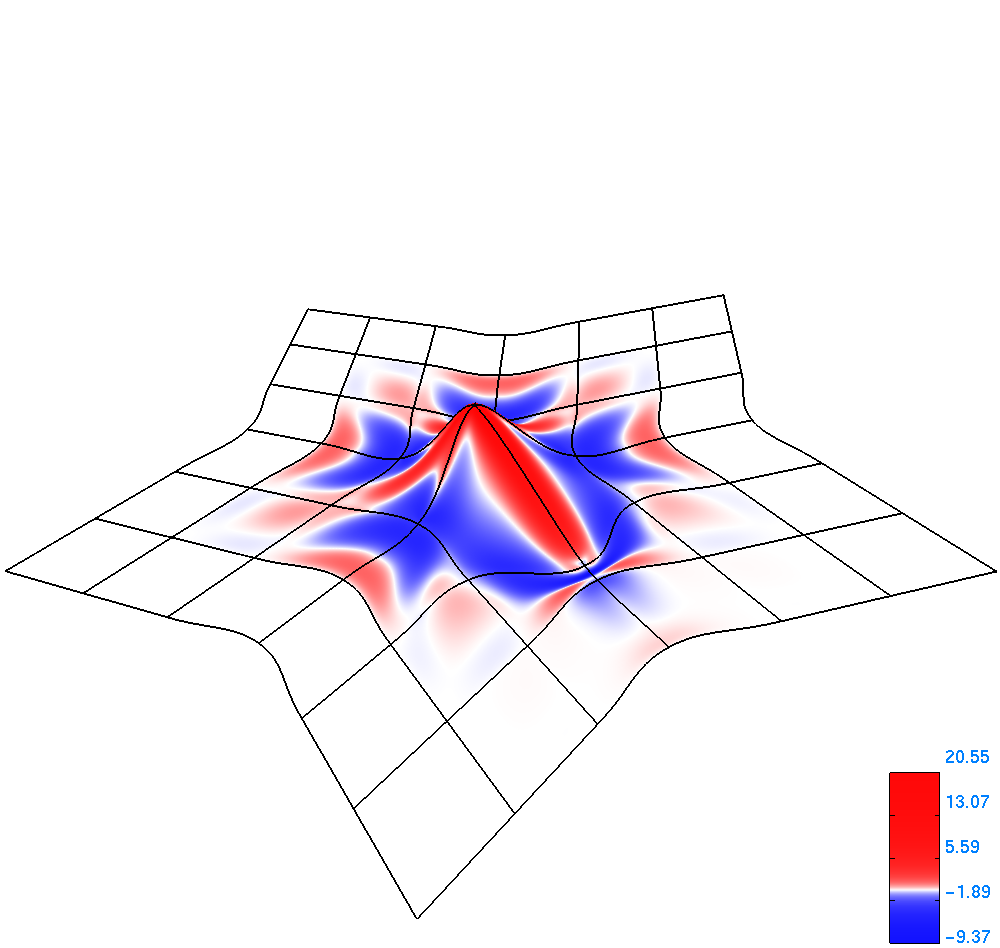}\label{fig:ex3_7}}
\hfill
\subfigure[]
{\includegraphics[width=0.85\textwidth/4]{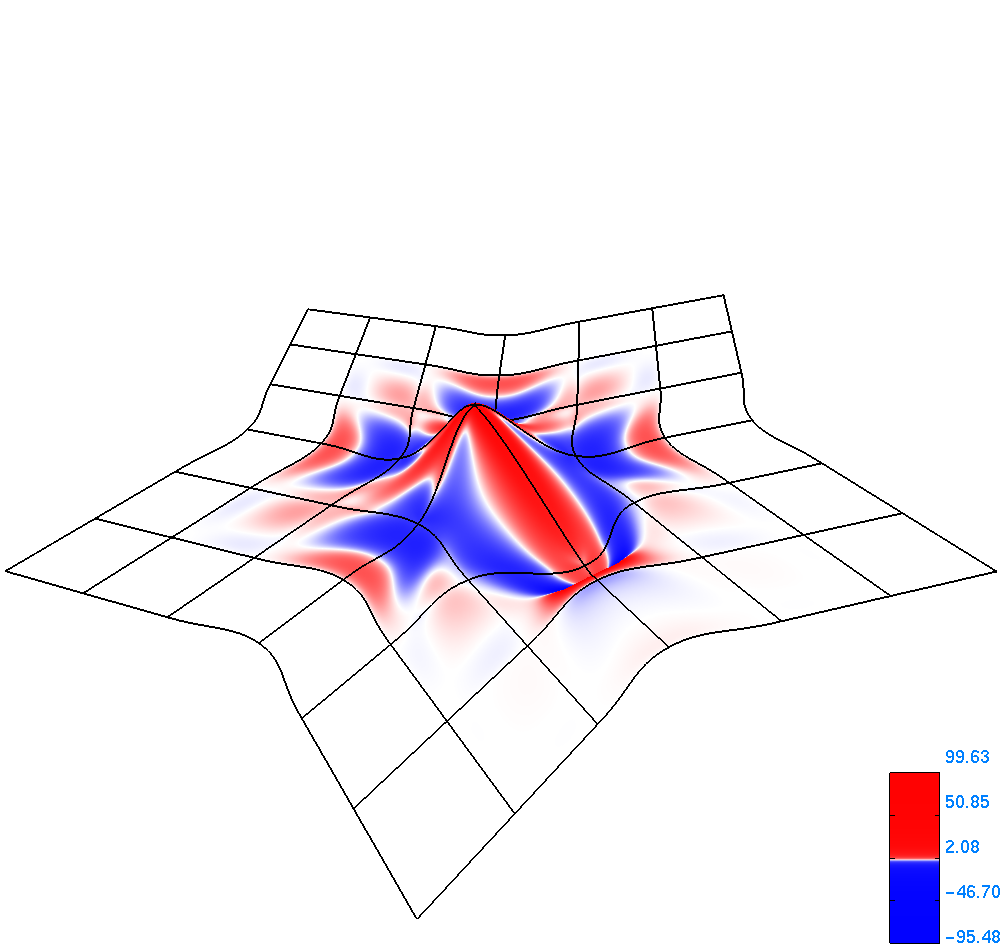}\label{fig:ex3_8}}
\caption{First line: surfaces from the class $D^5C^2P^2S^4$ interpolating the meshes in Figures \ref{fig:ex4_1} and \ref{fig:ex4_2} with augmented parametrization (\subref{fig:ex3_1} and \subref{fig:ex3_3}) and mean parametrization  (\subref{fig:ex3_2} and \subref{fig:ex3_4}).  Second line: corresponding mean curvature graph.}\label{fig:ex3}
\end{figure}

Figure \ref{fig:ex5} illustrates $G^2$ continuous complex surfaces of high quality obtained through the augmented parametrization.
Figures \ref{fig:ex5_2} and \ref{fig:ex5_6} show the regular patches only and the network of section curves.
As discussed in the previous section, the section curves passing through the extraordinary vertices can be thought off as open curves.
The curve segments emanating from the extraordinary vertices are generated as degree 5 polynomials that interpolate endpoint values and derivatives.
As a result, the entire section curves belong to class $D^5C^2P^2S^4$.
The augmented Coons-Gregory patches are illustrated in Figures \ref{fig:ex5_3} and \ref{fig:ex5_7}, whereas
Figures \ref{fig:ex5_4} and \ref{fig:ex5_8} show the entire composite surface.

For the open surface depicted in Figure \ref{fig:ex5}, the boundary patches are regular augmented patches of the form  \eqref{eq:S}.
To compute these patches, we have suitably extrapolated the mesh structure across the boundary in order to obtain an additional layer of faces and vertices.

Figure \ref{fig:ex7} compares the results of the augmented and mean parameterizations for the meshes in Figure \ref{fig:ex5} and emphasizes the better quality of the augmented surfaces.
In particular, in the extraordinary region, the curvature of the surface with mean parametrization oscillates several times between positive and negative values, while this does not happen when the augmented parametrization is used.

\begin{figure}[t]
\centering
\subfigure[Input mesh]
{\includegraphics[width=0.95\textwidth/4]{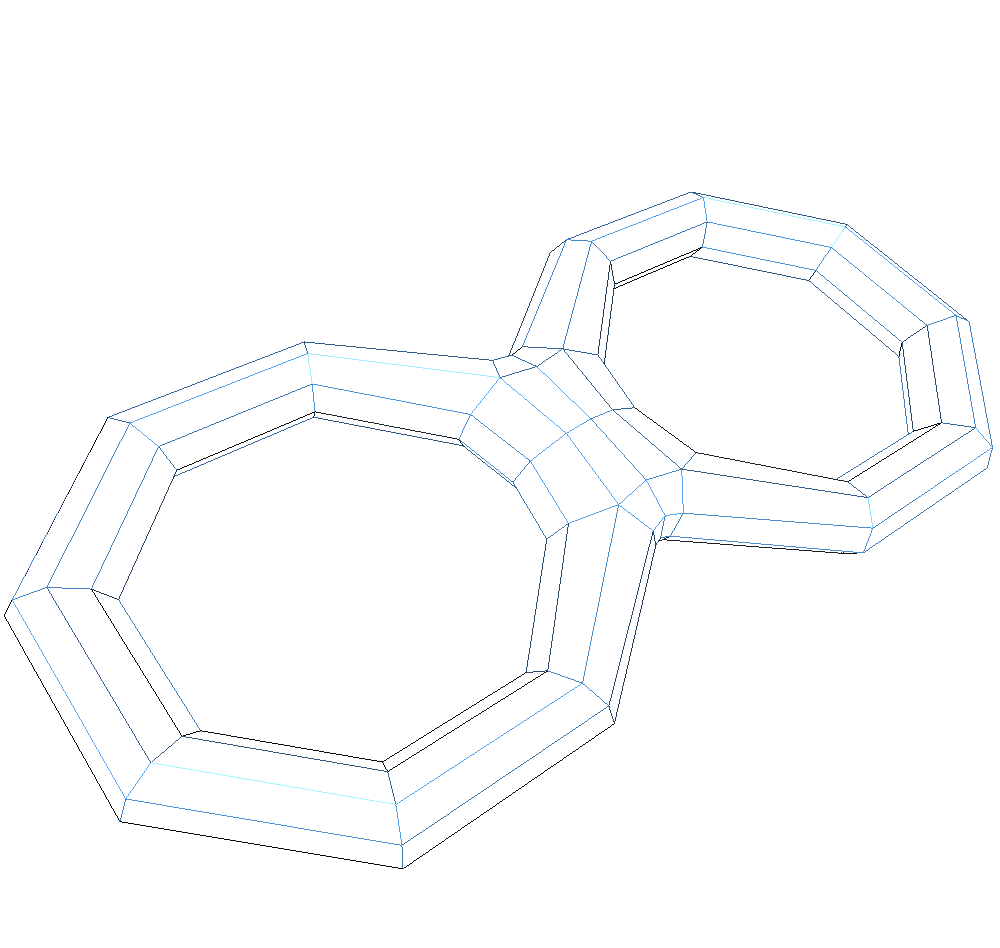}\label{fig:ex5_1}}
\hfill
\subfigure[Regular patches]
{\includegraphics[width=0.95\textwidth/4]{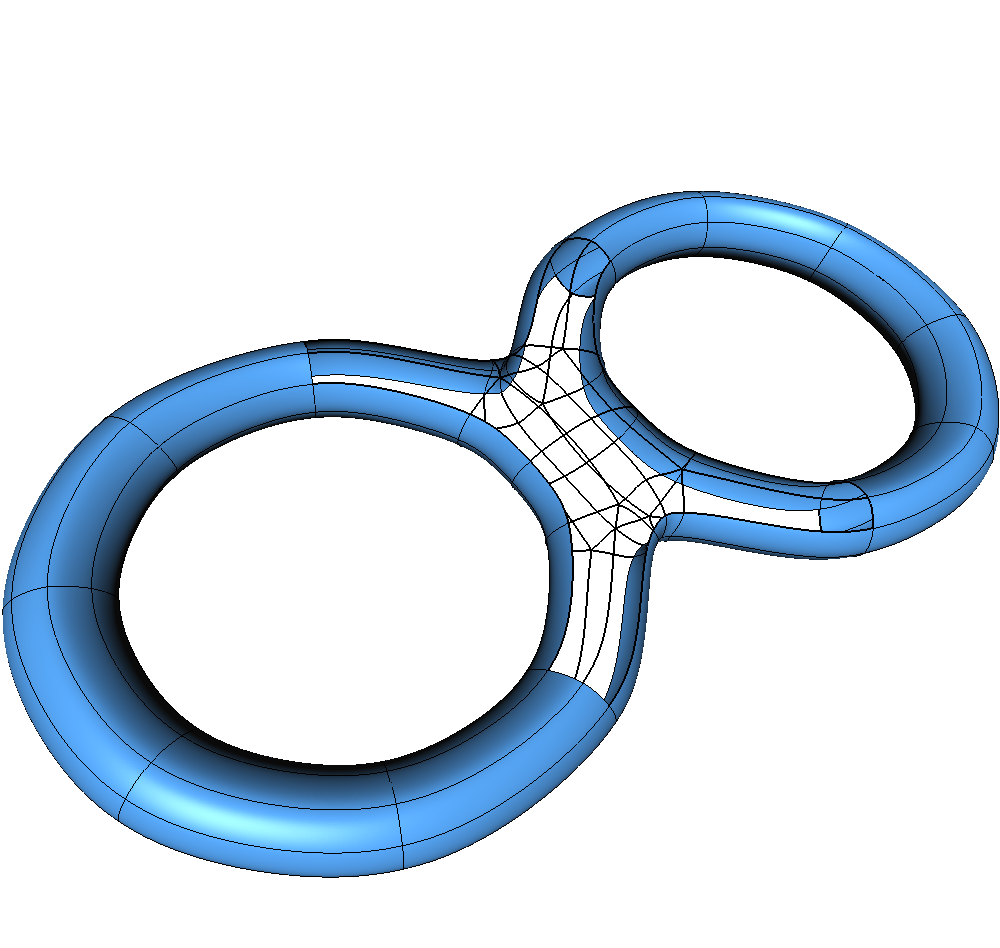}\label{fig:ex5_2}}
\hfill
\subfigure[Extraordinary patches]
{\includegraphics[width=0.95\textwidth/4]{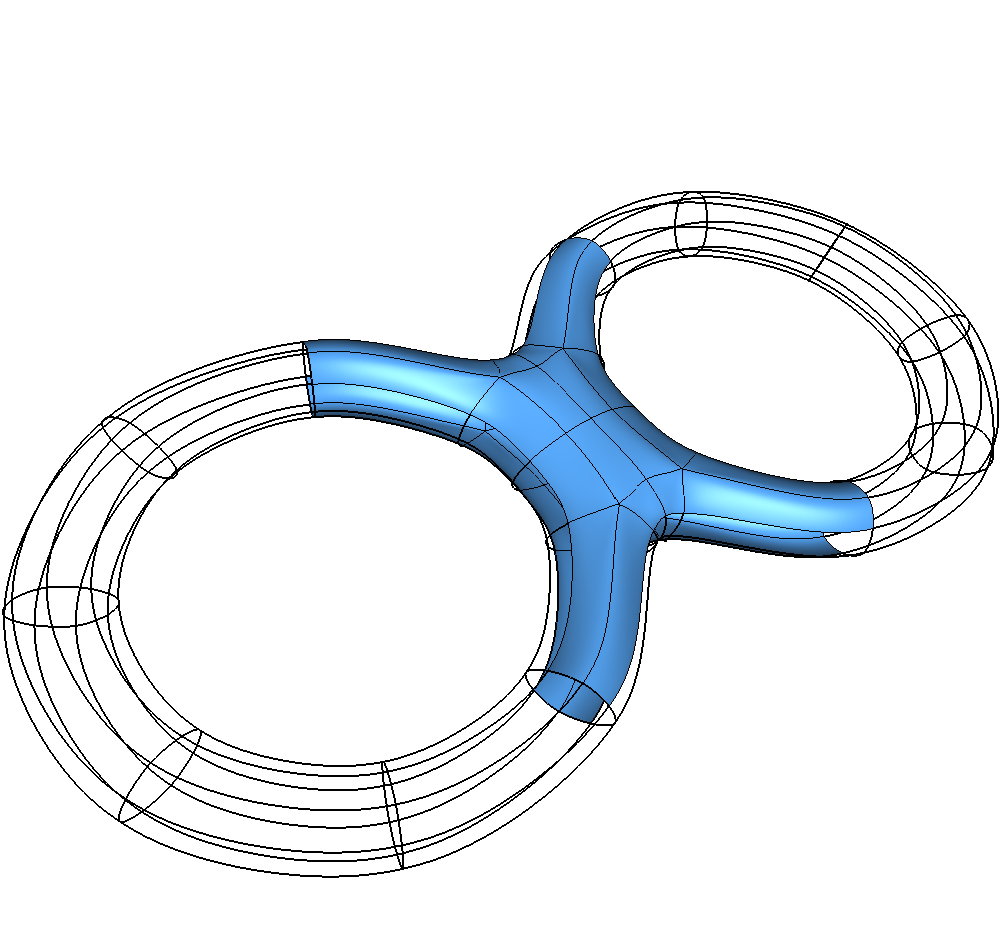}\label{fig:ex5_3}}
\hfill
\subfigure[Entire composite surface]
{\raisebox{-0mm}{\includegraphics[width=0.95\textwidth/4]{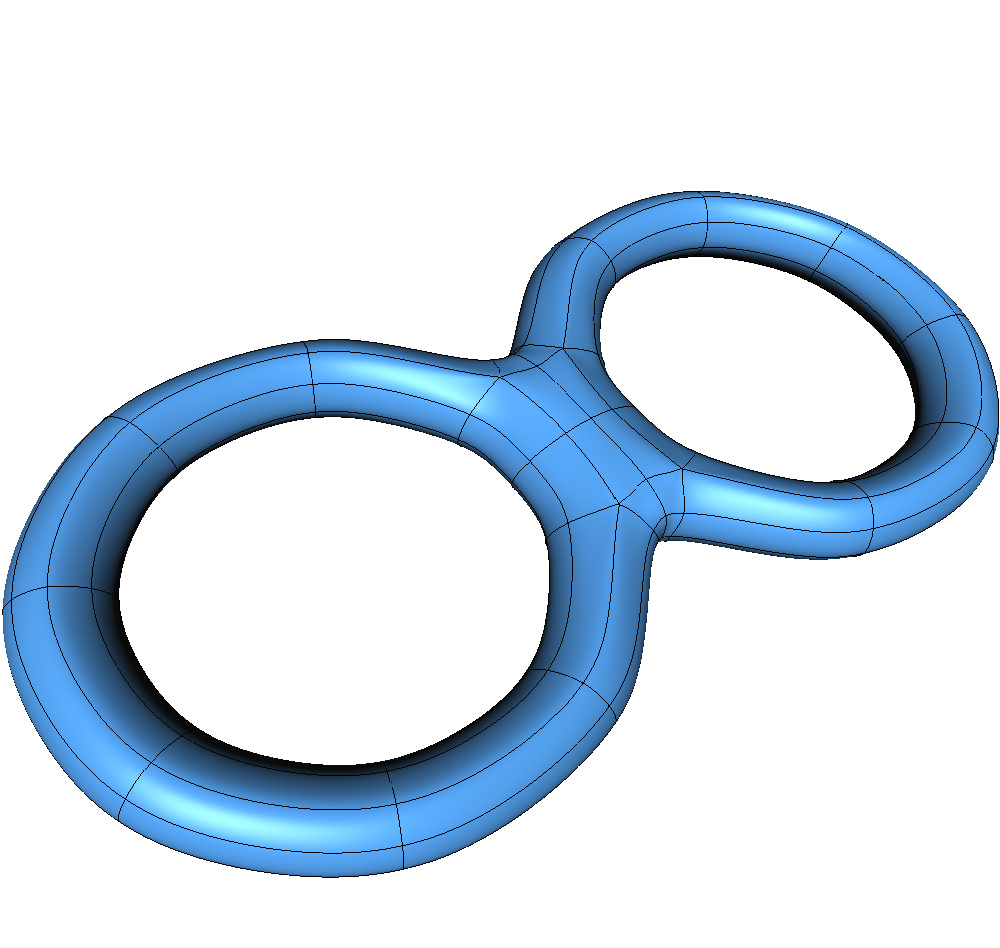}\label{fig:ex5_4}}}
\\
\subfigure[Input mesh]
{\includegraphics[width=0.95\textwidth/4]{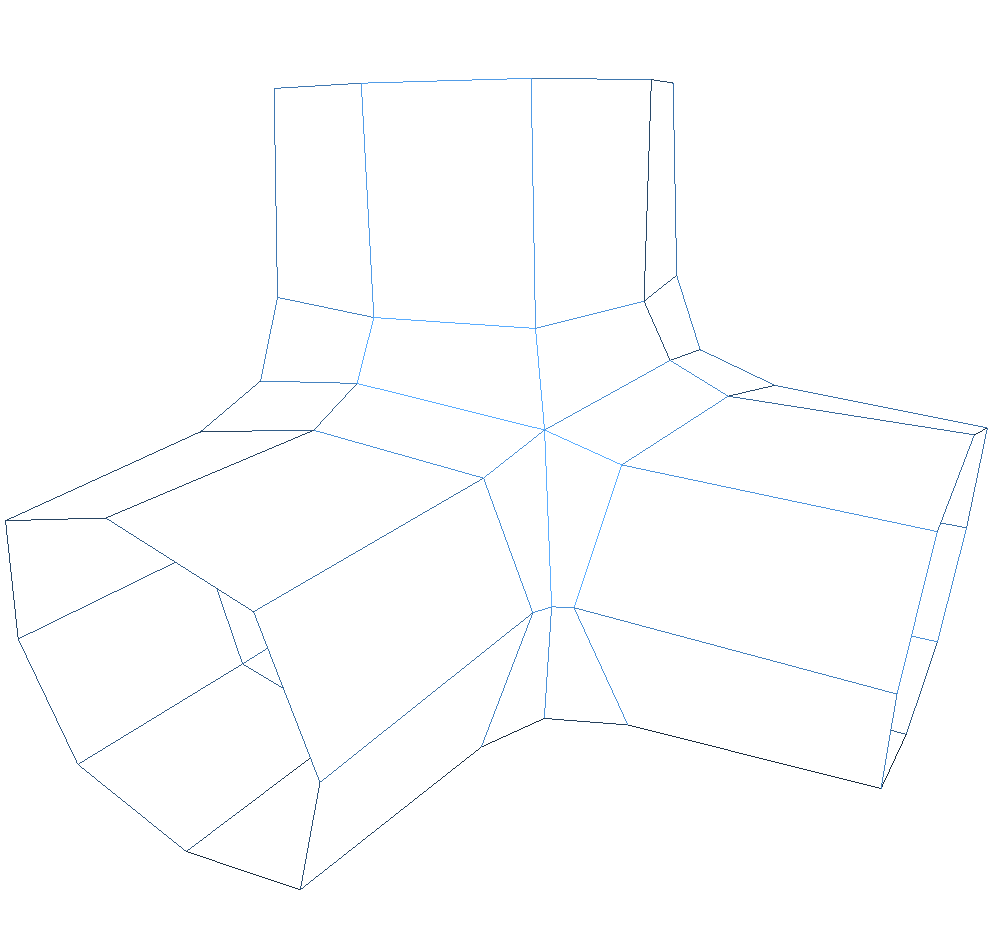}\label{fig:ex5_5}}
\hfill
\subfigure[Regular patches]
{\includegraphics[width=0.95\textwidth/4]{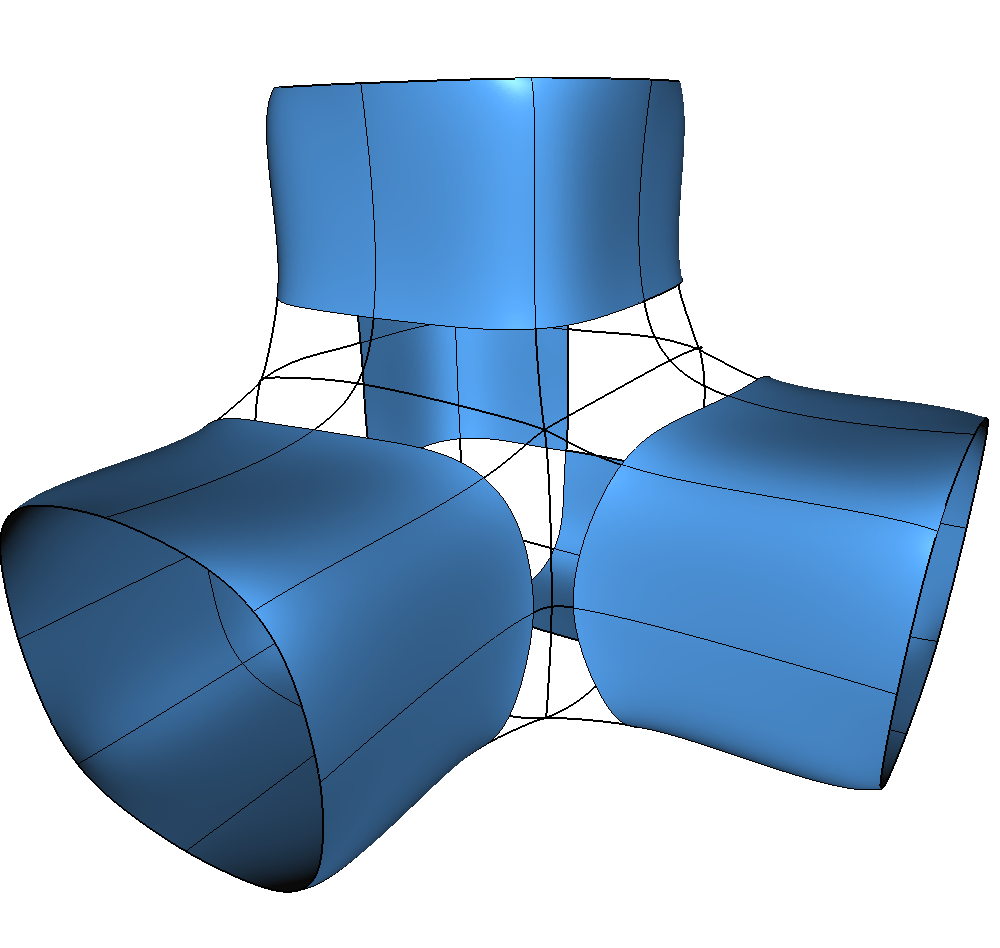}\label{fig:ex5_6}}
\hfill
\subfigure[Extraordinary patches]
{\includegraphics[width=0.95\textwidth/4]{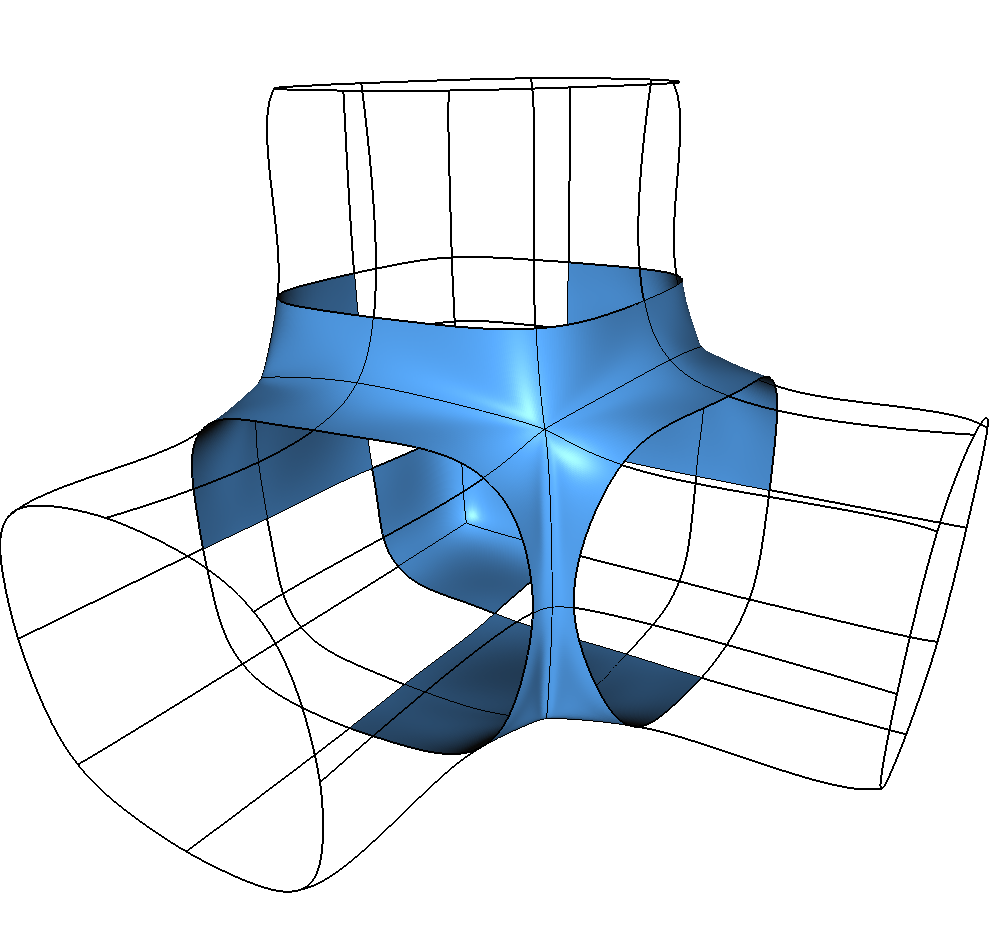}\label{fig:ex5_7}}
\hfill
\subfigure[Entire composite surface]
{\includegraphics[width=0.95\textwidth/4]{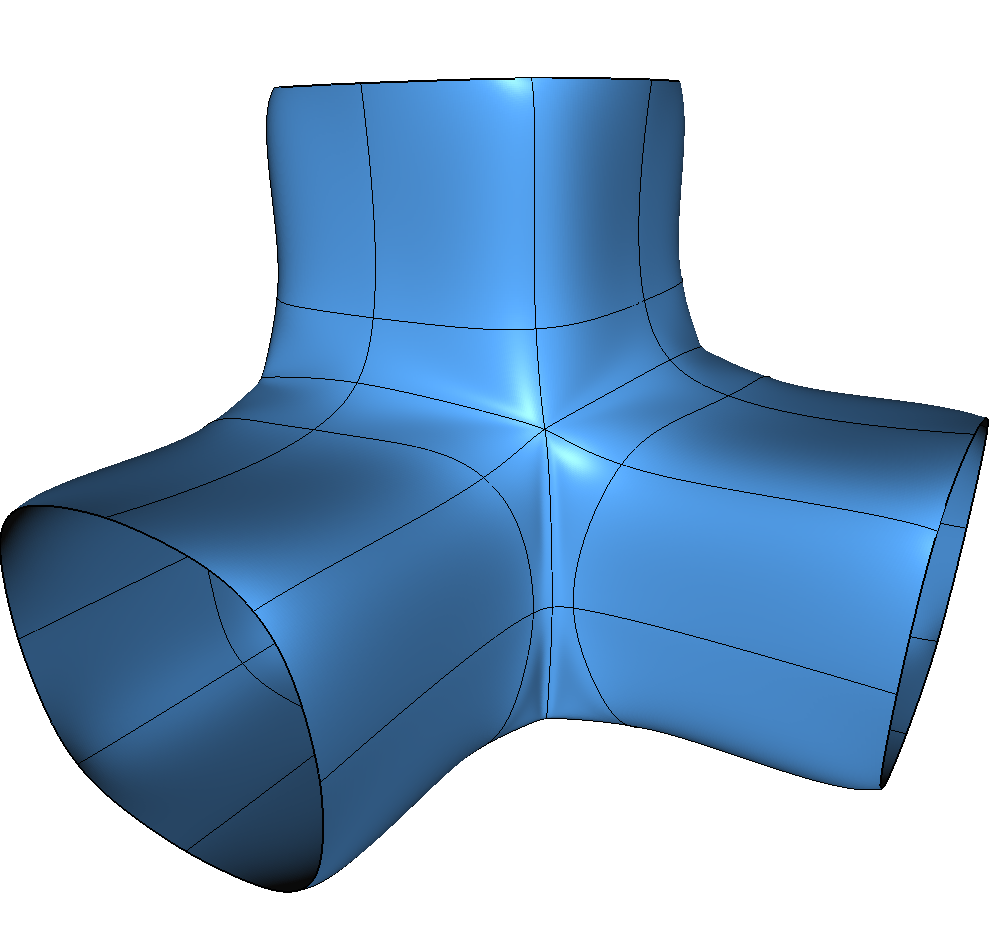}\label{fig:ex5_8}}
\caption{Surfaces from the class $D^5C^2P^2S^4$ with augmented parametrization.}\label{fig:ex5}
\end{figure}

\begin{figure}[t]
\centering
\subfigure[Mean parametrization]
{\includegraphics[width=0.95\textwidth/4]{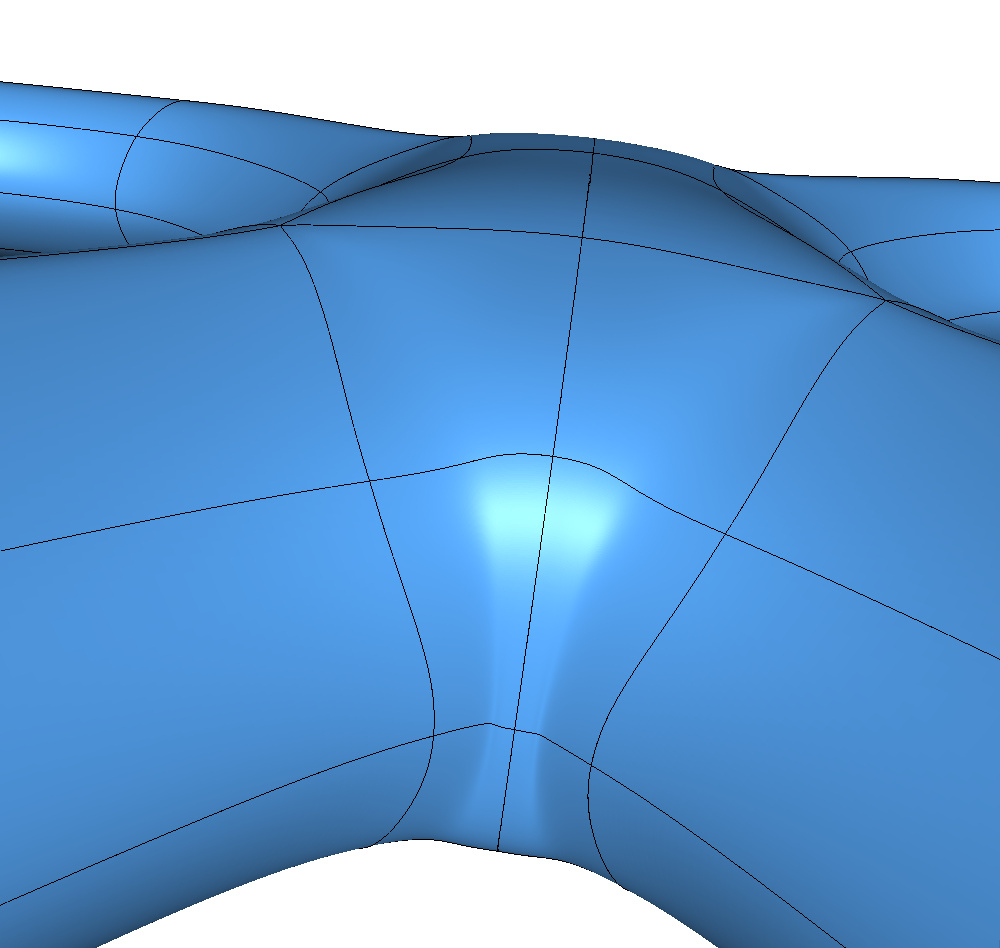}\label{fig:ex7_1}}
\hfill
\subfigure[Augmented parametrization]
{\includegraphics[width=0.95\textwidth/4]{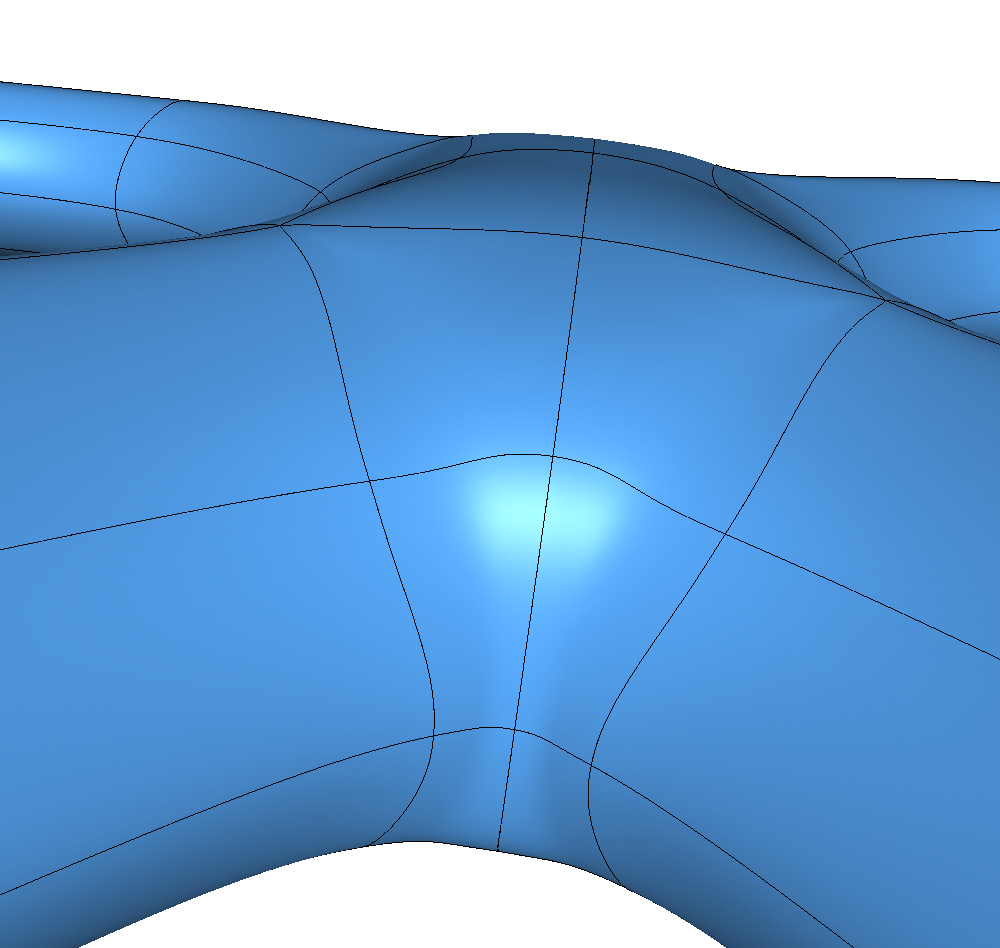}\label{fig:ex7_2}}
\hfill
\subfigure[Mean parametrization]
{\includegraphics[width=0.95\textwidth/4]{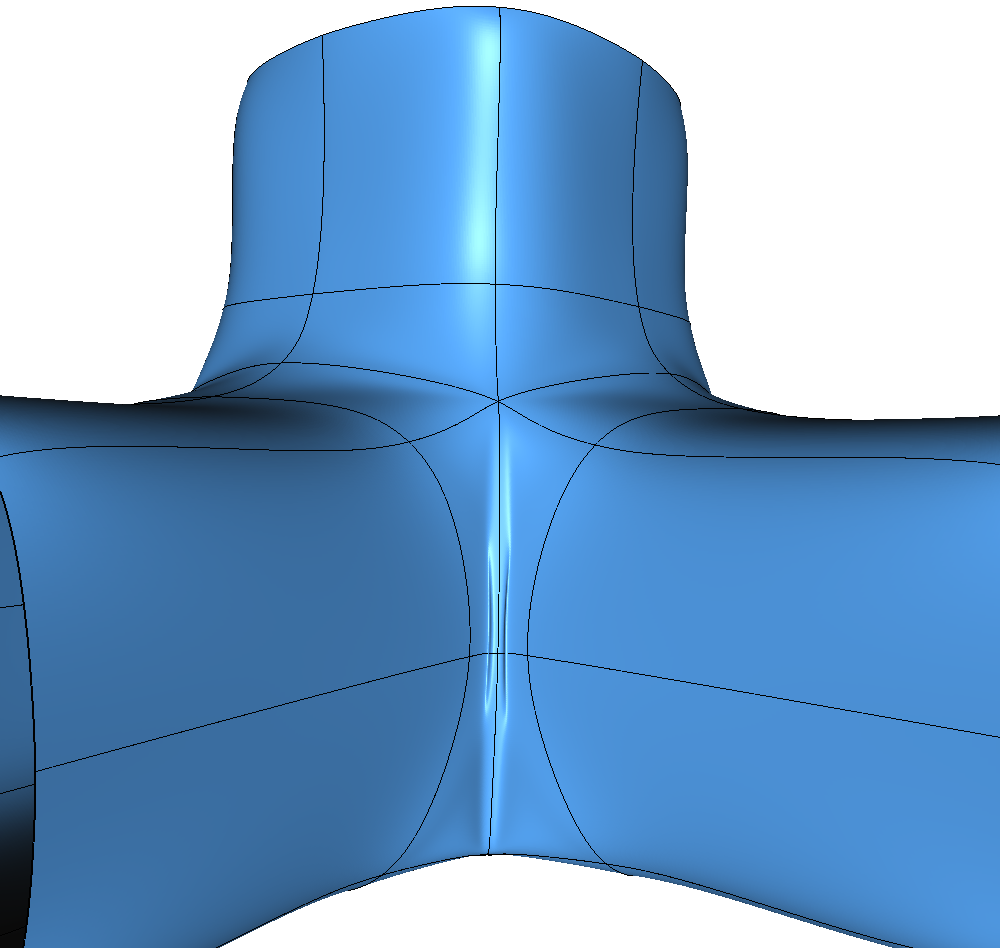}\label{fig:ex7_5}}
\hfill
\subfigure[Augmented parametrization]
{\includegraphics[width=0.95\textwidth/4]{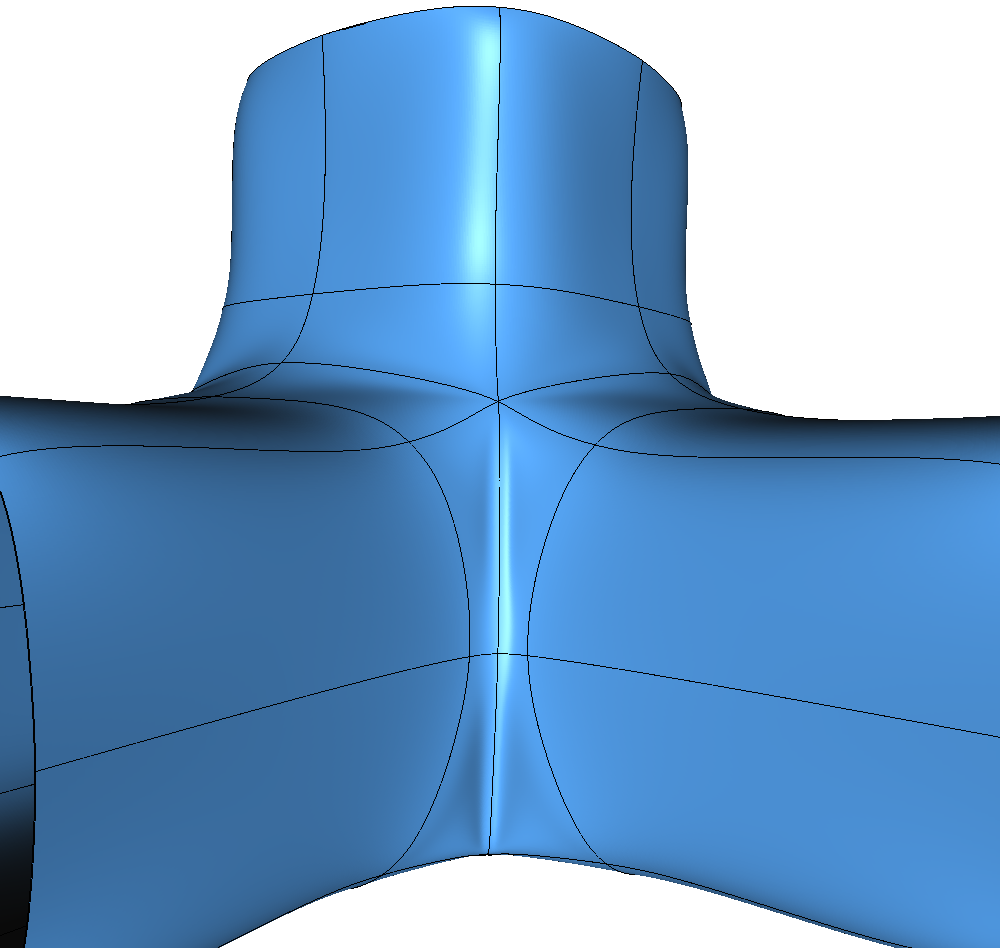}\label{fig:ex7_6}}\\
\subfigure[Mean curvature of \ref{fig:ex7_1}]
{\includegraphics[width=0.95\textwidth/4]{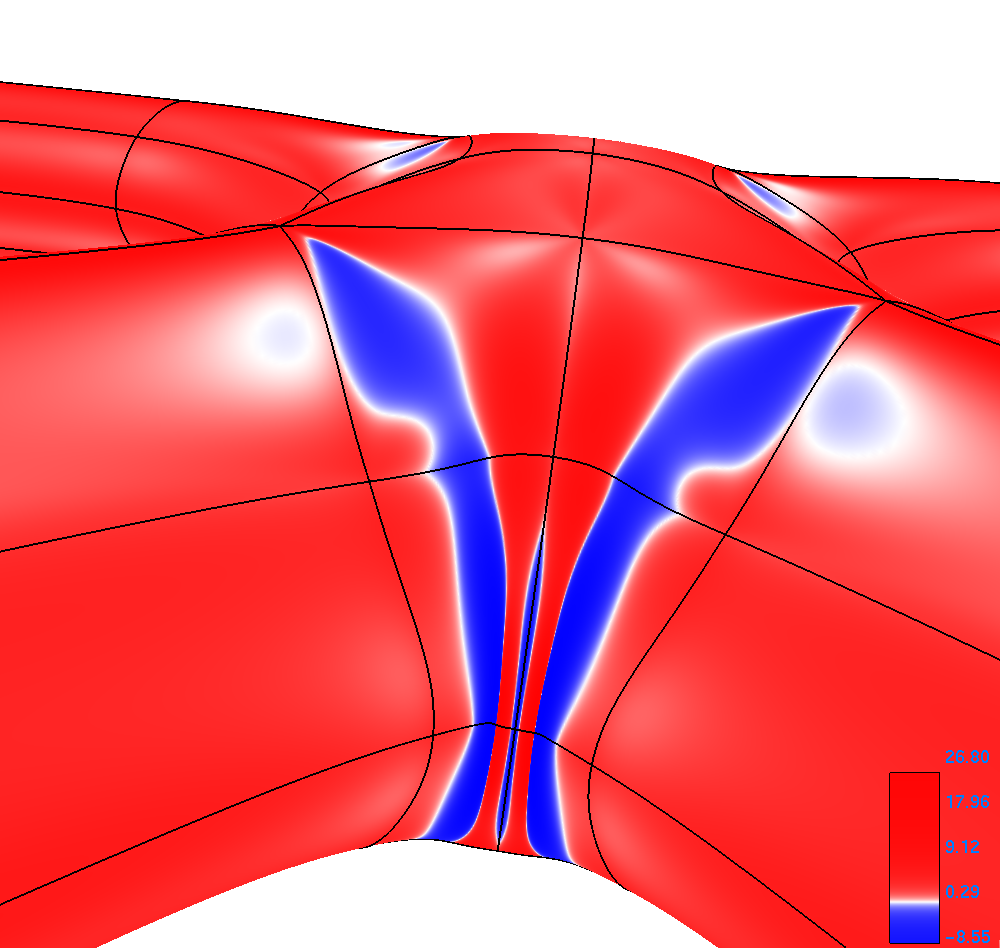}\label{fig:ex7_3}}
\hfill
\subfigure[Mean curvature of \ref{fig:ex7_2}]
{\raisebox{-0mm}{\includegraphics[width=0.95\textwidth/4]{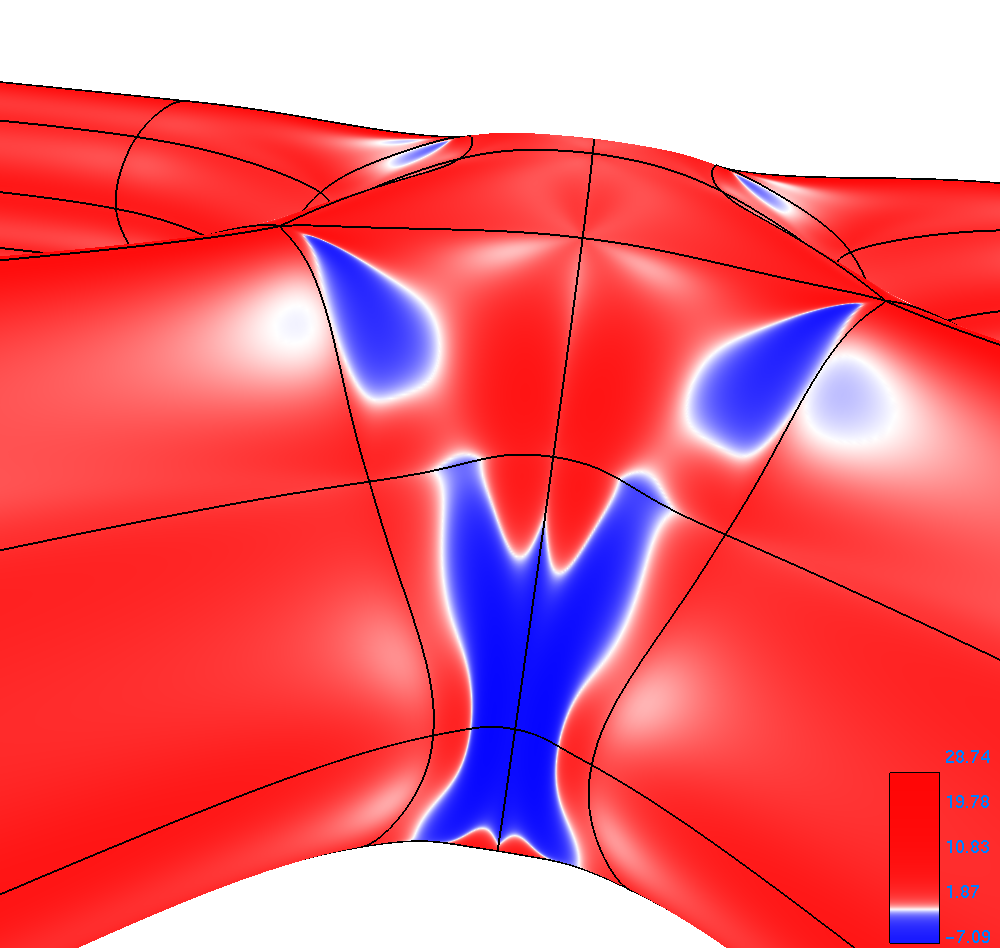}\label{fig:ex7_4}}}
\hfill
\subfigure[Mean curvature of \ref{fig:ex7_5}]
{\includegraphics[width=0.95\textwidth/4]{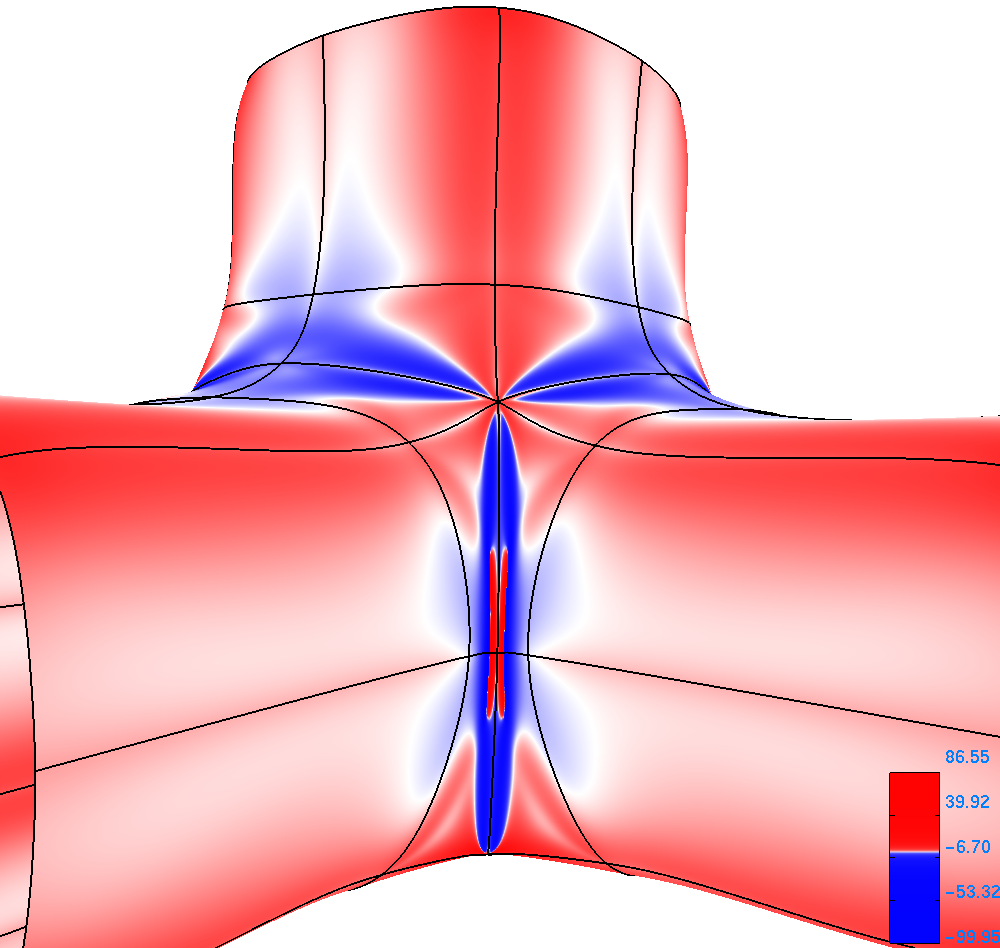}\label{fig:ex7_7}}
\hfill
\subfigure[Mean curvature of \ref{fig:ex7_6}]
{\includegraphics[width=0.95\textwidth/4]{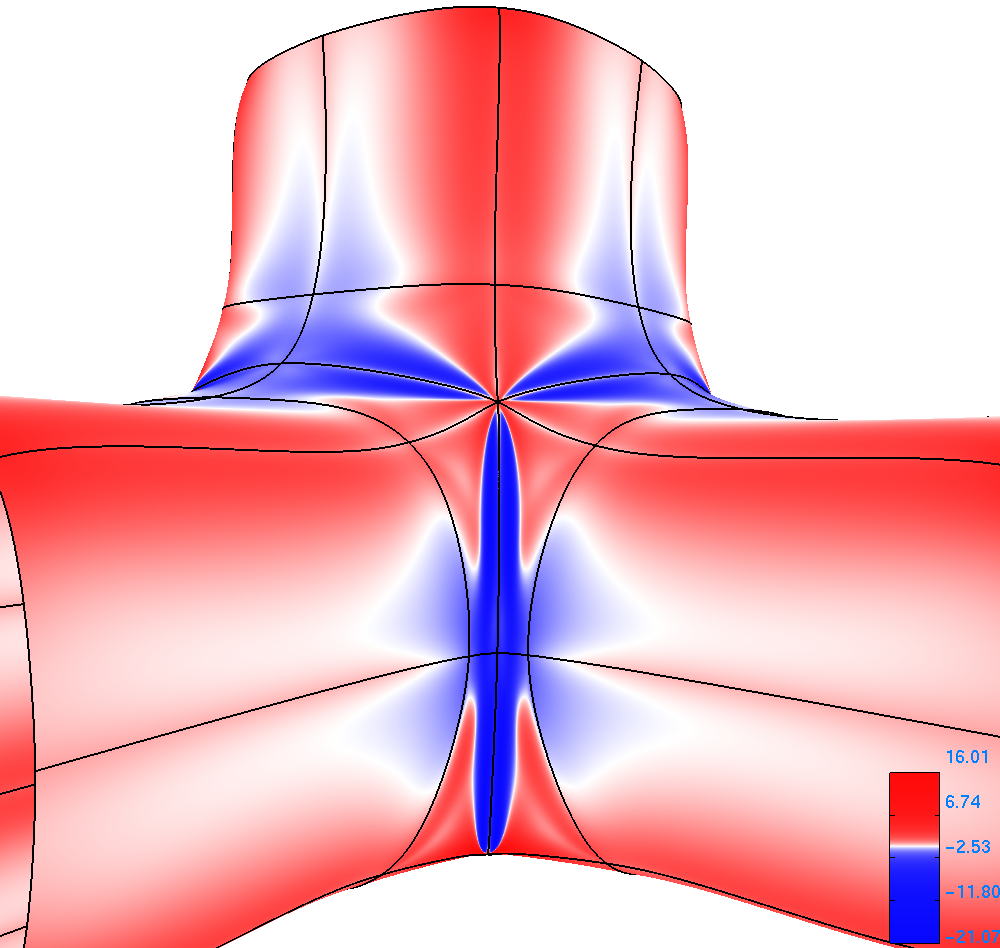}\label{fig:ex7_8}}
\caption{Zoom of the surfaces obtained from the meshes in Figure \ref{fig:ex5} with augmented and mean parametrization.}\label{fig:ex7}
\end{figure}

Finally, Figure \ref{fig:ex6} depicts some challenging meshes for interpolation and the corresponding augmented interpolatory surfaces.
The augmented surfaces are globally $G^2$, free of unwanted artifacts and overall approximate in a reasonable way the input mesh.
This is a highly nontrivial result for a local interpolation method.

\begin{figure}[t]
\centering
\subfigure[Input mesh]
{\includegraphics[width=0.85\textwidth/4]{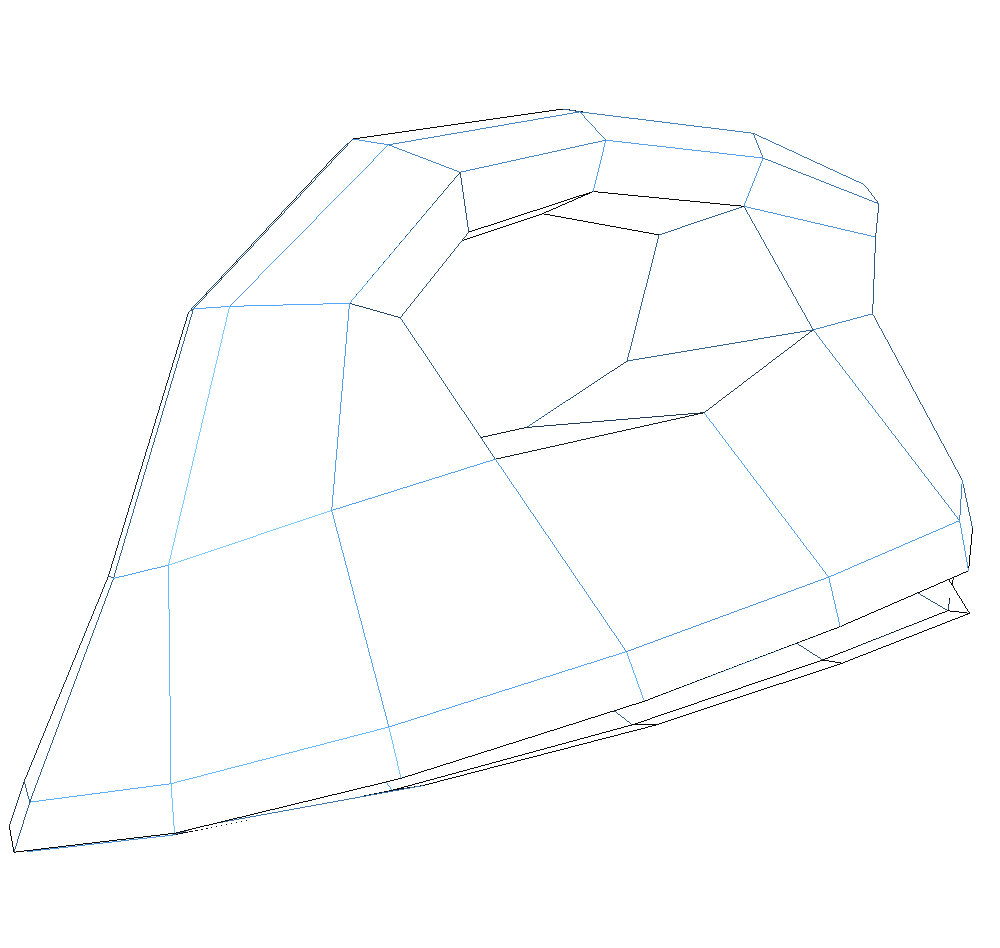}\label{fig:ex6_1}}
\hfill
\subfigure[Local interpolatory surface]
{\includegraphics[width=0.85\textwidth/4]{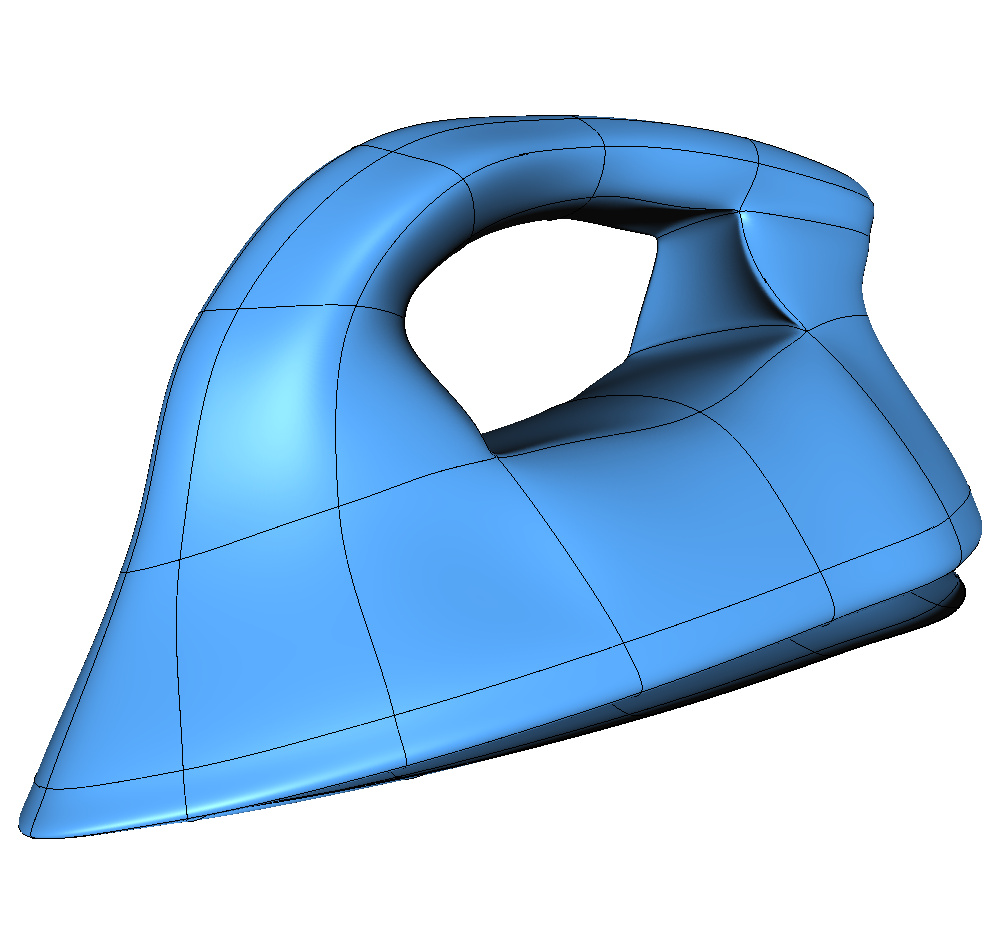}\label{fig:ex6_2}}
\hfill
\subfigure[Mean curvature]
{\includegraphics[width=0.85\textwidth/4]{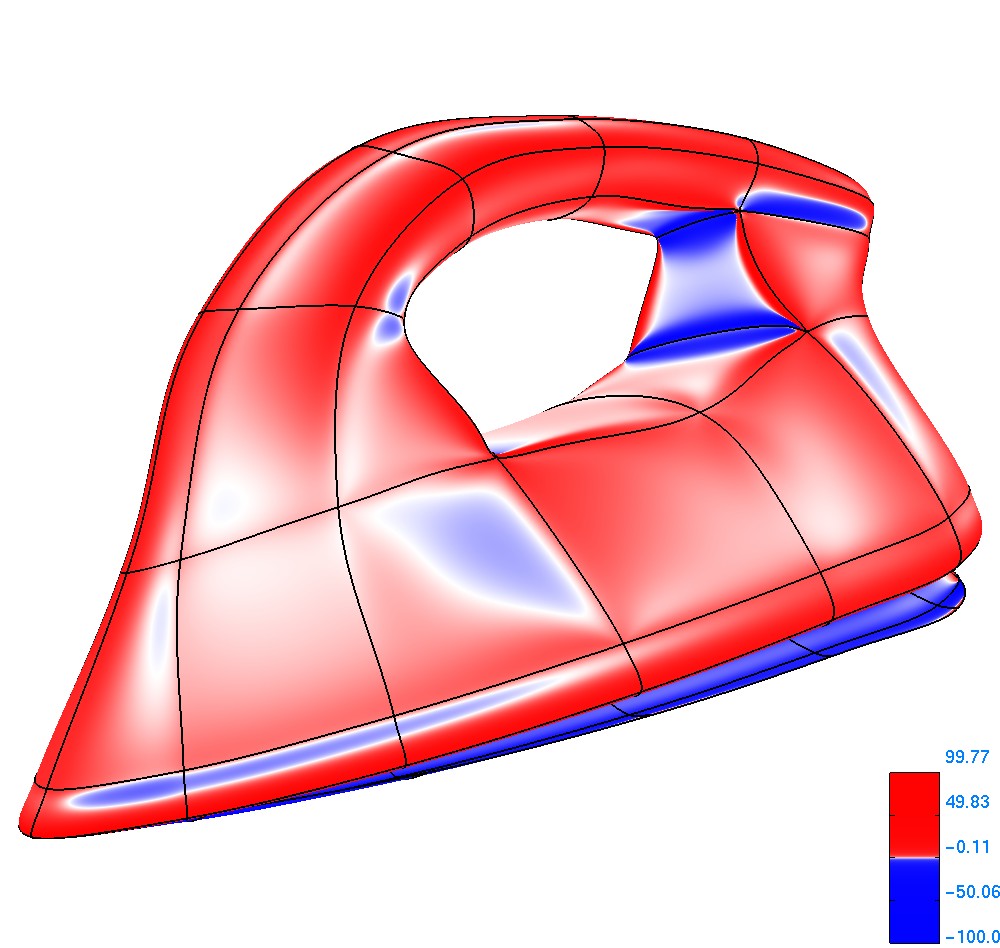}\label{fig:ex6_3}}
\hfill
\subfigure[Isophotes]
{\raisebox{-0mm}{\includegraphics[width=0.85\textwidth/4]{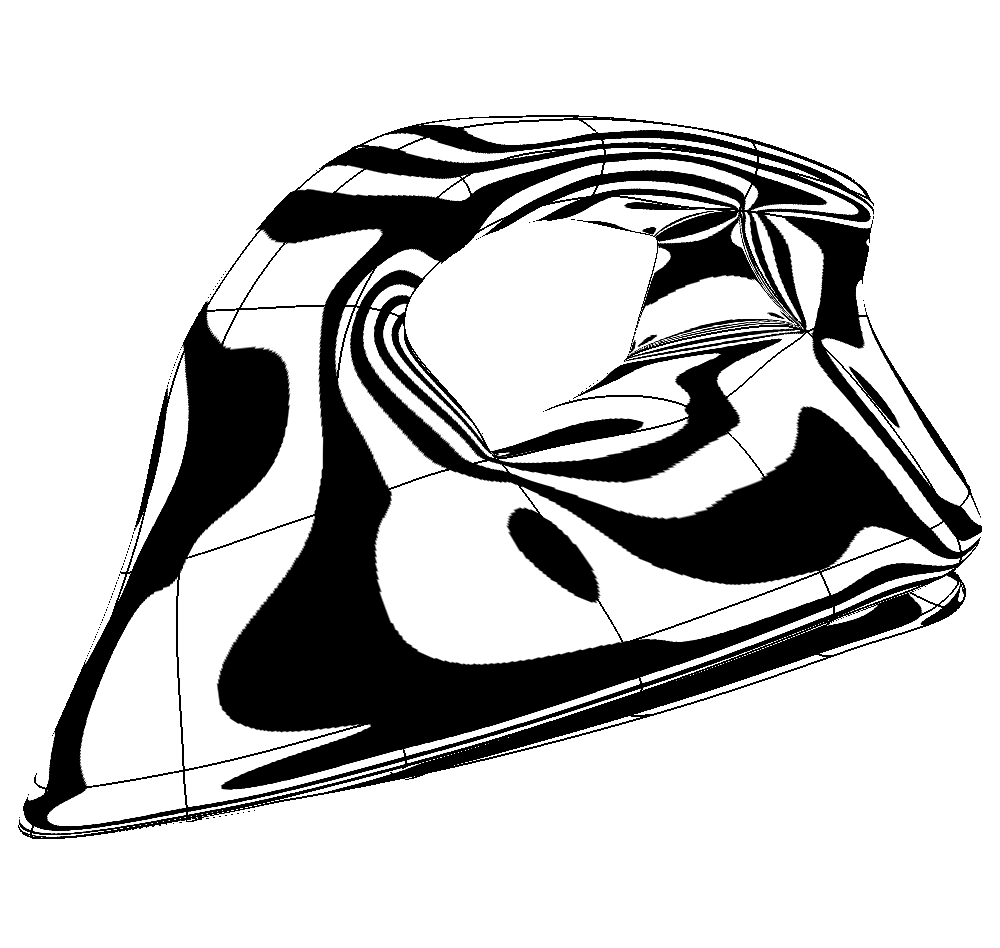}\label{fig:ex6_4}}}
\\
\subfigure[Input mesh]
{\includegraphics[width=0.95\textwidth/4]{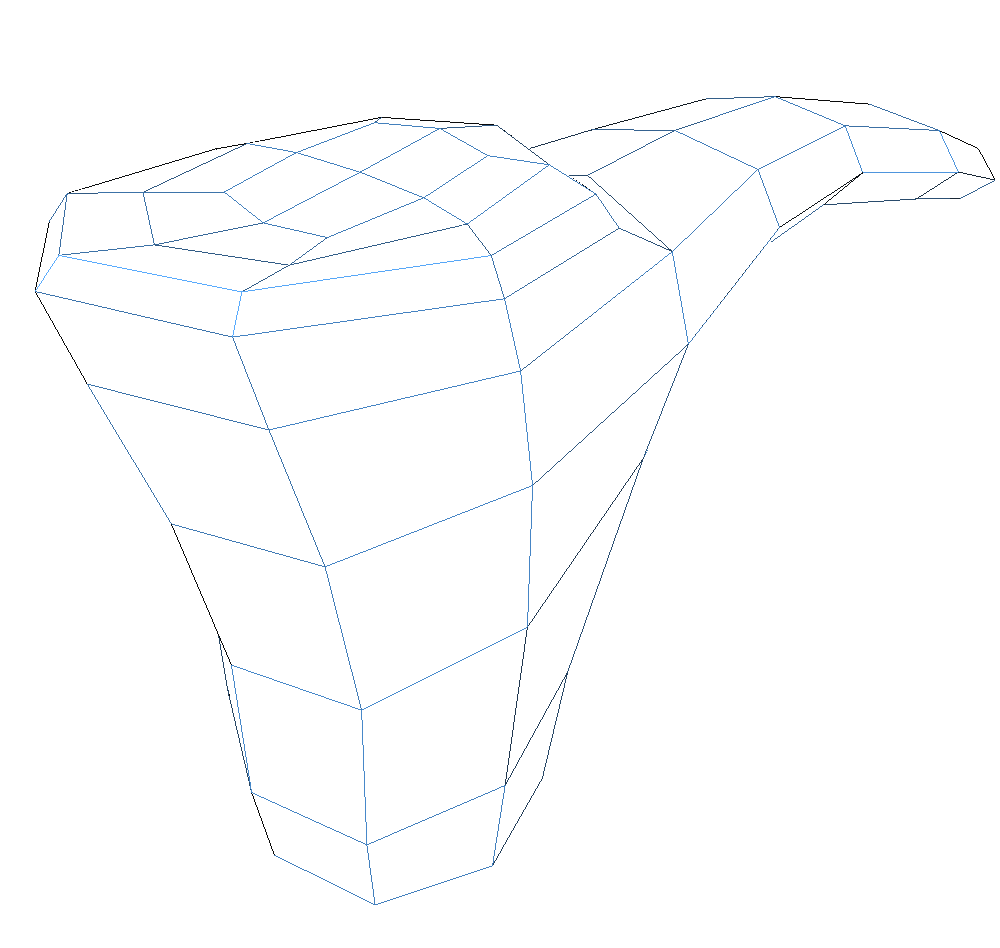}\label{fig:ex6_5}}
\hfill
\subfigure[Local interpolatory surface]
{\includegraphics[width=0.95\textwidth/4]{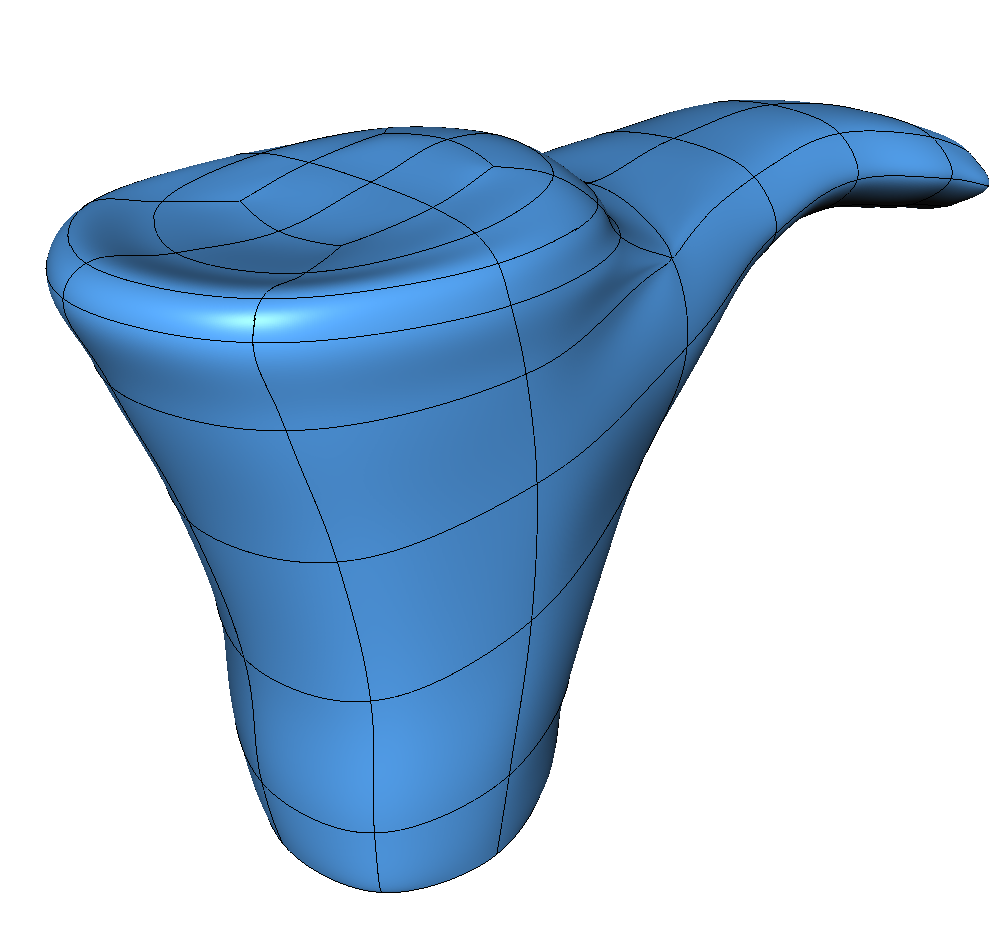}\label{fig:ex6_6}}
\hfill
\subfigure[Mean curvature]
{\includegraphics[width=0.95\textwidth/4]{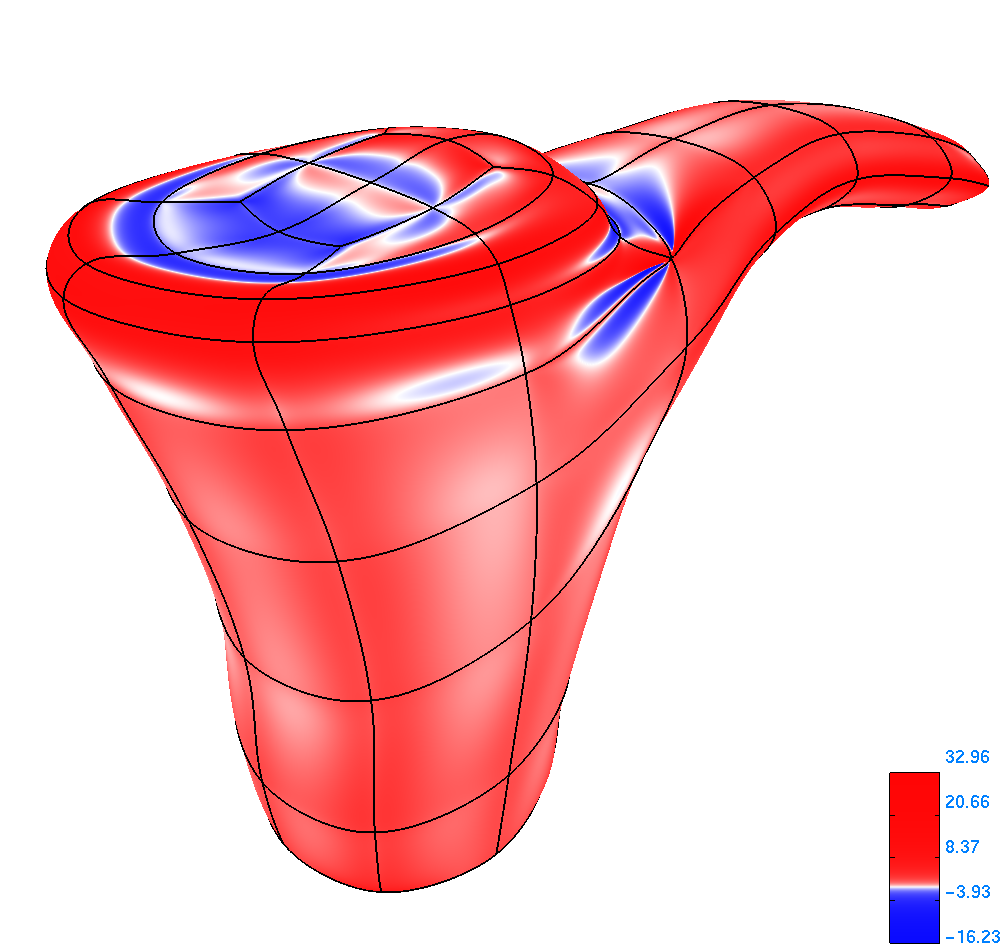}\label{fig:ex6_7}}
\hfill
\subfigure[Isophotes]
{\includegraphics[width=0.95\textwidth/4]{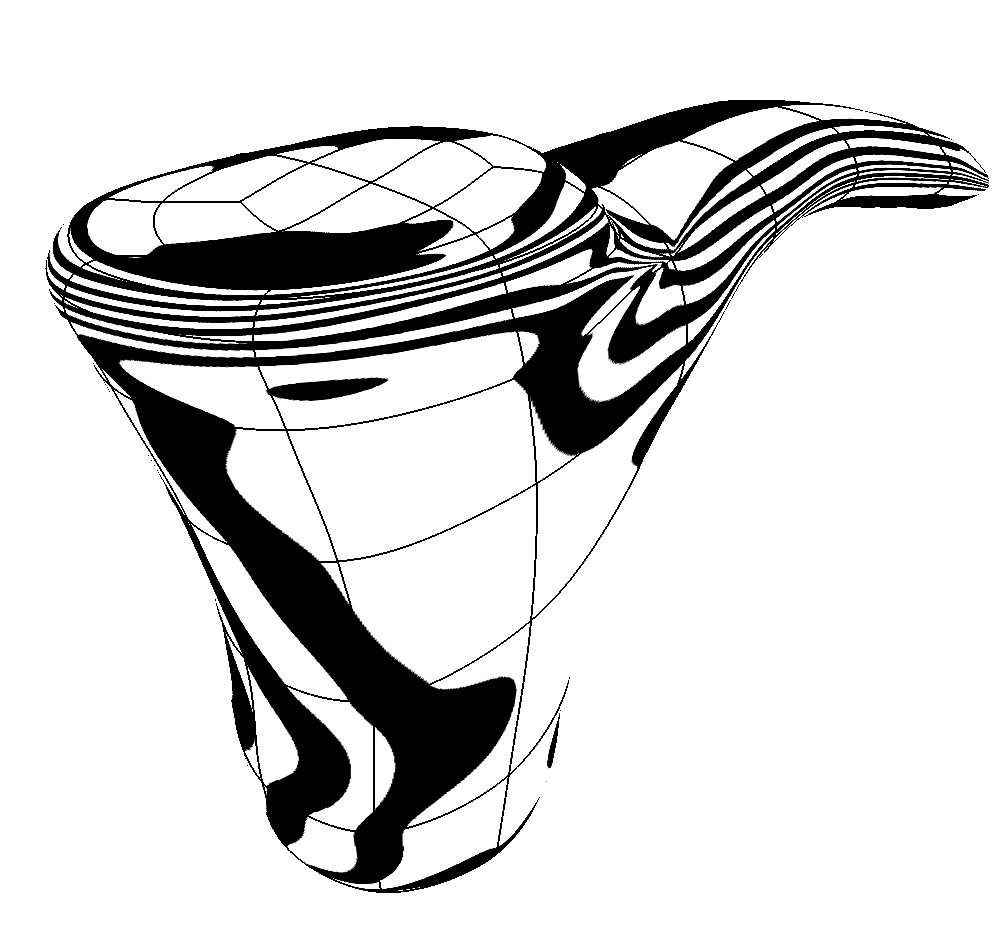}\label{fig:ex6_8}}
\caption{Surfaces from the class $D^5C^2P^2S^4$ with augmented parameterization.}\label{fig:ex6}
\end{figure}
\section{Conclusion}
\label{sec:conclusion}
We have presented a local construction for interpolatory composite surfaces, which is based on the use of local univariate spline interpolants having degree $g$, continuity $C^k$, polynomial reproduction degree $m$ and support width $w$.
Given as input a regular mesh, the composite surface has the same smoothness of the underlying class of univariate splines. Moreover, thanks to the augmented parametrization, each section curve can be parameterized independently and in the most appropriate way.
The section curves drive the shape of the surface, which turns out to be a good interpolant of the input mesh.

If a mesh contains extraordinary vertices, then the above local approach applies in the regular mesh regions. To patch the extraordinary regions we have proposed a modified form of Coons-Gregory patches that join with $G^1$ or $G^2$ continuity the regular portion of surface and, with respect to the latter, hold similar properties in terms of avoiding artifacts and poor shape quality.

We have demonstrated the effectiveness of the proposed technique by several application examples concerning both regular and extraordinary meshes.

The surface construction takes advantage from the local nature of the method, which allows for a combined use of regular an extraordinary augmented patches. Nevertheless, patching the extraordinary regions of a mesh is an independent challenging problem and we leave to future research the study of other possible approaches that could be paired with the regular patches.

Moreover, in principle, the augmented parametrization is not restricted to regular meshes, nor it can only be associated with the local interpolatory splines considered in this work.
Therefore, another interesting subject for future studies is to consider how other methods of surface generation may benefit from the augmented parametrization.


\section*{Acknowledgements}
This research was supported by the European Eurostars project NIIT4CAD \url{http://eurostars.unibo.it/} and by the Italian INdAM National Group for Scientific Calculus (GNCS).


\bibliographystyle{model3-num-names} 
\bibliography{patching} 

\appendix
\section{Some classes of fundamental functions}
\label{app:fund_eqs}

\noindent
The following fundamental spline functions are obtained following the approach reported in \cite{BCR13a}.

\subsection{Fundamental functions of the class \texorpdfstring{$D^3C^1P^2S^4$}{D3C1P2S4}}
Recalling our notation $d_i = x_{i+1}-x_i$, in the interval $[-d_{i-2}-d_{i-1},d_i+d_{i+1}]$ the expression of the fundamental function $\psi_i$ is given by
\begin{equation}
\psi_i(x) =
\begin{cases}
\displaystyle{\frac{\left(d_{i-1}+x\right) \left(d_{i-2}+d_{i-1}+x\right)^2}{d_{i-2} d_{i-1} \left(d_{i-2}+d_{i-1}\right)},}
&
\displaystyle{-d_{i-2}-d_{i-1}\leq x<-d_{i-1},} \\[2ex]
\displaystyle{\frac{\left(d_{i-1}+x\right) \left(d_i \left(-x^2+d_{i-1}d_{i-2}+d_{i-1}^2\right)-x \left(d_{i-2}+d_{i-1}\right) \left(d_{i-1}+x\right)\right)}{d_{i-1}^2 d_i \left(d_{i-2}+d_{i-1}\right)},}
&
\displaystyle{-d_{i-1}\leq x<0,} \\[2ex]
\displaystyle{\frac{\left(d_i-x\right) \left(-x^2 (d_{i-1}+d_i+d_{i+1})+x d_i \left(d_i+d_{i+1}\right)+d_{i-1} d_i \left(d_i+d_{i+1}\right)\right)}{d_{i-1} d_i^2 \left(d_i+d_{i+1}\right)},}
&
\displaystyle{0\leq x<d_i,} \\[2ex]
\displaystyle{\frac{\left(d_i-x\right) \left(d_i+d_{i+1}-x\right)^2}{d_i d_{i+1} \left(d_i+d_{i+1}\right)},}
& \displaystyle{d_i\leq x\leq d_i+d_{i+1}.}
\end{cases}
\end{equation}
The local parameter vector associated with $[x_s,x_{s+1}]$ is $\bd=\left(d_{s-1},d_s,d_{s+1}\right)$
and for any $x\in[0,d_s]$ the four nonzero fundamental functions $\psi_{i}$, $i=s-1,\dots,s+2$ have the expression
\begin{equation}\label{eq:CatmullRom}
\begin{aligned}
\psi_{s-1}(x;\bd) &= \displaystyle{-\frac{x \left(x-d_{s}\right)^2}{d_{s-1} d_{s} \left(d_{s-1}+d_{s}\right)}}, \\
\psi_{s}(x;\bd) &= \displaystyle{\frac{1}{d_{s}^2}\left(x-d_{s}\right) \left(\frac{x^2}{d_{s}+d_{s+1}}+\frac{x \left(x-d_{s}\right)}{d_{s-1}}-d_{s}\right)}, \\
\psi_{s+1}(x;\bd) &= \displaystyle{\frac{1}{d_{s}^2}x \left(\frac{d_{s} \left(d_{s-1}+2 x\right)-x^2}{d_{s-1}+d_{s}}-\frac{x \left(x-d_{s}\right)}{d_{s+1}}\right)}, \\
\psi_{s+2}(x;\bd) &= \displaystyle{\frac{x^2 \left(x-d_{s}\right)}{d_{s} d_{s+1} \left(d_{s}+d_{s+1}\right)}}.
\end{aligned}
\end{equation}

\subsection{Fundamental functions of the class \texorpdfstring{$D^5C^2P^2S^4$}{D5C2P2S4}}
Recalling our notation $d_i = x_{i+1}-x_i$, in the interval $[-d_{i-2}-d_{i-1},d_i+d_{i+1}]$ the expression of the fundamental function $\psi_i$ is given by
\begin{equation}
\psi_i(x) =
\begin{cases}
\text{\Large$
-\frac{\left(d_{i-1}+x\right) \left(d_{i-2}+d_{i-1}+x\right)^3 \left(-d_{i-2}+2 d_{i-1}+2 x\right)}{d_{i-2}^3 d_{i-1} \left(d_{i-2}+d_{i-1}\right)},$}
&
\hspace{-0.2cm} \displaystyle{-d_{i-2}-d_{i-1}\leq x<-d_{i-1},} \\[2ex]
\text{\Large$\frac{d_i \left(d_{i-1}+x\right) \left(3 x^3 d_{i-1}+d_{i-1}^3 \left(d_{i-2}+d_{i-1}\right)+2 x^4\right)-x \left(d_{i-2}+d_{i-1}\right) \left(d_{i-1}-2 x\right) \left(d_{i-1}+x\right)^3}{d_{i-1}^4 d_i \left(d_{i-2}+d_{i-1}\right)},$}
& \displaystyle{-d_{i-1}\leq x<0,} \\[2ex]
\text{\Large$\frac{\left(d_i-x\right) \left(\left(d_{i-1}+d_i+d_{i+1}\right)(2 x^4 -3 x^3 d_i)+x d_i^3 \left(d_i+d_{i+1}\right)+d_{i-1} d_i^3 \left(d_i+d_{i+1}\right)\right)}{d_{i-1} d_i^4 \left(d_i+d_{i+1}\right)},$}
& \displaystyle{0\leq x<d_i,} \\[2ex]
\text{\Large$-\frac{\left(x-d_i\right) \left(-2 d_i+d_{i+1}+2 x\right) \left(d_i+d_{i+1}-x\right)^3}{d_i d_{i+1}^3 \left(d_i+d_{i+1}\right)},$}
& \displaystyle{d_i\leq x\leq d_i+d_{i+1}.}\\[2ex]
\end{cases}
\end{equation}
The local parameter vector associated with $[x_s,x_{s+1}]$ is $\bd=\left(d_{s-1},d_s,d_{s+1}\right)$
and for any $x\in[0,d_s]$ the four nonzero fundamental functions $\psi_{i}$, $i=s-1,\dots,s+2$ have the expression
\begin{equation}\label{eq:D5C2P2S4}
\begin{aligned}
\psi_{s-1}(x;\bd) &= \displaystyle{\frac{x \left(x-d_{s}\right)^3 \left(d_{s}+2 x\right)}{d_{s-1} d_{s}^3 \left(d_{s-1}+d_{s}\right)}}, \\
\psi_{s}(x;\bd) &= \displaystyle{\frac{\left(d_{s}-x\right) \left(d_{s-1} \left(-3 x^3 d_{s}+d_{s}^4+d_{s}^3 d_{s+1}+2 x^4\right)+x \left(d_{s}+d_{s+1}\right) \left(d_{s}+2 x\right) \left(x-d_{s}\right)^2\right)}{d_{s-1} d_{s}^4 \left(d_{s}+d_{s+1}\right)},} \\
\psi_{s+1}(x;\bd) &= \displaystyle{\frac{1}{d_{s}^4}x \left(\frac{x^2 \left(2 x-3 d_{s}\right) \left(x-d_{s}\right)}{d_{s+1}}+\frac{-5 x^3 d_{s}+3 x^2 d_{s}^2+d_{s}^3 \left(d_{s-1}+x\right)+2 x^4}{d_{s-1}+d_{s}}\right)}, \\
\psi_{s+2}(x;\bd) &= \displaystyle{-\frac{x^3 \left(2 x-3 d_{s}\right) \left(x-d_{s}\right)}{d_{s}^3 d_{s+1} \left(d_{s}+d_{s+1}\right)}}.
\end{aligned}
\end{equation}

\end{document}